\documentclass[12pt]{article}

\usepackage{amsmath}
\usepackage{amsthm}
\usepackage{csvsimple,rotating}
\usepackage{comment}
\excludecomment{firstVcomment}

\usepackage[scr=boondoxo]{mathalpha}
\usepackage{xcolor,hyperref}
\definecolor{darkblue}{rgb}{0,0,0.75}
\definecolor{darkgreen}{rgb}{0,0.75,0}
\definecolor{darkred}{rgb}{0.75,0,0}
\hypersetup{colorlinks=true,linkcolor=darkblue,citecolor=darkgreen,urlcolor=darkred}

\usepackage{amssymb}
\usepackage[capitalise,sort]{cleveref}
\usepackage{bbm}

\usepackage[margin=0.7in, top=0.8in, bottom=0.8in]{geometry}
\usepackage{booktabs}
\usepackage{makecell}
\usepackage{needspace}

\usepackage{newfloat}

\usepackage{listings}

\DeclareFloatingEnvironment[fileext=frm,placement={!ht},name=Source Code]{listing}

\usepackage{tikz}
\usetikzlibrary{matrix,chains,positioning,decorations.pathreplacing,arrows,math}
\usetikzlibrary{shapes,arrows}
\usepackage{adjustbox}

\tikzset{
  font={\fontsize{7pt}{8}\selectfont}}
  
\usepackage[textwidth=0.7in]{todonotes}

\crefname{paragraph}{Section}{Sections}
\Crefname{paragraph}{Section}{Sections}

\newtheorem{theorem}{Theorem}[section]
\newtheorem{corollary}[theorem]{Corollary}
\newtheorem{lemma}[theorem]{Lemma}
\newtheorem{proposition}[theorem]{Proposition}

\newtheorem{definition}[theorem]{Definition}
\newtheorem{remark}[theorem]{Remark}

\newtheorem{algo}[theorem]{Method description}

\usepackage{verbatim}

\usepackage{cite}

\usepackage[shortlabels,inline]{enumitem}

\usepackage[utf8]{inputenc}
\usepackage[LGR,T1]{fontenc}

\DeclareMathAlphabet{\mathgtt}{LGR}{cmtt}{m}{n}

\usepackage{nicefrac}

\usepackage{xparse}
\usepackage{etoolbox}

\ExplSyntaxOn

\NewDocumentCommand{\enum}{ O{;} m o }
 {
  \my_enum:nnn { #1 } { #2 } { #3 }
 }

\seq_new:N \l__my_enum_seq
\tl_new:N \l__my_enum_item_tl
\prop_const_from_keyval:Nn \l__verbs
{
show = shows ,
imply = implies , 
demonstrate = demonstrates ,
prove = proves ,
establish = establishes ,
ensure = ensures ,
assure = assures
}
\int_new:N \l__number_of_args

\cs_new_protected:Nn \my_enum:nnn
 {
  \seq_set_split:Nnn \l__my_enum_seq { #1 } { #2 }
  \seq_remove_all:Nn \l__my_enum_seq {}
  \int_set_eq:NN \l__number_of_args { \seq_count:N \l__my_enum_seq }
  \seq_use:Nnnn \l__my_enum_seq { ~and~ } { ,~ } { ,~and~ }
  \IfNoValueTF{#3}{}{
    \space
    \int_compare:nNnTF{ \l__number_of_args } < {2}{ \prop_item:Nn \l__verbs {#3} }{ #3 }
  }
 }

\ExplSyntaxOff

 \ExplSyntaxOn

 \seq_new:N \g_cflist_loaded
 \seq_new:N \g_cflist_pending
 
 \NewDocumentCommand{\cfadd}{ m }
 {
   \seq_if_in:NnF \g_cflist_loaded { #1 } {
     \seq_if_in:NnF \g_cflist_pending { #1 } {
       \seq_gput_right:Nn \g_cflist_pending { #1 }
     }
   }
 }

 \NewDocumentCommand{\cfconsiderloaded}{ m }{
   \seq_gput_right:Nn \g_cflist_loaded {#1}
 }
 
 \NewDocumentCommand{\cfremove}{ m }
 {
   \seq_gremove_all:Nn \g_cflist_pending { #1 }
 }

 \NewDocumentCommand{\cfload}{ o }
 {
   \seq_if_empty:NTF \g_cflist_pending {\unskip} {
     (cf.\ \Cref{\seq_use:Nn \g_cflist_pending {,}})\IfValueTF{#1}{#1~}{\unskip}
     \seq_gconcat:NNN \g_cflist_loaded \g_cflist_loaded \g_cflist_pending
     \seq_gclear:N \g_cflist_pending
   }
 }
 
 \NewDocumentCommand{\cfclear} {} {
   \seq_gclear:N \g_cflist_loaded
   \seq_gclear:N \g_cflist_pending
 }
 
 \NewDocumentCommand{\cfout}{ o }
 {
   \seq_if_empty:NTF \g_cflist_pending {\unskip} {
     (cf.\ \Cref{\seq_use:Nn \g_cflist_pending {,}})\IfValueTF{#1}{#1~}{\unskip}
     \seq_gclear:N \g_cflist_pending
   }
 }
 
 \NewDocumentCommand{\ifnocf} { m } {
   \seq_if_empty:NT \g_cflist_pending { #1 }
 }
 
 \ExplSyntaxOff
 
\ExplSyntaxOn

\NewDocumentEnvironment {athm} {m m o} {%
\IfNoValueTF{#3}{%
\begin{#1}\label{#2}\global\def\loc{#2}%
}{%
\begin{#1}[#3]\label{#2}\global\def\loc{#2}%
}
}{%
\end{#1}%
}

\NewDocumentEnvironment {adef} {m} {%
\begin{definition}\label{#1}\global\def\loc{#1}%
}{%
\end{definition}%
}

\NewDocumentEnvironment{aproof} {} {%
\begin{proof}[Proof~of~\cref{\loc}]%
\bool_gset_false:N \g_finishproof_bool
}{%
\bool_if:NTF \g_finishproof_bool {}
{\finishproofthus}
\end{proof}%
}

\NewDocumentCommand{\finishproofthus} {} {
  \bool_gset_true:N \g_finishproof_bool 
  The~proof~of~\cref{\loc}~is~thus~complete.}
\NewDocumentCommand{\finishproofthis} {} {
  \bool_gset_true:N \g_finishproof_bool 
  This~completes~the~proof~of~\cref{\loc}.}

\ExplSyntaxOff

\ExplSyntaxOn

\NewDocumentEnvironment{flexmath}{ m o }{
  \str_if_eq:noTF {a} {#1} {
    \equation
    \IfValueT{#2}{\label{eq:\loc.#2}}
    \aligned
  } {
    \catcode`&=9
    \renewcommand{\\}{}
    \str_if_eq:noTF {d} {#1} {
      \equation
      \IfValueT{#2}{\label{eq:\loc.#2}}
    } {
      \math
    }
  }
}{
  \str_if_eq:noTF {i} {#1} {
    \endmath
    \catcode`&=4
  } {
    \str_if_eq:noTF {d} {#1} {
      \endequation
    } {
      \endaligned
      \endequation
    }
  }
}

\ExplSyntaxOff

\ExplSyntaxOn

\bool_new:N \g_noteobserve

\NewDocumentCommand{\setnote}{}{
  \bool_gset_true:N \g_noteobserve
}

\NewDocumentCommand{\setobserve}{}{
  \bool_gset_false:N \g_noteobserve
}

\NewDocumentCommand{\nobs}{ o }{
  \IfValueT{#1}{
    \str_if_eq:noTF {note} {#1} {
      \bool_gset_true:N \g_noteobserve
    } {
      \str_if_eq:noTF {Note} {#1} {
        \bool_gset_true:N \g_noteobserve
      } {
        \bool_gset_false:N \g_noteobserve
      }
    }
  }
  \bool_if:nTF { \g_noteobserve } {
    \bool_gset_false:N \g_noteobserve
    note
  } {
    \bool_gset_true:N \g_noteobserve
    observe
  }
  \IfValueF{#1}{~}
}

\NewDocumentCommand{\Nobs}{ o }{
  \IfValueT{#1}{
    \str_if_eq:noTF {note} {#1} {
      \bool_gset_true:N \g_noteobserve
    } {
      \str_if_eq:noTF {Note} {#1} {
        \bool_gset_true:N \g_noteobserve
      } {
        \bool_gset_false:N \g_noteobserve
      }
    }
  }
  \bool_if:nTF { \g_noteobserve } {
    \bool_gset_false:N \g_noteobserve
    Note\ that\unskip
  } {
    \bool_gset_true:N \g_noteobserve
    Observe\  that\unskip
  }
  \IfValueF{#1}{
    \unskip
  }
  \IfValueT{#1}{
    \IfEndWith{#1}{,}{}
  }
}

\bool_new:N \g_hencetherefore

\NewDocumentCommand{\hence}{ o }{
  \IfValueT{#1}{
    \str_if_eq:noTF {hence} {#1} {
      \bool_gset_true:N \g_hencetherefore
    } {
      \str_if_eq:noTF {Hence} {#1} {
        \bool_gset_true:N \g_hencetherefore
      } {
        \bool_gset_false:N \g_hencetherefore
      }
    }
  }
  \bool_if:nTF { \g_hencetherefore } {
    \bool_gset_false:N \g_hencetherefore
    hence
  } {
    \bool_gset_true:N \g_hencetherefore
    therefore
  }
  \IfValueF{#1}{~}
}

\NewDocumentCommand{\Hence}{ o }{
  \IfValueT{#1}{
    \str_if_eq:noTF {hence} {#1} {
      \bool_gset_true:N \g_hencetherefore
    } {
      \str_if_eq:noTF {Hence} {#1} {
        \bool_gset_true:N \g_hencetherefore
      } {
        \bool_gset_false:N \g_hencetherefore
      }
    }
  }
  \bool_if:nTF { \g_hencetherefore } {
    \bool_gset_false:N \g_hencetherefore
    Hence,~we~obtain
  } {
    \bool_gset_true:N \g_hencetherefore
    Therefore,~we~obtain
  }
  \IfValueF{#1}{~}
}

\int_new:N \g_furthermore

\NewDocumentCommand{\Moreover}{ o o }{
  \IfValueT{#1}{
    \str_case:nn {#1} {
	  {Next} {\int_gset:Nn {\g_furthermore} {0}}      
      {Furthermore} {\int_gset:Nn {\g_furthermore} {1}}
      {Moreover} {\int_gset:Nn {\g_furthermore} {2}}
      {In~addition} {\int_gset:Nn {\g_furthermore} {3}}
      {note} {\bool_gset_true:N \g_noteobserve}
      {observe} {\bool_gset_false:N \g_noteobserve}
    }
    \IfValueT{#2}{
      \str_case:nn {#2} {
	    {Next} {\int_gset:Nn {\g_furthermore} {0}}        
        {Furthermore} {\int_gset:Nn {\g_furthermore} {1}}
        {Moreover} {\int_gset:Nn {\g_furthermore} {2}}
        {In~addition} {\int_gset:Nn {\g_furthermore} {3}}
        {note} {\bool_gset_true:N \g_noteobserve}
        {observe} {\bool_gset_false:N \g_noteobserve}
      }
    }
  }
  \int_case:nn { \int_mod:nn {\g_furthermore} {4} } {
	{ 0 } { Next~\nobs ~that}    
    { 1 } { Furthermore,~\nobs ~that}
    { 2 } { Moreover,~\nobs ~that}
    { 3 } { In~addition,~\nobs ~that}
  }
  \int_incr:N \g_furthermore
 
}
\int_new:N \g_furthermorea
\NewDocumentCommand{\Moreovera}{ o o }{
  \IfValueT{#1}{
    \str_case:nn {#1} {
	  {Next} {\int_gset:Nn {\g_furthermorea} {0}}      
      {Furthermore} {\int_gset:Nn {\g_furthermorea} {1}}
      {Moreover} {\int_gset:Nn {\g_furthermorea} {2}}
      {In~addition} {\int_gset:Nn {\g_furthermorea} {3}}
      {note} {\bool_gset_true:N \g_noteobserve}
      {observe} {\bool_gset_false:N \g_noteobserve}
    }
    \IfValueT{#2}{
      \str_case:nn {#2} {
	    {Next} {\int_gset:Nn {\g_furthermorea} {0}}        
        {Furthermore} {\int_gset:Nn {\g_furthermorea} {1}}
        {Moreover} {\int_gset:Nn {\g_furthermorea} {2}}
        {In~addition} {\int_gset:Nn {\g_furthermorea} {3}}
        {note} {\bool_gset_true:N \g_noteobserve}
        {observe} {\bool_gset_false:N \g_noteobserve}
      }
    }
  }
   \int_case:nn { \int_mod:nn {\g_furthermorea} {4} } {
	{ 0 } { Next~\nobs}    
    { 1 } { Furthermore,~\nobs}
    { 2 } { Moreover,~\nobs}
    { 3 } { In~addition,~\nobs}
  }
  \int_incr:N \g_furthermore
 
}
\bool_new:N \g_exampleinstance
\NewDocumentCommand{\ex}{ o }{
  \IfValueT{#1}{
    \str_if_eq:noTF {ex} {#1} {
      \bool_gset_true:N \g_exampleinstance
    } {
      \str_if_eq:noTF {Ex} {#1} {
        \bool_gset_true:N \g_exampleinstance
      } {
        \bool_gset_false:N \g_exampleinstance
      }
    }
  }
  \bool_if:nTF { \g_exampleinstance} {
    \bool_gset_false:N \g_exampleinstance
    for~example,
  } {
    \bool_gset_true:N \g_exampleinstance
    for~instance,
  }
  \IfValueF{#1}{~}
}
\NewDocumentCommand{\Ex}{ o }{
  \IfValueT{#1}{
    \str_if_eq:noTF {ex} {#1} {
      \bool_gset_true:N \g_exampleinstance
    } {
      \str_if_eq:noTF {Ex} {#1} {
        \bool_gset_true:N \g_exampleinstance
      } {
        \bool_gset_false:N \g_exampleinstance
      }
    }
  }
  \bool_if:nTF { \g_exampleinstance } {
    \bool_gset_false:N \g_exampleinstance
    For~example,
  } {
    \bool_gset_true:N \g_exampleinstance
   For~instance, 
  }
  \IfValueF{#1}{~}
}
\ExplSyntaxOff

\newcommand{\I}{\rm{I}}
\newcommand{\algoG}{\mathcal{G}}
\newcommand{\algoM}{\mathcal{M}}

\newcommand{\N}{\mathbbm N}
\newcommand{\R}{\mathbbm R}

\newcommand{\PP}{\mathbb P}
\newcommand{\Borel}{\mathcal{B}}

\newcommand{\defaultNetDim}{\mathfrak d}
\newcommand{\netDim}{\defaultNetDim}
\newcommand{\deflink}[2]{#2}
\newcommand{\Aff}{\cfadd{def:affine}\deflink{def:affine}{\mathcal{A}}}
\newcommand{\probspacetilde}{(\tilde{\Omega},\tilde{\F}, \tilde{\P})}
\newcommand{\tildeOmega}{\tilde{\Omega}}

\ExplSyntaxOn

\NewDocumentCommand{\mEE}{ o m }{
  \IfValueTF{#1}{
    \str_case:on {#1} {
      {0}{\mathbb E\br{#2}}
      {1}{\mathbb E\br[\big]{#2}}
      {2}{\mathbb E\mkern-1.1mu\br[\Big]{#2}}
      {3}{\mathbb E\mkern-1.3mu\br[\bigg]{#2}}
      {4}{\mathbb E\mkern-1.5mu\br[\Bigg]{#2}}
    }
  } {
    \mathbb E\br{#2}
  }
}

\ExplSyntaxOff

\usepackage{mleftright}

\usepackage{mathtools}
\DeclarePairedDelimiter{\pr}{(}{)}
\DeclarePairedDelimiter{\br}{[}{]}

\DeclarePairedDelimiter{\abs}{\lvert}{\rvert}
\DeclarePairedDelimiter{\norm}{\lVert}{\rVert}

\makeatletter
\newcommand{\mylabel}[2]{#2\def\@currentlabel{#2}\label{#1}}
\makeatother
\newcommand{\basicset}{V}
\newcommand{\E}{\mathbb{E}}
\newcommand{\thmL}{\mathscr{L}}
\newcommand{\thmW}{W}
\newcommand{\lossexp}{\mathscr{l}}
\newcommand{\trandom}{\mathbb{T}}
\newcommand{\xrandom}{\mathbb{X}}
\newcommand{\yrandom}{\mathbb{Y}}
\newcommand{\solutionset}{S}
\newcommand{\thmp}{p}

\newcommand{\thmb}{b}
\newcommand{\act}{\mathfrak{a}}

\newcommand{\F}{\mathcal{F}}

\renewcommand{\P}{\mathbb{P}}

\newcommand{\BM}{B}
\newcommand{\s}{s}

\newcommand{\activation}{a}

\newcommand{\qandq}{\qquad\text{and}\qquad}
\newcommand{\andq}{\text{and}\qquad}

\newcommand{\multdim}{\cfadd{def:multidim_version}\deflink{def:multidim_version}{\mathfrak M}}

\DeclareMathOperator{\id}{id}

\usepackage{mleftright}

\renewenvironment{pmatrix}{\mleft(\begin{matrix}}{\end{matrix}\mright)}

\usepackage{subcaption}

\usepackage{marginfix}

\newcommand{\beq}{\begin{equation}}
    \newcommand{\eeq}{\end{equation}}
    \NewDocumentEnvironment{cproof}{m}
    {\begin{proof}[Proof of \Cref{#1}]}%
    {\noindent The proof of \Cref{#1} is thus complete.
    \end{proof}}
    \NewDocumentEnvironment{cproof2}{m}
    {\begin{proof}[Proof of \Cref{#1}]}%
    {\noindent This completes the proof of \Cref{#1}.
    \end{proof}}
    \ExplSyntaxOn

    \seq_const_from_clist:Nn \g_prove_mru {
      establish,
      demonstrate,
      prove,
      show,
      imply,
      ensure
    }
    
    \tl_new:N \g_wordtmp
    \seq_new:N \l_mytmps
    
    \cs_generate_variant:Nn \str_if_in:nnTF { nVTF }
    
    \NewDocumentCommand{\prove}{ o }{
      \IfValueTF{#1}{
        \seq_clear:N \l_mytmps
        \seq_map_inline:Nn \g_prove_mru {
          \str_if_eq:nnTF {##1} {ensure} {
            \str_set:Nn \l_temps {n}
          } {
            \str_set:Nx \l_temps {\str_head_ignore_spaces:n {##1}}
          }
          \str_if_in:nVTF {#1} \l_temps {
            \seq_put_right:Nn \l_mytmps {##1}
          } { }
        }
        \seq_get_right:NN \l_mytmps \g_wordtmp
      } {
        \seq_get_right:NN \g_prove_mru \g_wordtmp
      }
      \tl_use:N \g_wordtmp
      \seq_gput_left:NV \g_prove_mru \g_wordtmp
      \seq_gremove_duplicates:N \g_prove_mru
    }
    
    \NewDocumentCommand{\proves}{ o }{
      \IfValueTF{#1}{
        \seq_clear:N \l_mytmps
        \seq_map_inline:Nn \g_prove_mru {
          \str_if_eq:nnTF {##1} {ensure} {
            \str_set:Nn \l_temps {n}
          } {
            \str_set:Nx \l_temps {\str_head_ignore_spaces:n {##1}}
          }
          \str_if_in:nVTF {#1} \l_temps {
            \seq_put_right:Nn \l_mytmps {##1}
          } { }
        }
        \seq_get_right:NN \l_mytmps \g_wordtmp
      } {
        \seq_get_right:NN \g_prove_mru \g_wordtmp
      }
      \str_set:NV \l_tmpa_str \g_wordtmp
      \prop_get:NVN \l__verbs \l_tmpa_str \l_tmpa_tl
      \tl_use:N \l_tmpa_tl
      \seq_gput_left:NV \g_prove_mru \g_wordtmp
      \seq_gremove_duplicates:N \g_prove_mru
    }
    
    \ExplSyntaxOff
    \newcommand{\with}{\curvearrowleft}

    \crefname{enumi}{item}{items}
    \crefname{subsec}{Section}{Sections}
    \crefname{listing}{Source Code}{Source Codes}
    \crefname{equation}{}{}
    \crefname{lemma}{Lemma}{Lemmas}
    \crefname{figure}{Figure}{Figures}

\DeclarePairedDelimiterXPP{\Exp}[1]{\E}{[}{]}{}{#1}
\newcommand{\providecommandordefault}[2]{%
    \providecommand{#1}{}%
    \renewcommand{#1}{#2}%
}

\newcommand{\convolution}{\cfadd{def:convolution}{\circledast}}
\newcommand{\onetensor}{\cfadd{def:onetensor}\mathbf{I}}

\newcommand{\indicator}[1]{\mathbbm{1}_{#1}}

\newcommand{\tensordim}{d}
\newcommand{\dimindex}{k}
\newcommand{\matrixindex}{i}
\newcommand{\matrixindexalt}{j}
\newcommand{\matrixdim}{\mathfrak{d}}
\newcommand{\layerindex}{\ell}

\newcommand{\neuralOp}{\mathscr{N}}
\newcommand{\fd}{\mathfrak{d}}

\newcommand{\scp}[2]{\langle#1,#2\rangle}

\newcommand{\dft}{\mathtt{DFT}\cfadd{def:DFT}}
\newcommand{\idft}{\mathtt{IDFT}\cfadd{def:IDFT}}
\newcommand{\nrspacediscr}{N}

\newcommand{\cF}{\mathcal{F}}

\newcommand{\cL}{\mathcal{L}}

\newcommand{\Z}{\mathbb{Z}}
\newcommand{\C}{\mathbb{C}}

\newcommand{\fx}{\mathfrak{x}}

\DeclarePairedDelimiterX{\lossmetric}[1]{\lvert\!\lvert\!\lvert}{\rvert\!\rvert\!\rvert}{#1}

\newcommand{\laplaceFactor}{c}

\newcommand{\IVfunction}{g}
\newcommand{\sol}{u}

\newcommand{\solOp}{\mathcal{S}}
\newcommand{\solOpAlt}{\tilde{\mathcal{S}}}
\newcommand{\initialValues}{\mathcal{I}}
\newcommand{\EndValues}{\mathcal{O}}

\newcommand{\initialRV}{\mathfrak{I}}

\newcommand{\Ltwo}[1]{L^2(#1 ; \R)}

\newcommand{\EndVariable}{h}

\makeatletter
\ExplSyntaxOn
\seq_new:N \g_abbrs
\prop_new:N \g_abbr_counts
\tl_new:N \l_abbr_count_tl

\NewDocumentCommand{\abbr}{m m O{#1} m m O{#4}}{
    \expandafter\newcommand\csname#3\endcsname[1][]{
        \seq_if_in:NnTF \g_abbrs {#1} {
            \prop_get:NnN \g_abbr_counts {#1} \l_abbr_count_tl
            \prop_gput:Nnx \g_abbr_counts {#1} {\int_eval:n {\l_abbr_count_tl + 1}}
            \hyperref[#1]{#1}
        } {
            \seq_gput_left:Nn \g_abbrs {#1}
            \prop_gput:Nnn \g_abbr_counts {#1} {1}
            \expandafter\gdef\csname#1@def\endcsname{#2}
            \phantomsection\label{#1}
            \str_if_eq:nnTF{##1}{}{\emph{#2}}{##1}~(\hyperref[#1]{#1})
        }
    }
    \expandafter\newcommand\csname#6\endcsname[1][]{
        \seq_if_in:NnTF \g_abbrs {#1} {
            \prop_get:NnN \g_abbr_counts {#1} \l_abbr_count_tl
            \prop_gput:Nnx \g_abbr_counts {#1} {\int_eval:n {\l_abbr_count_tl + 1}}
            \hyperref[#1]{#4}
        } {
            \expandafter\gdef\csname#1@def\endcsname{#5}
            \seq_gput_left:Nn \g_abbrs {#1}
            \prop_gput:Nnn \g_abbr_counts {#1} {1}
            \phantomsection\label{#1}
            \str_if_eq:nnTF{##1}{}{\emph{#5}}{##1}~(\hyperref[#1]{#4})
        }
    }
}

\ExplSyntaxOff
\makeatother

\abbr{ANN}{artificial neural network}{ANNs}{artificial neural networks}
\abbr{CNN}{convolutional neural network}{CNNs}{convolutional neural networks}
\abbr{SGD}{stochastic gradient descent}{SGDs}{Stochastic Gradient Descents}
\abbr{PDE}{partial differential equation}{PDEs}{partial differential equations}
\abbr{ODE}{ordinary differential equation}{ODEs}{ordinary differential equations}
\abbr{FNO}{Fourier neural operator}{FNOs}{Fourier neural operators}
\abbr{ADANN}{algorithmically designed artificial neural network}{ADANNs}{algorithmically designed artificial neural networks}
\abbr{LIRK}{linearly implicit Runge-Kutta}{LIRKs}{Linearly Implicit Runge--Kutta}
\abbr{GELU}{Gaussian Error Linear Unit}{GELUs}{Gaussian Error Linear Units}
\abbr{DeepONet}{deep operator network}{DeepONets}{deep operator networks}
\abbr{MLP}{full history recursive multilevel Picard}{MLPs}{full history recursive multilevel Picards}
\abbr{MC}{Monte Carlo}{MCs}{Monte Carlos}
\abbr{MLMC}{multilevel Monte Carlo}{MLMCs}{multilevel Monte Carlos}
\abbr{QMC}{quasi--Monte Carlo}{QMCs}{quasi--Monte Carlos}
\abbr{LRV}{learning the random variables}{LRVs}{learning the random variables}
\abbr{SDE}{stochastic differential equation}{SDEs}{stochastic differential equations}
\abbr{BSDE}{backward stochastic differential equation}{BSDEs}{backward stochastic differential equations}
\abbr{cod}{curse of dimensionality}{cods}{curses of dimensionality}
\abbr{SFPE}{stochastic fixed point equation}{SFPEs}{stochastic fixed point equations}
\abbr{FEM}{finite element method}{FEMs}{finite element methods}
\abbr{FDM}{finite difference method}{FDMs}{finite difference methods}
\abbr{IKNO}{integral kernel neural operator}{IKNOs}{integral kernel neural operators}
\abbr{PINO}{physics--informed neural operator}{PINOs}{physics--informed neural operators}
\abbr{PINN}{physics--informed neural network}{PINNs}{physics--informed neural networks}
\abbr{GKNO}{graph kernel neural operator}{GKNOs}{graph kernel neural operators}
\abbr{MGNO}{multipole graph neural operator}{MGNOs}{multipole graph neural operators}
\abbr{FBSDE}{forward--backward stochastic differential equation}{FBSDEs}{forward--backward stochastic differential equations}
\abbr{PIDE}{partial integro--differential equation}{PIDEs}{partial integro--differential equations}
\abbr{FFT}{fast Fourier transform}{FFTs}{fast Fourier transforms}
\abbr{IFFT}{inverse fast Fourier transform}{IFFTs}{inverse fast Fourier transforms}

\title{An Overview on Machine Learning Methods for \\ Partial Differential Equations: from Physics  Informed \\ Neural Networks to Deep Operator Learning}

\author{\noindent
Lukas Gonon$^{1}$, 
Arnulf Jentzen$^{2,3}$,
Benno Kuckuck$^{4}$,\\
Siyu Liang$^{5,6,7}$,
Adrian Riekert$^8$,
Philippe von Wurstemberger$^{9,10}$
\bigskip
\\
\small{$^1$ Department of Mathematics, Imperial College London,}
\vspace{-0.1cm}\\
\small{United Kingdom; e-mail: \texttt{l.gonon@imperial.ac.uk}}
\smallskip
\\
\small{$^2$ School of Data Science and Shenzhen Research Institute}
\vspace{-0.1cm}\\
\small{of Big Data, The Chinese University of Hong Kong, Shenzhen}
\vspace{-0.1cm}\\
\small{(CUHK-Shenzhen), China; e-mail: \texttt{ajentzen@cuhk.edu.cn}}
\smallskip
\\
\small{$^3$ Applied Mathematics: Institute for Analysis and Numerics,}
\vspace{-0.1cm}\\
\small{Faculty of Mathematics and Computer Science, University of}
\vspace{-0.1cm}\\
\small{M\"unster, Germany; e-mail: \texttt{ajentzen@uni-muenster.de}}
\smallskip
\\
\small{$^4$ Applied Mathematics: Institute for Analysis and Numerics,}
\vspace{-0.1cm}\\
\small{Faculty of Mathematics and Computer Science, University of}
\vspace{-0.1cm}\\
\small{M{\"u}nster,  Germany; e-mail: \texttt{bkuckuck@uni-muenster.de}}
\smallskip
\\
\small{$^5$ School of Mathematics and Statistics,}
\vspace{-0.1cm}\\
\small{Nanjing University of Science and Technology,}
\vspace{-0.1cm}\\
\small{Nanjing, China; e-mail: \texttt{liangsiyu@njust.edu.cn}}
\smallskip
\\
\small{$^6$ School of Data Science, The Chinese University}
\vspace{-0.1cm}\\
\small{of Hong Kong, Shenzhen (CUHK-Shenzhen), China}
\smallskip
\\
\small{$^{7}$ Mathematisches Institut,  Ludwig-Maximilians-Universit{\"a}t M{\"u}nchen,  Germany}
\smallskip
\\
\small{$^8$ Applied Mathematics: Institute for Analysis and Numerics,}
\vspace{-0.1cm}\\
\small{Faculty of Mathematics and Computer Science, University of}
\vspace{-0.1cm}\\
\small{M\"unster, Germany; e-mail: \texttt{ariekert@uni-muenster.de}}
\smallskip
\\
\small{$^{9}$ School of Data Science, The Chinese University}
\vspace{-0.1cm}\\
\small{of Hong Kong, Shenzhen (CUHK-Shenzhen),}
\vspace{-0.1cm}\\
\small{China; e-mail: \texttt{philippevw@cuhk.edu.cn}}
\smallskip
\\
\small{$^{10}$ Risklab, Department of Mathematics,} 
\vspace{-0.1cm}\\
\small{ETH Zurich, Switzerland;}
\vspace{-0.1cm}\\
\small{e-mail: \texttt{philippe.vonwurstemberger@math.ethz.ch}}
}

\begin{document}

\maketitle

\begin{abstract} 

The approximation of solutions of partial differential equations (PDEs) with numerical algorithms is a central topic in applied mathematics.
For many decades, various types of methods for this purpose have been developed and extensively studied.
One class of methods which has received a lot of attention in recent years are machine learning-based methods, which typically involve the training of artificial neural networks (ANNs) by means of stochastic gradient descent type optimization methods.
While approximation methods for PDEs using ANNs have first been proposed in the 1990s 
they have only gained wide popularity in the last decade with the rise of deep learning.
This article aims to provide an introduction to some of these methods and the mathematical theory on which they are based.
We discuss methods such as physics-informed neural networks (PINNs) and deep BSDE methods and consider several operator learning approaches.
\end{abstract}

\tableofcontents

\section{Introduction}

The approximation of solutions of \PDEs\ with numerical algorithms is a central topic in applied mathematics.
For many decades, various types of methods for this purpose have been developed and extensively studied such as \FDMs, \FEMs, and spectral methods.
One class of methods which has received a lot of attention in recent years are machine learning-based methods, which typically involve the training of \ANNs\ by means of \SGD-type optimization methods.
While approximation methods for \PDEs\ using \ANNs\ have first been proposed in the 1990s 
(cf., e.g., \cite{dissanayake1994neural,lagaris1998artificial,jianyu2003numerical}) they have only gained wide popularity in the last decade with the rise of deep learning 
(cf., e.g., \cite{Raissi19,MR3874585,MR3847747,MR3736669} for some seminal modern references).
This article aims to provide an introduction to some of these methods and the mathematical theory on which they are based.
There are already several other excellent survey articles on this topic and we refer, e.g., to 
\cite{beck2020overview,germain2021neural,ruf2019neural,MR4270459,MR4356985,MR4412280,Karniadakis21,MR4209661,MR4457972,yunus2022overview,huang2022partial,pratama2021anns,mishra2017deep,brunton2023machine,hu2023recent,MR4567794,MR4571806,pateras2023taxonomic,brenner2021machine,ghosh2023taxonomic,yadav2023data,karlsson2023solving,schmidhuber2015deep,capponi_lehalle_2023,sharma2023review,chen2023mathematical,gupta2023survey,tanyu2022deep,gopalakrishna2023essays,kumar2023deep,kianiharchegani2023data,rizvi2023data,herrmann2023deep,barbara2023train,herrmann2024deep,MR4680273,zhou2023deep,Jentzen2023}.

In this article we consider two different types problems that machine learning methods for \PDEs\ seek to address.
The first type of problems focuses on solving single instances of \PDEs\ while the second type of problems is concerned with the approximation of solution operators corresponding to \PDEs.
Machine learning methods for the first type of problems generally involve reformulating \PDE\ approximation problems as stochastic optimization problems 
which can be solved using \SGD-type optimization methods.
In this article we discuss two main classes of such methods and distinguish between them based on how this reformulation is achieved.
Methods in the first class make use of residuals of \PDEs, that is,
the difference between the left-hand sides and the right-hand sides of \PDEs\ for solution candidates, and are discussed in \cref{sect:residualformulations}.
On the other hand, methods in the second class make use of Feynman-Kac formulas which provide stochastic representations of solutions of parabolic \PDEs, and are discussed in \cref{sec:FCformulations}.
For both classes, we identify abstract reformulation results that allow the use of residuals and Feynman-Kac formulas, respectively, to recast \PDE\ problems as suitable stochastic optimization problems (cf.\ \cref{sect:reformulationI,sect:reformulationII}).
We illustrate how these reformulation results can be applied to various types of \PDEs.
Specifically, we present reformulations which underpin methods such as
\PINNs\ (cf., \cref{secpinn}),
deep Kolmogorov methods (cf., \cref{sect:FeynmanKac}),
and
deep \BSDE\ methods (cf., \cref{sect:FeynmanKac}).

In \cref{section:NeurOP} we discuss operator learning methods for \PDEs\ where the aim is to learn solution operators of \PDEs\ which corresponds to the second type of problems mentioned above.
For instance, such operators could be mappings which assign to all suitable initial values or boundary conditions the corresponding \PDE\ solutions.
We rigorously define several architectures which make it possible to learn such mappings from function spaces to function spaces and discuss how they can be trained.

In the remainder of this introduction we now introduce in the next subsection the definition of fully-connected feedforward \ANNs\ which we will use throughout this article.

\subsection{Fully-connected feedforward artificial neural networks (ANNs)}

The definitions in this subsection are strongly inspired by \cite[Chapter 1]{Jentzen2023}.

\begin{definition}[Affine linear functions]\label{def:affine}
 {%
 Let $\mathfrak d,m,n \in \N$, $s \in \N_0$, 
$ \theta = ( \theta_1, \dots,\allowbreak \theta_\netDim ) \in \R^\netDim $ 
satisfy $\netDim \geq \s + m n + m$.
Then we denote by $\Aff_{m,n}^{\theta, \s}\colon\R^n\to\R^m$ the function which satisfies 
for all $x = (x_1,\ldots, x_{n}) \in \R^{n}$ that
\begin{equation}
  \label{eq:affine}
\begin{split}
   &\Aff_{m,n}^{\theta,\s}( x ) 
= 
  \begin{pmatrix}
      \theta_{ \s + 1 }
    &
      \theta_{ \s + 2 }
    &
      \cdots
    &
      \theta_{ \s + n }
    \\
      \theta_{ \s + n + 1 }
    &
      \theta_{ \s + n + 2 }
    &
      \cdots
    &
      \theta_{ \s + 2 n }
    \\
      \theta_{ \s + 2 n + 1 }
    &
      \theta_{ \s + 2 n + 2 }
    &
      \cdots
    &
      \theta_{ \s + 3 n }
    \\
      \vdots
    &
      \vdots
    &
      \ddots
    &
      \vdots
    \\
      \theta_{ \s + ( m - 1 ) n + 1 }
    &
      \theta_{ \s + ( m - 1 ) n + 2 }
    &
      \cdots
    &
      \theta_{ \s + m n }
    \end{pmatrix}
    \begin{pmatrix}
      x_1
    \\
      x_2
    \\
      x_3
    \\
      \vdots 
    \\
      x_{n}
    \end{pmatrix}
  +
    \begin{pmatrix}
      \theta_{ \s + m n + 1 }
    \\
      \theta_{ \s + m n + 2 }
    \\
      \theta_{ \s + m n + 3 }
    \\
      \vdots 
    \\
      \theta_{ \s + m n + m }
    \end{pmatrix}
.
\end{split}
\end{equation}}
\end{definition}

\begingroup
\providecommand{\d}{}
\renewcommand{\d}{\fd}
\begin{definition}[Multi-dimensional versions of one-dimensional functions]
\label{def:multidim_version}
Let $\psi \colon \R \to \R$ be a function.
Then we denote by 
$
	\multdim_{\psi} 
\colon  
	\pr{
		\cup_{\tensordim \in \N} \cup_{\d_1, \d_2, \allowbreak \ldots, \d_\tensordim \in \N} 
			\R^{\d_1 \times \d_2 \times \ldots \times \d_\tensordim} 
	}
\to 
	\pr{
		\cup_{\tensordim \in \N} \cup_{\d_1, \d_2, \allowbreak \ldots, \d_\tensordim \in \N} 
			\R^{\d_1 \times \d_2 \times \ldots \times \d_\tensordim} 
	}
$ 
the function which satisfies for all 
$\tensordim \in \N$,
$\d_1, \d_2, \allowbreak \ldots, \d_\tensordim \in \N$,
$ 
	x 
= 
	( x_{\matrixindex_1, \matrixindex_2, \ldots, \matrixindex_\tensordim} )_{
		(\matrixindex_1, \matrixindex_2, \ldots, \matrixindex_\tensordim) \in (\bigtimes_{\dimindex = 1}^\tensordim  \{1, 2, \ldots, \d_\dimindex\})
	} \in \R^{\d_1 \times \d_2 \times \ldots \times \d_\tensordim}
$,
$ 
	y
= 
	( y_{\matrixindex_1, \matrixindex_2, \ldots, \matrixindex_\tensordim} )_{
		(\matrixindex_1, \matrixindex_2, \ldots, \matrixindex_\tensordim) \in (\bigtimes_{\dimindex = 1}^\tensordim  \{1, 2, \ldots, \d_\dimindex\})
	} \in \R^{\d_1 \times \d_2 \times \ldots \times \d_\tensordim}
$
with 
$\forall \, \matrixindex_1 \in \{1, 2, \ldots, \d_1\}$,  $\matrixindex_2 \in \{1, 2, \ldots, \d_2\}$, $\dots$, $\matrixindex_\tensordim \in \{1, 2, \ldots, \d_\tensordim\} \colon y_{\matrixindex_1, \matrixindex_2, \ldots, \matrixindex_\tensordim} = \psi(x_{\matrixindex_1, \matrixindex_2, \ldots, \matrixindex_\tensordim})$
that
\begin{equation}
\label{multidim_version:Equation}
\multdim_{\psi}( x ) 
=
y.
\end{equation}
\end{definition}
\endgroup

\begin{definition}[Realizations of fully-connected feedforward \ANNs]\label{def:nets}
{ \sl Let   $L \in \N \backslash \{ 1 \}$,  $l_0,l_1,\dots,l_L
\in \N$,       for every
$j\in\{0,1,\dots,L\}$
let
$\mathfrak d_j\in\N_0$
satisfy $\mathfrak d_j=\sum_{ k = 1 }^{ j } l_k (l_{k-1} + 1)$, 
let $\act\colon \R\rightarrow \R$ be a function, 
 and 
let
$\theta\in \R^{\mathfrak d_{L}}$.   Then we denote by
$
  \mathcal N^{l_0,l_1,\dots,l_L}_{\act,\theta}
 \colon \R^{l_0}\rightarrow \R^{l_L}
$
the function which satisfies 
\begin{equation}
  \label{eq:paramFFNN}
  {\mathcal N}^{l_0,l_1,\dots,l_L}_{\act,\theta}
  = 
   \Aff_{l_{L},l_{L-1}}^{\theta,\mathfrak d_{L-1}}\circ  \multdim_{ \act }
  \circ\Aff_{l_{L-1},l_{L-2}}^{\theta,\mathfrak d_{L-2}}
  \circ\ldots
  \circ  \multdim_{ \act }\circ \Aff_{l_1,l_0}^{\theta,\mathfrak d_0}
\end{equation}
\cfload.}
\end{definition}

\cfclear
\cfconsiderloaded{def:nets}

\section{Machine learning approximation methods for PDEs based on residual formulations}
\label{sect:residualformulations}

In this section we discuss machine learning approximation methods for \PDEs\ which are based on minimizing residuals of \PDEs.
We first introduce in \cref{sect:reformulationI}
an abstract reformulation result upon which many machine learning methods for \PDEs\ rely to reformulate \PDEs\ as stochastic minimization problems.
Thereafter, in \cref{secpinn} we discuss \PINNs\ which are one of the most famous machine learning approximation methods for \PDEs\ based on residual formulations.

\subsection{Basic reformulation result for machine learning methods for PDEs I}
\label{sect:reformulationI}

Machine learning generally involves the training of \ANNs\ through \SGD-type methods.
Applying \SGD-type methods to train \ANNs\ requires the formulation of a suitable training objective given by a parametric expectation.
Therefore, to develop machine learning techniques for \PDEs, it is necessary to reformulate \PDE\ problems into stochastic minimization problems, where the objective function is given by a parametric expectation.

In this section we introduce in \cref{lemmaBcor}
an abstract reformulation result upon which many machine learning methods for \PDEs\ rely to reformulate \PDE\ problems into such stochastic minimization problems.

 \begin{proposition}\label{lemmaB}
 {%

 Let $\solutionset$ and $\basicset$ be non-empty sets with $\solutionset  \subseteq \basicset$,  let $(\Omega, \F,\P)$  be a probability space,   for every $v \in \basicset$  let $\thmL_v\colon \Omega \to  [0,\infty]$ be a random variable,  
 assume for all $v \in \basicset$ that  $\P( \thmL_v = 0 ) = 1 $ if and only if $v \in \solutionset $,   and let $\lossexp \colon \basicset \to [0,\infty]$ satisfy for all $v \in \basicset$ that $\lossexp(v) =
 \E[ \thmL_v]$.
 Then
 \begin{enumerate}[(i)]
 \item  \label{prop21item1}
 for all $v \in \basicset$  it holds that
  \beq\label{IveqinfIw}\lossexp(v) = \inf_{w \in \basicset}\lossexp(w)  \eeq
if and only if
 $ v \in \solutionset   $
and
 \item\label{prop21item2}
  there exists $v \in \basicset$ such that
\beq\label{IveqinfIw2}\lossexp(v) = \inf_{w \in \basicset}\lossexp(w).\eeq
  \end{enumerate}
 }
 \end{proposition}
 \begin{cproof}{lemmaB}
 \Nobs\ the assumption that      
   for all $v \in \basicset$  it holds that  $\P( \thmL_v = 0 ) = 1$  if and only if $v \in \solutionset $
  \proves\   that for all $v\in \solutionset $ it holds that 
   \beq  \P( \thmL_v = 0 ) = 1.\eeq
   The assumption that for all $v\in \basicset$  it holds that $\lossexp(v) =   \E[ \thmL_v]$ \hence \proves\  that
   for all $v\in \solutionset $ it holds that 
 \beq\label{Iveq0new}  \lossexp(v)=  \E[ \thmL_v]=0. \eeq
 Combining this and the fact that for all $w \in \basicset$ it holds that $\lossexp(w) \geq 0$ \proves\ that for all $v\in  \solutionset $  it holds that 
 \beq
 0\leq \inf_{w\in \basicset}\lossexp(w)\leq\lossexp(v)=0.
 \eeq
  \Hence    for all $v\in \solutionset $ that
\begin{equation}\label{eq:9}
 \lossexp(v) = \inf_{w \in \basicset}\lossexp(w)=0.
\end{equation}
Combining this and the assumption that $\solutionset  \neq \emptyset$  \proves\  \cref{prop21item2}.
 \Nobs\    \cref{eq:9} and the assumption that $\solutionset \neq\emptyset$    
  \prove\  that 
  \beq
  \inf_{w\in \basicset}\lossexp(w)=0.
  \eeq
  \Hence that   for all $v \in \basicset$   with
  $\lossexp(v) = \inf_{w \in \basicset}\lossexp(w)$  it holds that  
  $\lossexp(v)=  0$.
 The assumption that for all $v\in \basicset$  it holds that $\lossexp(v) =   \E[ \thmL_v]$ \hence \proves\  that
 for all $v \in \basicset$   with
  $\lossexp(v) = \inf_{w \in \basicset}\lossexp(w)$ 
  it holds  that
  \beq
  \E[ \thmL_v]=0.
  \eeq
  \Hence that for all $v \in \basicset$   with
  $\lossexp(v) = \inf_{w \in \basicset}\lossexp(w)$ it holds that
  \beq  \P( \thmL_v = 0 ) = 1.\eeq
  The assumption that for all $v \in \basicset$  it holds that $\P( \thmL_v = 0 ) = 1 $  if and only if $v \in \solutionset$ \hence \proves\ that for all  $v \in \basicset$  with $\lossexp(v) = \inf_{w \in \basicset}\lossexp(w)$  it holds that $v\in \solutionset $.  Combining this with \cref{eq:9}  \proves\  \cref{prop21item1}.
 \end{cproof}
 
  \begin{corollary}[Basic reformulation result for machine learning methods for \PDEs\ I]\label{lemmaBcor}
 {%

 Let $\solutionset$ and $\basicset$ be non-empty sets with $\solutionset  \subseteq \basicset$,  let $(\Omega, \F,\P)$  be a probability space,   
   let $( \mathfrak{R}, \mathcal{R})$ be a measurable space,   let $R \colon \Omega \to  \mathfrak{R}$ be a random variable,  let $N\in \N$, 
for every $v \in \basicset$   let $U_v \colon   \mathfrak{R}  \to \R^N$ be  measurable,
 assume for all $v \in \basicset$ that  $\P( U_v(R ) = 0 ) = 1 $ if and only if $v \in \solutionset $,   and let $\lossexp \colon \basicset \to [0,\infty]$ satisfy\footnote{\Nobs\   for all $d \in \N$,  $v = (v_1, \dots, v_d)$, $w = (w_1,\dots,w_d) \in \R^d$ it holds that $\| v \|^2 = \sum_{ i = 1 }^d | v_i |^2$ and  $\langle v, w \rangle = \sum_{ i = 1 }^d v_i w_i$.}
  for all $v \in \basicset$ that $\lossexp(v) =
 \E[  \| U_v(R ) \|^2]$.
 Then
 \begin{enumerate}[(i)]
 \item\label{lemmaBcoritemi} 
 for all $v \in \basicset$  it holds that
  \beq \lossexp(v) = \inf_{w \in \basicset}\lossexp(w)  \eeq
if and only if
 $ v \in \solutionset   $
and
 \item\label{lemmaBcoritemii} 
  there exists $v \in \basicset$ such that
\beq \lossexp(v) = \inf_{w \in \basicset}\lossexp(w).\eeq
  \end{enumerate}
 }
 \end{corollary}
 \begin{cproof}{lemmaBcor}
\Nobs\   \cref{lemmaB} (applied with $\basicset\with\basicset$,  $\solutionset\with\solutionset$,   $(\thmL_v)_{v\in\basicset} \with ( \|U_v(R)\|^2)_{v\in \basicset}$ in the notation of \cref{lemmaB})
\proves\  \cref{lemmaBcoritemi}  and \cref{lemmaBcoritemii}.
  \end{cproof}

\subsection{Physics-informed neural networks (PINNs)}\label{secpinn}

In this section we discuss \PINNs.
The term \PINNs\ and the corresponding methodology was introduced in \cite{Raissi2017PhysicsID,Raissi2017PhysicsID2,Raissi19} for two different types of problems.
The first type of problem is called data-driven solutions of \PDEs, where to goal is to approximate solutions of \PDEs\ using \ANNs.
The second type of problem is called data-driven discovery of \PDEs, where the goal is to estimate parameters of \PDEs\ based on observation data.
In this section we consider \PINNs\ for the first type of problem, i.e., the approximation of solutions of \PDEs.

\subsubsection{PINNs for general boundary value PDE problems}
\label{sect:PINNsBVP}

In this section we introduce a \PINNs\ methodology for general boundary value \PDE\ problem in \cref{algo26copy} below.
To motivate \cref{algo26copy} we introduce in \cref{thm213acopy} a result which allows reformulating solutions of general boundary value \PDE\ problems as solutions of stochastic minimization problems.
The proof of \cref{thm213acopy} relies on the abstract reformulation result for \PDEs\ in \cref{lemmaBcor} above and on \cref{sec2lem6} below.

\begin{lemma}\label{sec2lem6}
{%
  Let $d,\delta \in \N$,   let $D\subseteq  \R^d$,   let $f \in C( D, \R^{\delta} )$,    let $( \Omega, \F, \P )$ be a probability space,     let $X\colon \Omega \to D $  be a random variable, 
 assume for all $x \in D$,  $\varepsilon \in (0,\infty)$ that $\P( \|X-x\|< \varepsilon  ) > 0$,    and assume $\P( f( X ) = 0 ) = 1$.  Then it holds for all $x \in D$ that 
 \beq\label{fxequals0} f(x) = 0. \eeq}
\end{lemma}
\begin{cproof}{sec2lem6}
We establish \cref{fxequals0} by a contradiction. 
For this let  $x \in D$ satisfy $f(x) \neq 0$.   \Nobs\ the fact that $\| f(x) \| > 0$  and the fact that $ \R \ni y \mapsto \|f(y)\| \in \R$ is continuous \prove\ that there exists $\eta \in (0,\infty)$ which satisfies  for all $y \in  \{ z\in D\colon \|z-x\|<\eta\}$  that
 \beq\label{absfpositive}\| f( y ) \| > 0.\eeq
\Moreover\   \cref{absfpositive}  and the assumption that $\P( f(X) = 0 ) = 1$  \prove\    that
\beq \label{eq:14}
0 = 1 -\P( f(X) = 0 ) = \P(f (X) \neq  0) = \P(\| f (X)\| > 0) \geq  \P(\| X - x\|< \eta).
\eeq
 \Moreover\  the assumption that for all $y \in  D$, $\varepsilon \in (0,\infty)$    it holds that
 $\P( \|X-y\|< \varepsilon  ) > 0$
  \proves\  that 
$\P( \|X-x\|< \eta ) > 0$.
Combining this with  \cref{eq:14}  \proves\  that 
\beq 0= \P( \|f(X)\| > 0 ) \geq \P( \| X - x \| < \eta ) > 0.\eeq
\end{cproof}

  \begin{theorem}[\PINNs\ for general boundary value \PDE\ problems]\label{thm213acopy}
 {%
 Let  $d,\delta,\thmb \in \N$,    let $D\subseteq \R^d$,
  let $\basicset\subseteq C(  D,\R^{\delta})$,  
  let $\thmW \subseteq \partial D$,
  let $\mathscr{D}\colon \basicset\to C( D,\R^{\delta}) $ and $\mathscr{B}\colon \basicset\to C(\thmW,\R^{\thmb})$ be functions,  let 
 $u\in \basicset$ satisfy  for all  $x\in D$,  $y\in \thmW$  that 
  \begin{equation}\label{weakpdeuacopy}
(\mathscr{D}(u))(x)=0 \qquad \text{and} \qquad
(\mathscr{B}(u))(y)=0,  \qquad 
\end{equation}
  let $(\Omega,\mathcal F,\PP)$ be a probability space,
  let  $\xrandom\colon \Omega\to D$ and $\yrandom\colon \Omega\to \thmW$
    be independent  random variables, 
       assume 
    for all
    $\mathcal{X},\mathcal{Y}\in\{\mathcal{Z}\subseteq \R^d\colon\mathcal{Z} \text{ is open}\}$
       with $\mathcal{X}\cap D \neq \emptyset$ and $\mathcal{Y}\cap \thmW \neq \emptyset $ 
    that $\PP( \{\xrandom\in \mathcal{X}\} \cap \{\yrandom\in \mathcal{Y}\})>0$,  
let  $\lossexp\colon \basicset\to [0,\infty]$ satisfy for all $v\in \basicset$ that
 \beq\label{Ivgeneraleqcopy}\begin{split}
  \lossexp(v)=\E \bigl[\|(\mathscr{D}(v))(\xrandom) \|^2 +\|(\mathscr{B}(v))(\yrandom) \|^2 \bigr],  \end{split}\eeq
 	 and let $v \in \basicset$.   Then the following two statements are equivalent:
 	\begin{enumerate}[(i)]
 	\item\label{thm23item1copy} 
 	  It holds that
 	 $\lossexp(v)=\inf_{ w \in \basicset }\lossexp(w)$.
 \item 	 It holds
   for all  $x\in D$,  $y\in \thmW$
   that 
  $
(\mathscr{D}(v))(x)=0 \qquad \text{and} \qquad
(\mathscr{B}(v))(y)=0$.
\end{enumerate}

 }
 \end{theorem}
 \begin{cproof}{thm213acopy}
 Throughout  this proof let $\solutionset\subseteq  \basicset$ satisfy 
 \beq\label{defofthmSacopy}
 \solutionset=\bigl\{v\in \basicset\colon 
  \bigl(\forall \,   x\in D,  y\in  \thmW\colon  \|(\mathscr{D}(v))(x)\|=
\|(\mathscr{B}(v))(y)\|=0
\bigr)\bigr\}
\eeq
  and for every $v\in  \basicset$   let  $U_v\colon
   D\times  \thmW
    \to \R^{\delta+b}$  satisfy
   for all   $x\in D$,        $y\in \thmW$ that
 \beq\label{scrFwacopy}
 U_v(x,y)=\bigl( (\mathscr{D}(v))(x), (\mathscr{B}(v))(y) \bigr).
 \eeq
\Nobs\  the assumption that $(\Omega,\mathcal F,\PP)$ is a probability space  \proves\  that $\Omega \neq \emptyset$.  This and the fact that $\xrandom\colon \Omega\to D$ and  $\yrandom\colon \Omega\to \thmW$ 
    are functions \prove\ that
\beq D \neq \emptyset \qquad \text{and} \qquad \thmW \neq \emptyset. \eeq
\Moreover\   \cref{defofthmSacopy} and
\cref{scrFwacopy} \prove\ that for all  $v\in\solutionset$ it holds that 
$ \P( U_v(\xrandom,\yrandom)= 0)=1$.
\Moreover\
the assumption that 
 for all  
   $\mathcal{X},\mathcal{Y}\in\{\mathcal{Z}\subseteq \R^d\colon\mathcal{Z}\text{ is open}\}$
   with $\mathcal{X}\cap D\neq \emptyset$ and $\mathcal{Y}\cap \thmW\neq \emptyset $ 
it holds  that $\PP( \{\xrandom\in \mathcal{X}\} \cap \{\yrandom\in \mathcal{Y}\} )>0$
 \proves\  that 
for all  open  $\mathcal{X}\subseteq \R^d$ with $\mathcal{X}\cap  D\neq  \emptyset$ it holds
 that $\PP( \{\xrandom\in \mathcal{X}\} )>0$.
\Moreover\    \cref{scrFwacopy} \proves\ that  for all  $v\in \basicset$ with $ \P( U_v(\xrandom,\yrandom)= 0)=1$      it holds that
\beq
 \P\bigl((\mathscr{D}(v))(\xrandom)=0\bigr)=\P\bigl((\mathscr{B}(v))(\yrandom)=0\bigr)=1.
 \eeq
 Combining this with the fact that
  for all   open  $\mathcal{X}\subseteq \R^d$ with $\mathcal{X}\cap  D\neq  \emptyset$
it holds that $\PP( \{\xrandom\in \mathcal{X}\} )>0$ and \cref{sec2lem6}
 (applied 
  for every $v\in \basicset$
  with  $d\with d$,  $ \delta \with \delta$,   $f\with (  D \ni x  \mapsto (\mathscr{D}(v))(x) \in \R^{\delta})$ in the notation of \cref{sec2lem6})  \proves\ that
   for all  $v\in \basicset$,  $x\in D$  with $ \P( U_v(\xrandom,\yrandom)= 0)=1$
it holds that 
     \begin{equation}\label{Dv0acopy}
(\mathscr{D}(v))(x)=0.
\end{equation}
\Moreover\
the assumption that 
 for all  
   $\mathcal{X},\mathcal{Y}\in\{\mathcal{Z}\subseteq \R^d\colon\mathcal{Z} \text{ is open}\}$
   with $\mathcal{X}\cap D\neq \emptyset$ and $\mathcal{Y}\cap \thmW\neq \emptyset $ it holds
 that $\PP( \{\xrandom\in \mathcal{X}\} \cap \{\yrandom\in \mathcal{Y}\} )>0$
 \proves\ that 
for all  open  $\mathcal{Y}\subseteq \R^d$ with $\mathcal{Y}\cap  \thmW \neq  \emptyset$ it holds that $\PP( \{\yrandom\in \mathcal{Y}\} )>0$.  Combining this and 
\cref{sec2lem6} (applied  
  for every $v\in \basicset$
  with
  $d\with d$,  $ \delta \with b$, 
   $f\with  ( \thmW   \ni  y \mapsto ( \mathscr{B}(v))(y)\in \R^{\thmb} ) $ in the notation of \cref{sec2lem6})  \proves\ that
 for all  $v\in \basicset$,   $y\in \thmW$ with $ \P( U_v(\xrandom,\yrandom)= 0)=1$   it holds that 
  \beq \label{Bv0acopy}
  (\mathscr{B}(v))(y)=0.
  \eeq
 This and \cref{Dv0acopy} \prove\ that  for all $v\in \basicset$ with $ \P( U_v(\xrandom,\yrandom)= 0)=1$    it holds  that
 $v\in \solutionset$.
The assumption that $u\in \basicset$ \hence  \proves\  that $S \neq \emptyset$ and $V \neq \emptyset$.
Combining  this with \cref{Dv0acopy}, \cref{Bv0acopy},  and  \cref{lemmaBcor} (applied with
$ \mathfrak{R}\with   D\times  \thmW $,
$N\with \delta+b$,  $R\with (\xrandom, \yrandom)$, 
$(U_v)_{v\in\basicset} \with(U_v)_{v\in\basicset}$,
$\solutionset\with\solutionset$,  $\basicset\with \basicset$,     $(\Omega, \F,\P) \with (\Omega, \F,\P)$ in the notation of  \cref{lemmaBcor})    \proves\ 
(\cref{thm23item1} \allowbreak$  \longleftrightarrow  $\allowbreak \cref{thm23item2}).
  \end{cproof}

\begin{corollary}[\PINNs\ for Laplace equations]\label{thm213acopyLap}
  {%
 Let  $d,\thmb \in \N$,    let $D\subseteq \R^d$,
    let $f\in C(\partial D,\R)$,  
  $u\in C^2(  D,\R)$ satisfy  for all  $x\in D$,  $y\in \partial D$  that 
    \begin{equation}\label{weakpdeuacopylap}
  (  \Delta u)(x)=0 \qquad \text{and} \qquad
  u(y)=f(y),  \qquad 
  \end{equation}
    let $(\Omega,\mathcal F,\PP)$ be a probability space,
    let  $\xrandom\colon \Omega\to D$ and $\yrandom\colon \Omega\to \partial D$
      be independent  random variables, 
        assume 
      for all
      $\mathcal{X},\mathcal{Y}\in\{\mathcal{Z}\subseteq \R^d\colon\mathcal{Z} \text{ is open}\}$
        with $\mathcal{X}\cap D \neq \emptyset$ and $\mathcal{Y}\cap \partial D \neq \emptyset $ 
      that $\PP( \{\xrandom\in \mathcal{X}\} \cap \{\yrandom\in \mathcal{Y}\})>0$,  
  let  $\lossexp\colon C^2(  D,\R) \to [0,\infty]$ satisfy for all $v\in C^2(  D,\R)$ that
  \beq\label{IvgeneraleqcopyLap}\begin{split}
    \lossexp(v)=\E \bigl[|(  \Delta v)(\xrandom) |^2 +|v(\yrandom) - f(\yrandom)|^2 \bigr],\end{split}\eeq
    and let $v\in C^2(  D,\R)$.  Then the following two statements are equivalent:
    \begin{enumerate}[(i)]
    \item\label{thm23item2copyLap}
  It holds that
      \beq\label{eq99acopyLap}
  \lossexp(v)=\inf_{ w \in C^2(  D,\R) }\lossexp(w).
  \eeq
    \item\label{thm23item1copyLap} 
    It holds
    for all  $x\in D$,  $y\in \partial D$
    that 
    \begin{equation}\label{eq100acopyLap}
  (  \Delta v)(x)=0 \qquad \text{and} \qquad
  v(y)=f(y).
  \end{equation}
    
    \end{enumerate}

  }
  \end{corollary}
  \begin{cproof}{thm213acopyLap} 
  \Nobs\ 
  \cref{thm213acopy} (applied with  $d\with d$,  $\delta\with 1$,  $\thmb\with\thmb$, $D\with D$,  $\thmW\with \partial D$, $\basicset\with   C^2(  D,\R)$,   $\mathscr{D}\with (  C^2(D,\R)\ni u\mapsto \Delta u\in C(D,\R))  $,   $\mathscr{B}\with (C^2(D,\R)\ni v\mapsto v|_{ \partial D } -f\in  C(\partial D,\R))$ in the notation of \cref{thm213acopy}) \proves\    (\cref{thm23item2copyLap}\allowbreak$  \longleftrightarrow  $\allowbreak\cref{thm23item1copyLap}).
  \end{cproof}
    
  \cfclear

 \begin{lemma}\label{NinCP}
{%
  Let $ \mathfrak{d},p, L\in \N$,   $\theta \in \R^{ \mathfrak{d} }$, $\act\in C^p(\R,\R)$,   $l_0,l_1,\dots,l_L
\in \N$ satisfy $\mathfrak d=\sum_{ k = 1 }^{ L } l_k (l_{k-1} + 1)$.
Then
$ \cfadd{def:nets}{ \mathcal N}^{l_0,l_1,\dots,l_L}_{\act,\theta}
 \in C^p(\R^{l_0},\R^{l_L})$
\cfload.
}
\end{lemma}
 \begin{cproof}{NinCP}
 \Nobs\
  item (i) in Lemma  6.1.5,  item (i) in  Lemma 6.1.6,  and  item (i) in 6.1.7  in \cite{Jentzen2023} \prove\ that 
 $ \cfadd{def:nets}{ \mathcal N}^{l_0,l_1,\dots,l_L}_{\act,\theta}
 \in C^p(\R^{l_0},\R^{l_L})$
    (cf.,  \ex  the proofs of Lemma 10.13.2,  Lemma 10.13.3,  and Corollary 10.13.4 in \cite{Jentzen2023}).
 \end{cproof}

  \begin{algo}[\PINN\  methods  for general boundary value \PDE\ problems\footnote{Note that for every $ k, d, \delta \in \N $ and every $D \subseteq \R^d$ it holds that $C^k( D, \R^{ \delta } ) = \{ f \in C( D, \R^{ \delta } ) \colon ( \exists\, g \in C^k( \R^d, \R^{ \delta } ) \colon g|_D = f )\}$.}]\label{algo26copy}
  {%
 Let  $d, \delta,  b,\thmp \in \N$,     let $D\subseteq \R^d$,   
  let $\thmW \subseteq \partial D$,
  let $\mathscr{D}\colon  C^{\thmp}(  D, \R^{ \delta } ) \to C( D,\R^{\delta})$ and $\mathscr{B}\colon  C^{\thmp}(  D, \R^{ \delta } ) \to C(\thmW,\R^b)$ be functions,  let 
 $u\in C^p(  D,\R^{\delta})$ satisfy  for all  $x\in D$,  $y\in \thmW$  that 
  \begin{equation}\label{weakpdeu2copy}
(\mathscr{D}(u))(x)=0 \qquad \text{and}\qquad
(\mathscr{B}(u))(y)=0,  
\end{equation}
  let  $\act \in C^{ \thmp }( \R, \R )$,   $\mathfrak{d},L\in \N$,  $ l_0, l_1,\dots, l_{L} \in \N$ 
  satisfy $l_0=d$,  $l_L = \delta$,  and $\mathfrak d=\sum_{ k = 1 }^{ L } l_k (l_{k-1} + 1)$,
 let  $ \mathscr{L} \colon \R^{ \mathfrak{d}} \times D\times \thmW \to \R$ satisfy for all $\theta \in \R^{ \mathfrak{d}} $,  $x\in D$, $y\in \thmW$ that
 \beq 
  \mathscr{L}( \theta,x,y ) =   \|(\mathscr{D}(\cfadd{def:nets,NinCP}{\mathcal N}^{l_0,l_1,\dots,l_L}_{\act,\theta}|_{ D }))(x) \|^2 +\|(\mathscr{B}(\cfadd{def:nets,NinCP}{\mathcal N}^{l_0,l_1,\dots,l_L}_{\act,\theta}|_{ D })(y) \|^2
 \eeq  \cfload, \cfclear
   let $ \algoG \colon \R^{ \mathfrak{d} } \times D\times \thmW \to \R^{ \mathfrak{d} }$ 
satisfy
for all $x \in D$,  $y \in \thmW$, $ \theta \in \{ \vartheta\in \R^{ \mathfrak{d}}\colon   \mathscr{L}( \cdot,x,y )  \text{  is  differentiable at }  \vartheta \}$ 
that $ \algoG( \theta,x,y)=( \nabla_{ \theta } \mathscr{L} )(\theta,x,y)$, 
   let $(\Omega,\mathcal F,\PP)$ be a probability space,
     let  $M, \algoM \in \N$, 
     let  $\xrandom_{ m}\colon \Omega\to D$,  $m \in\{1,2,\dots, M\}$, 
   be  i.i.d.\  random variables,     let $\yrandom_{m} \colon \Omega\to W $,  $m \in\{1,2,\dots, M\}$, be  i.i.d.\  random variables, 
  for every $n\in \N$, $m\in\{1,2,\dots, \algoM\}$  let $ \xi_{ n, m } \colon\Omega \to \{ 1, 2, \dots, M \}$ be  a random variable,
 assume that       $( \xrandom_{ m})_{m\in  \{1,2,\dots, M\}}$,    $( \yrandom_{ m})_{m\in  \{1,2,\dots, M\}}$,   and     $ (\xi_{ n, m })_{(n, m)\in\N\times \{1,2,\dots, \algoM\}}$      are  independent,
let $(\gamma_n )_{n \in \N} \subseteq \R$, 
 and let $\Theta\colon \N_0 \times \Omega \to \R^{ \mathfrak{d}}$ satisfy for all $n \in \N_0$ that 
\begin{equation}
\Theta_{ n  } = \Theta_{n-1} - \frac{ \gamma_n }{ \algoM }\left[  \sum_{ m = 1 }^{\algoM}\algoG( \Theta_{ n - 1 }, \xrandom_{ \xi_{ n, m } }, \yrandom_{ \xi_{ n, m } } )\right].
\end{equation}  }
   \end{algo}

\begin{remark}[Explanations for \cref{algo26copy}]
Loosely speaking, 
in \cref{algo26copy} we think for sufficiently large 
  $N \in \N$  
of the \ANN\ realization 
$
  \cfadd{def:nets,NinCP}{\mathcal N}^{l_0,l_1,\dots,l_L}_{\act,\Theta_N}|_{ D }
$
as an approximation 
\begin{equation}
\label{T_B_D}
\begin{split}
  &\cfadd{def:nets,NinCP}{\mathcal N}^{l_0,l_1,\dots,l_L}_{\act,\Theta_N}|_{ D } \approx u
\end{split}
\end{equation}
of the solution $u\in C^p(  D,\R^{\delta})$ of the boundary value \PDE\ problem in \eqref{weakpdeu2copy}.
\end{remark}
    
\begin{remark}
In \cref{algo26copy} we employ fully-connected feedforward \ANNs\ as defined in \cref{def:nets} to approximate \PDE\ solutions.
Many other architectures could be used; for example, random feature networks where only the last layer is learned 
(cf., e.g., \cite{MR4633578} for results analyzing such methods).
\end{remark}

\subsubsection{PINNs for time-dependent  initial value PDE problems}
\label{sect:PINNsIVP}

In this section we introduce a \PINNs\ methodology for general time-dependent initial value \PDE\ problems in \cref{algo26} below.
This section proceeds in an analogous way to \cref{sect:PINNsBVP}:
We motivate \cref{algo26} with a reformulation result in \cref{thm213a} which allows recasting time-dependent initial value \PDE\ problems as stochastic minimization problems.

 \begin{corollary}[\PINNs\ for time-dependent initial value \PDE\ problems]\label{thm213a}
 {%
 Let  $T\in (0,\infty)$,  $d,\delta,\thmb \in \N$,    let $D\subseteq \R^d$,
  let   $\phi\in C(D,\R^{\delta})$,  let $\basicset\subseteq C( [0,T] \times D,\R^{\delta})$,  let $\mathscr{D}\colon \basicset\to C([0,T] \times D,\R^{\delta}) $ and $\mathscr{B}\colon \basicset\to C([0,T]\times\partial D,\R^{\thmb})$ be functions,  let 
 $u\in \basicset$ satisfy  for all $t\in [0,T]$,  $x\in D$,  $y\in \partial D$  that 
  \begin{equation}\label{weakpdeua}
(\mathscr{D}(u))(t,x)=0,  \qquad
(\mathscr{B}(u))(t,y)=0,  \qquad \text{and} \qquad
u(0,x)=\phi(x),
\end{equation}
  let $(\Omega,\mathcal F,\PP)$ be a probability space,
  let  $\xrandom\colon \Omega\to D$,  $\yrandom\colon \Omega\to \partial D$,  and $\trandom\colon \Omega\to[0,T]$ 
    be independent  random variables, 
     assume 
    for all  $r\in (0,T)$,  $s\in (r,T)$, 
   $\mathcal{X},\mathcal{Y}\in\{\mathcal{Z}\subseteq \R^d\colon\mathcal{Z} \text{ is open}\}$
    with $\mathcal{X}\cap D\neq \emptyset$ and $\mathcal{Y}\cap \partial D\neq \emptyset $ 
 that $\PP( \{\xrandom\in \mathcal{X}\} \cap \{\yrandom\in \mathcal{Y}\} \cap\{  r<\trandom<s\})>0$,  
and   let  $\lossexp\colon \basicset\to [0,\infty]$ satisfy for all $v\in \basicset$ that
 \beq\label{Ivgeneraleq}\begin{split}
  \lossexp(v)=\E \bigl[\|(\mathscr{D}(v))(\trandom,\xrandom) \|^2 +\|(\mathscr{B}(v))(\trandom,\yrandom) \|^2
+ \| v(0,\xrandom)-\phi(\xrandom)\|^2\bigr].\end{split}\eeq
 	Then 
 	\begin{enumerate}[(i)]
 	\item\label{thm23item1} 
 	for all $v\in \basicset$  it holds that
 	 \beq\label{eq99a}
\lossexp(v)=\inf_{ w \in \basicset }\lossexp(w)
\eeq
  if and only if 	 it holds
   for all $t\in [0,T]$,  $x\in D$,  $y\in \partial D$
   that 
  \begin{equation}\label{eq100a}
(\mathscr{D}(v))(t,x)=0,  \qquad
(\mathscr{B}(v))(t,y)=0,  \qquad \text{and} \qquad
v(0,x)=\phi(x)
\end{equation}
and
 	\item\label{thm23item2}
 	there exists $v\in \basicset$ such that 
 	\beq
 	\lossexp(v)=\inf_{ w \in \basicset }\lossexp(w).
 	\eeq
 	\end{enumerate}

 }
 \end{corollary}
 
 \begin{cproof}{thm213a}
 \Nobs\ \cref{thm213acopy} (applied with $d\with d+1$, $\delta\with\delta$, $b\with b$,  $D\with [0,T]\times D$,  $\thmW\with [0,T]\times \partial D$   in the notation of \cref{thm213acopy})
 \proves\ \cref{thm23item1,thm23item2}.
 \end{cproof}

  \begin{algo}[\PINN-\SGD\ methods for time-dependent initial value \PDE\ problems\footnote{Note that for every $ k, d, \delta \in \N $ and every $D \subseteq \R^d$ it holds that $C^k( D, \R^{ \delta } ) = \{ f \in C( D, \R^{ \delta } ) \colon ( \exists\, g \in C^k( \R^d, \R^{ \delta } ) \colon g|_D = f )\}$.}]\label{algo26}
  {%
 Let $T\in (0,\infty)$,  $d, \delta,  b,\thmp \in \N$,     let $D\subseteq \R^d$,  let   $\phi\in C(D,\R^{\delta})$,   let $\mathscr{D}\colon  C^{\thmp}( [0,T] \times D, \R^{ \delta } ) \to C([0,T] \times D,\R^{\delta})$ and $\mathscr{B}\colon  C^{\thmp}( [0,T] \times D, \R^{ \delta } ) \to C([0,T]\times\partial D,\R^b)$ be functions,  let 
 $u\in C^p( [0,T] \times D,\R^{\delta})$ satisfy  for all $t\in [0,T]$,  $x\in D$,  $y\in \partial D$  that 
  \begin{equation}\label{weakpdeu2}
(\mathscr{D}(u))(t,x)=0,  \qquad
(\mathscr{B}(u))(t,y)=0,  \qquad \text{and} \qquad
u(0,x)=\phi(x),
\end{equation}
let  $\act \in C^{ \thmp }( \R, \R )$, 
$\mathfrak{d},L\in \N$,  $ l_0, l_1,\dots, l_{L} \in \N$ 
  satisfy $l_0=d+1$,  $l_L = \delta$,  and $\mathfrak d=\sum_{ k = 1 }^{ L } l_k (l_{k-1} + 1)$,
   let $ \mathscr{L}\colon  \R^{ \mathfrak{d}}\times [0,T]\times   D  \times \partial D \to \R$ satisfy for all $\theta \in \R^{ \mathfrak{d}} $,  $t\in [0,T]$, $x\in D$, $y\in \partial D$ that
 \beq 
 \begin{split}
  \mathscr{L}(\theta,t,x,y )& = 
  \|(\mathscr{D}(\cfadd{def:nets,NinCP}{\mathcal N}^{l_0,l_1,\dots,l_L}_{\act,\theta}|_{  [0,T] \times D }))(t,x) \|^2 +\|(\mathscr{B}(\cfadd{def:nets,NinCP}{\mathcal N}^{l_0,l_1,\dots,l_L}_{\act,\theta}|_{  [0,T] \times  D }))(t,y) \|^2\\
  &\quad+ \| \cfadd{def:nets,NinCP}{\mathcal N}^{l_0,l_1,\dots,l_L}_{\act,\theta}(0,x)-\phi(x)\|^2
  \end{split}
 \eeq 
 \cfload, \cfclear
  let $ \algoG\colon   \R^{ \mathfrak{d}}\times [0,T]\times   D  \times \partial D  \to \R^{ \mathfrak{d} }$ 
satisfy
for all
$t\in [0,T]$,
 $x \in D$,  $y \in \partial D$, $ \theta \in \{ \vartheta\in \R^{ \mathfrak{d}}\colon   \mathscr{L}( \cdot,x,y )  \text{  is  differentiable at }  \vartheta \}$ 
that $ \algoG( t,\theta,x,y)=( \nabla_{ \theta } \mathscr{L} )(t,\theta,x,y)$, 
   let $(\Omega,\mathcal F,\PP)$ be a probability space,
     let  $M, \algoM \in \N$, 
     let  $\xrandom_{ m}\colon \Omega\to D$,  $m \in\{1,2,\dots, M\}$, 
   be  i.i.d.\  random variables,     let $\yrandom_{m} \colon \Omega\to \partial D $,  $m \in\{1,2,\dots, M\}$, be  i.i.d.\  random variables, 
     let $\trandom_{m} \colon \Omega\to [0,T] $,  $m \in\{1,2,\dots, M\}$, be  i.i.d.\  random variables, 
   for every $n\in \N$, $m\in\{1,2,\dots, \algoM\}$   let $ \xi_{ n, m } \colon\Omega \to \{ 1, 2, \dots, M \}$ be  a random variable,
 assume that       $( \xrandom_{ m})_{m\in  \{1,2,\dots, M\}}$,    $( \yrandom_{ m})_{m\in  \{1,2,\dots, M\}}$,  $( \trandom_{ m})_{m\in  \{1,2,\dots, M\}}$,   and     $ (\xi_{ n, m })_{(n,m)\in\N\times \{1,2,\dots, \algoM\}}$      are  independent,
let $(\gamma_n )_{n \in \N} \subseteq \R$, 
 and let $\Theta\colon \N_0 \times \Omega \to \R^{ \mathfrak{d}}$ satisfy for all $n \in \N_0$ that 
\begin{equation}
\Theta_{ n + 1 } = \Theta_n - \frac{ \gamma_n }{ \algoM }\left[  \sum_{ m = 1 }^{\algoM}\algoG( \Theta_{ n - 1 }, \trandom_{ \xi_{ n, m } },\xrandom_{ \xi_{ n, m } }, \yrandom_{ \xi_{ n, m } } )\right].
\end{equation}
  }
\end{algo}

\begin{remark}[Explanations for \cref{algo26}]
  Loosely speaking, 
  in \cref{algo26} we think for sufficiently large 
    $N \in \N$  
  of the \ANN\ realization 
  $
  \cfadd{def:nets,NinCP}{\mathcal N}^{l_0,l_1,\dots,l_L}_{\act,\Theta_N}|_{  [0,T] \times D }
  $
  as an approximation 
  \begin{equation}
  \label{T_B_D}
  \begin{split}
  \cfadd{def:nets,NinCP}{\mathcal N}^{l_0,l_1,\dots,l_L}_{\act,\Theta_N}|_{  [0,T] \times D }\approx u
  \end{split}
  \end{equation}
  of the solution $u \in C^p( [0,T] \times D,\R^{\delta})$ of the time-dependent initial value \PDE\ problem in \cref{weakpdeu2}.
\end{remark}

\subsubsection{PINNs for free boundary Stefan problems}

In this section we introduce a \PINNs\ methodology for free boundary Stefan problems in \cref{algo:stefan} below.
This section proceeds in an analogous way to \cref{sect:PINNsBVP,sect:PINNsBVP}:
We motivate \cref{algo:stefan} with a reformulation result in \cref{stefanproblem} which allows recasting free boundary Stefan problems as stochastic minimization problems.
The Stefan problem considered in this section and the corresponding \PINNs\ methodology in \cref{algo:stefan} are based on \cite{MR4199366}.

\begin{theorem}[\PINNs\ for free boundary Stefan problems]\label{stefanproblem}
	{%
 Let  $T, \psi \in (0,\infty)$,    
	$\phi,g,h_0,h_1 \in C(\R,\R)$,
	 let $\basicset = C^{ 1, 2 }( [0,T]\times\R, \R ) \times C( [0,T], \R)$, 
	 let $ u = (u_1, u_2) = ((u_1(t,x))_{ (t,x) \in [0,T]\times\R }, (u_2(t))_{t \in [0,T] }) \in \basicset
	 $,
let $D=\{(t,x) \in (0,T) \times (0,\infty) \colon x < u_2(t)  \}$,
  assume   for all $(t,x) \in  D$  that
\beq  \tfrac{\partial u_1}{\partial t}(t,x)=\tfrac{\partial^2 u_1}{\partial x^2}(t,x), \qquad
\tfrac{\partial u_1}{\partial x}(t,0) = g(t),   \qquad u_1(t,u_2 (t)) = h_0(t),  \eeq
 \beq 
   \tfrac{\partial u_1}{\partial x}(t,u_2 (t)) =  h_1(t), \qquad  u_1(0,x) =  \phi(x),  \qquad\text{and}\qquad 
     u_2(0)=\psi, \eeq
 let $(\Omega,\mathcal F,\PP)$ be a probability space,
  let   $\trandom\colon \Omega\to \R$  and $\xrandom\colon \Omega\to \R$ be  random variables which satisfy  for all non-empty  open $U\subseteq D$ that $0<\P((\trandom,\xrandom)\in U)\leq \P((\trandom,\xrandom)\in D)=1$,
and  let $\lossexp\colon \basicset \to [0,\infty]$ satisfy for all $v=(v_1,v_2)\in \basicset$  that
 \beq
\begin{split}
\lossexp(v)&=
\E\big[|\tfrac{\partial v_1}{\partial t}(\trandom,\xrandom) - \tfrac{\partial^2 v_1}{\partial x^2}(\trandom,\xrandom)|^2 
+ |\tfrac{\partial v_1}{\partial x}(\trandom,0)- g(\trandom)  |^2+|v_1(\trandom,v_2(\trandom))-h_0(\trandom)|^2\\&\quad +|\tfrac{\partial v_1}{\partial x}(\trandom,v_2(\trandom)) - h_1(\trandom)|^2+| v_1(0,\xrandom)- \phi(\xrandom) |^2+|v_2(0) - \psi|^2\big].
\end{split}
\eeq
 	Then 
 	\begin{enumerate}[(i)]
 	\item\label{thm28item1} 
 	for all $v=(v_1,v_2)\in \basicset$  it holds that
 	 \beq\label{eq992}
\lossexp(v)=\inf_{ w \in \basicset}\lossexp(w)
\eeq
  if and only if 	 
    it holds  for all $(t,x) \in  D$ that
\beq  \tfrac{\partial v_1}{\partial t}(t,x)  = \tfrac{\partial^2 v_1}{\partial x^2}(t,x),
\qquad \tfrac{\partial v_1}{\partial x}(t,0) = g(t),  \qquad v_1(t,v_2(t)) = h_0(t), 
\eeq
 \beq   \qquad
   \tfrac{\partial v_1}{\partial x}(t,v_2(t)) =  h_1(t), \qquad v_1(0,x) =  \phi(x), \qquad \text{and}
   \qquad v_2(0) = \psi,\eeq
and
   	\item\label{thm28item2} 
	there exists $v\in \basicset$ such that 
	\beq   \lossexp(v)=\inf_{ w \in \basicset}\lossexp(w).           \eeq
 	\end{enumerate}
	}
\end{theorem}

\begin{cproof}{stefanproblem}
Throughout  this proof
for every $v \in V$ let $U_v \colon D \to  \R^6$ satisfy for all $(t,x) \in D$ that
\begin{multline}\label{defoffvtx}
U_{v}(t,x)=\bigl(\tfrac{\partial v_1}{\partial t}(t,x) - \tfrac{\partial^2 v_1}{\partial x^2}(t,x),\tfrac{\partial v_1}{\partial x}(t,0)- g(t) , 
v_1(t,v_2(t))-h_0(t), \\ \tfrac{\partial v_1}{\partial x}(t,v_2(t)) - h_1(t),v_1(0, x)- \phi(x),v_2(0) - \psi\bigr)
\end{multline}
and let
 $\solutionset\subseteq   V$ satisfy 
 \beq\label{defofSthm215} \begin{split}
 \solutionset=\{ v\in    V \colon
 \textstyle\sup_{(t,x)\in D}\| U_{v}(t,x)\|=0    \}.
\end{split}\eeq
\Nobs\  \cref{defofSthm215}, the assumption that $\P((\trandom,\xrandom)\in D)=1$,  the fact
that 
$ \{ \phi, g, h_0, h_1 \} \subseteq C( R, R )$,   and   the fact that  $[0,T]\times\R \supseteq D$
\prove\  that  for all $v\in \solutionset$ it holds that 
$ \P( U_{v}(\trandom,\xrandom)= 0)=1$.
\Moreover\ \cref{defoffvtx} \proves\  that for all $v\in V$
with 
\beq\label{plv0eq1}  \P( U_{v}(\trandom,\xrandom)= 0)=1 \eeq
it holds that  $v_2(0)=\psi$ and 
\beq\begin{split}
\P\bigl(\tfrac{\partial v_1}{\partial t}(\trandom,\xrandom)&=\tfrac{\partial^2 v_1}{\partial x^2}(\trandom,\xrandom)\bigr)=\P\bigl(v_1(0,\xrandom)= \phi(\xrandom)
\bigr)=\P\bigl( \tfrac{\partial v_1}{\partial x}(\trandom,0)= g(\trandom)  \bigr)\\
&=\P\bigl( v_1(\trandom,v_2(\trandom))=h_0(\trandom)\bigr)
= \P\bigl( \tfrac{\partial v_1}{\partial x}(\trandom,v_2(\trandom)) = h_1(\trandom)\bigr)
=1.
\end{split}\eeq
Combining  this,  the fact that  for all $v\in \basicset$  with
\cref{plv0eq1} it holds that  $v_2(0)=\psi$,   the assumption 
that for all non-empty  open $U\subseteq D$ it holds that $\P((\trandom,\xrandom)\in U)>0$,
 and 
\cref{sec2lem6} (applied with  $d\with 2$,  $ \delta \with 6$,   $f\with (  D \ni (t,x ) \mapsto U_{v}(t,x) \in \R^6)$ in the notation of \cref{sec2lem6})  \proves\ that
for all  $(t,x)\in D$,  $v\in  V$
with $ \P(  U_{v}(\trandom,\xrandom)= 0)=1$ it holds   that 
\beq U_{v}(t,x)=0.\eeq
\Hence that for all $v\in  V$
with $ \P(  U_{v}(\trandom,\xrandom)= 0))=1$ it holds that   $v\in S$.
Combining this and \cref{lemmaBcor}
(applied with
$N\with 6$,  $R\with (\trandom, \xrandom)$, 
$(U_v)_{v\in\basicset} \with (U_v)_{v\in\basicset}$,
$\solutionset\with\solutionset$,  $\basicset\with \basicset$,     $(\Omega, \F,\P) \with (\Omega, \F,\P)$ in the notation of  \cref{lemmaBcor}) \proves\  \cref{thm28item1,thm28item2}.
\end{cproof}
\cfclear

\begin{algo}[\PINNs\ for free boundary Stefan problems]
\label{algo:stefan}
{%
   Let  $T, \psi \in (0,\infty)$,    
	$\phi,g,h_0,h_1 \in C(\R,\R)$,
	 let $V = C^{ 1, 2 }( [0,T]\times\R, \R ) \times C( [0,T], \R)$, 
	 let $ u = (u_1, u_2) =\allowbreak ((u_1(t,\allowbreak x))_{ (t,x) \in [0,T]\times\R }, \allowbreak(u_2(t))_{t \in [0,T] }) \in V
	 $,
let $D=\{(t,x) \in (0,T) \times (0,\infty) \colon x < u_2(t)  \}$,
  assume   for all $(t,x) \in  D$  that
\beq  \label{stefan:eq1}
\tfrac{\partial u_1}{\partial t}(t,x)=\tfrac{\partial^2 u_1}{\partial x^2}(t,x), \qquad
\tfrac{\partial u_1}{\partial x}(t,0) = g(t),   \qquad u_1(t,u_2 (t)) = h_0(t),  \eeq
 \beq \label{stefan:eq2}
   \tfrac{\partial u_1}{\partial x}(t,u_2 (t)) =  h_1(t), \qquad  u_1(0,x) =  \phi(x),  \qquad\text{and}\qquad 
     u_2(0)=\psi, \eeq
     let $\act \in C^{ 2 }( \R, \R )$,  $\mathfrak{d}_1,\mathfrak{d}_2,L_1,L_2 \in \N$,   $l_0^{1}, l_1^{1}, \dots, l_{L_1}^{1},  l_0^{2}, l_1^{2}, \dots, l_{L_2}^{2}  \in \N$ 
  satisfy 
  for all $j\in \{1,2\}$ that $l_0^j=3-j$,
   $l_{L_j}^j=1$,  $\mathfrak d_j=\sum_{ k = 1 }^{ L_j } l_k^j(l_{k-1}^j+ 1)$,
 for every $n\in \N$ let $ \mathscr{L}_n\colon \R^{ \mathfrak{d}_1+\mathfrak{d}_2}  \times \Omega \to \R$ satisfy for all
$ \theta = ( \theta_1,  \dots,  \theta_{  \mathfrak{d}_1 +  \mathfrak{d}_2 } ) \in \R^{  \mathfrak{d}_1 +  \mathfrak{d}_2 }$, $(t,x)\in D$
    that
   \beq\begin{split}
    \mathscr{L}(\theta,t,x)&= 
  \Bigl( 
    | 
      \pr{
        \tfrac{\partial }{\partial t}\cfadd{def:nets,NinCP}{\mathcal N}^{l_0^1,l_1^1,\dots,l_{L_1}^1}_{\act,(  \theta_1, \theta_2,  \dots, \theta_{  \mathfrak{d}_1  } )}
      }(t,x)
      - 
      \pr{
        \tfrac{\partial^2 }{\partial x^2}\cfadd{def:nets,NinCP}{\mathcal N}^{l_0^1,l_1^1,\dots,l_{L_1}^1}_{\act,(  \theta_1, \theta_2,  \dots, \theta_{  \mathfrak{d}_1  } )}
      }(t,x)
    |^2 
    + 
    |
      \pr{
        \tfrac{\partial }{\partial x}\cfadd{def:nets,NinCP}{\mathcal N}^{l_0^1,l_1^1,\dots,l_{L_1}^1}_{\act,(  \theta_1, \theta_2,  \dots, \theta_{  \mathfrak{d}_1  } )}
      }(t,0)
      - 
      g(t)  
    |^2 
    \\&\quad
    + 
    |
        \cfadd{def:nets,NinCP}{\mathcal N}^{l_0^1,l_1^1,\dots,l_{L_1}^1}_{\act,(  \theta_1, \theta_2,  \dots, \theta_{  \mathfrak{d}_1  } )}(t,\cfadd{def:nets,NinCP}{\mathcal N}^{l_0^2,l_1^2,\dots,l_{L_2}^2}_{\act,(\theta_{\mathfrak{d}_1+1},    \theta_{\mathfrak{d}_1+2},\dots,   \theta_{  \mathfrak{d}_1 +  \mathfrak{d}_2 } )        }(t))
      -
      h_0(t)
    |^2
    \\&\quad
    +
    |
      \pr{
        \tfrac{\partial }{\partial x}\cfadd{def:nets,NinCP}{\mathcal N}^{l_0^1,l_1^1,\dots,l_{L_1}^1}_{\act,(  \theta_1, \theta_2,  \dots, \theta_{  \mathfrak{d}_1  } )}
      }(
        t,
        \cfadd{def:nets,NinCP}{\mathcal N}^{l_0^2,l_1^2,\dots,l_{L_2}^2}_{\act,(\theta_{\mathfrak{d}_1+1},    \theta_{\mathfrak{d}_1+2},\dots,   \theta_{  \mathfrak{d}_1 +  \mathfrak{d}_2 }}
      )
      - 
      h_1(t)
    |^2\\
    &\quad
    +
    | 
      \cfadd{def:nets,NinCP}{\mathcal N}^{l_0^1,l_1^1,\dots,l_{L_1}^1}_{\act,(  \theta_1, \theta_2,  \dots, \theta_{  \mathfrak{d}_1  } )}(0,x)
      - 
      \phi(x) 
    |^2
    +
    |
      \cfadd{def:nets,NinCP}{\mathcal N}^{l_0^2,l_1^2,\dots,l_{L_2}^2}_{\act,(\theta_{\mathfrak{d}_1+1},    \theta_{\mathfrak{d}_1+2},\dots,   \theta_{  \mathfrak{d}_1 +  \mathfrak{d}_2 })}(0) 
      - 
      \psi
    |^2
  \Bigr),
\end{split}
\eeq
 \cfload, \cfclear
  let $ \algoG\colon \R^{ \mathfrak{d}_1+\mathfrak{d}_2 } \times [0,T]\times\R \to \R^{ \mathfrak{d}_1+\mathfrak{d}_2}$ 
satisfy for all $t\in [0,T]$, $x\in\R$,  $\theta  \in \{ \vartheta\in \R^{ \mathfrak{d}_1+\mathfrak{d}_2}\colon   \mathscr{L}( \cdot,t,x )  \text{  is  differentiable at }  \vartheta \}$
that
     $ \algoG( t,\theta,x,y)=( \nabla_{ \theta } \mathscr{L} )(t,\theta,x,y)$, 
      let $(\Omega,\mathcal F,\PP)$ be a probability space,
  let  $M, \algoM \in \N$, 
     let  $\trandom_{ m}\colon \Omega\to D$,  $m \in\{1,2,\dots, M\}$, 
   be  i.i.d.\  random variables,     let $\xrandom_{m} \colon \Omega\to \partial D $,  $m \in\{1,2,\dots, M\}$, be  i.i.d.\  random variables, 
   for every $n\in \N$, $m\in\{1,2,\dots, \algoM\}$   let $ \xi_{ n, m } \colon\Omega \to \{ 1, 2, \dots, M \}$ be  a random variable,
 assume that       $( \trandom_{ m})_{m\in  \{1,2,\dots, M\}}$,    $( \xrandom_{ m})_{m\in  \{1,2,\dots, M\}}$,   and     $ (\xi_{ n, m })_{(n,m)\in \N\times\{1,2,\dots, \algoM\}}$   
    are  independent,
let $(\gamma_n )_{n \in \N} \subseteq \R$, 
 and let $\Theta\colon \N_0 \times \Omega \to \R^{ \mathfrak{d}_1+\mathfrak{d}_2 }$ satisfy for all $n \in \N_0$ that 
\begin{equation}
\Theta_{ n + 1 } = \Theta_n - \frac{ \gamma_n }{ \algoM }\left[  \sum_{ m = 1 }^{\algoM}\algoG( \Theta_{ n - 1 }, \trandom_{ \xi_{ n, m } },\xrandom_{ \xi_{ n, m } }, )\right].
\end{equation}
}
\end{algo}

\begin{remark}[Explanations for \cref{algo:stefan}]
  Loosely speaking, 
  in \cref{algo:stefan} we think for sufficiently large 
    $N \in \N$  
  for all
      $\theta = ( \theta_1,  \dots,  \theta_{  \mathfrak{d}_1 +  \mathfrak{d}_2 } ) \colon \Omega \to \R^{  \mathfrak{d}_1 +  \mathfrak{d}_2 }$
  with $\theta = \Theta_N$
  of the \ANN\ realizations 
  $
  \cfadd{def:nets,NinCP}{\mathcal N}^{l_0^1,l_1^1,\dots,l_{L_1}^1}_{\act,(  \theta_1, \theta_2,  \dots, \theta_{  \mathfrak{d}_1  } )}\big|_{ [0,T] \times \R }
  $
  and
  $
  \cfadd{def:nets,NinCP}{\mathcal N}^{l_0^2,l_1^2,\dots,l_{L_2}^2}_{\act,(\theta_{\mathfrak{d}_1+1},    \theta_{\mathfrak{d}_1+2},\dots,   \theta_{  \mathfrak{d}_1 +  \mathfrak{d}_2 } )        }\big|  _{ [0,T] }
  $
  as an approximation 
  \begin{equation}
  \label{T_B_D}
  \begin{split}
    \pr[\big]{
      \cfadd{def:nets,NinCP}{\mathcal N}^{l_0^1,l_1^1,\dots,l_{L_1}^1}_{\act,(  \theta_1, \theta_2,  \dots, \theta_{  \mathfrak{d}_1  } )}\big|_{ [0,T] \times \R }
      ,\,
      \cfadd{def:nets,NinCP}{\mathcal N}^{l_0^2,l_1^2,\dots,l_{L_2}^2}_{\act,(\theta_{\mathfrak{d}_1+1},    \theta_{\mathfrak{d}_1+2},\dots,   \theta_{  \mathfrak{d}_1 +  \mathfrak{d}_2 } )        }\big|  _{ [0,T] }
    }
    \approx 
    (
      u_1
      ,
      u_2
    )
  \end{split}
  \end{equation}
  of the solution $u \in V$ of the Stefan problem in \cref{stefan:eq1,stefan:eq2}.
\end{remark}

\section{Machine learning approximation methods for PDEs  based on stochastic Feynman-Kac-type representations}
\label{sec:FCformulations}

In this section we discuss machine learning approximation methods for \PDEs\ which employ Feynman-Kac-type formulas in their derivation.
We first introduce in \cref{sect:reformulationII} a refinement of the abstract reformulation result in \cref{sect:reformulationI} involving conditional expectations.
We then recall linear and nonlinear Feynman-Kac formulas in \cref{sect:FeynmanKac} upon which the machine learning methods discussed in this section are based.
Thereafter, in \cref{sect:DKM} we discuss methods which rely on linear Feynman-Kac formulas and in \cref{sec23} we discuss methods which rely on nonlinear Feynman-Kac formulas.

\subsection{Basic reformulation result for machine learning methods for PDEs II}
\label{sect:reformulationII}

As discussed in \cref{sect:reformulationI}, machine learning methods for \PDEs\ generally involve the transformation of \PDE\ problems into stochastic optimization problems.
In \cref{lemmaBcor} in \cref{sect:reformulationI} we introduced a basic reformulation result which can be used to derive such transformations.
\cref{lemmaBcor} is however not sufficient for some of the transformations considered in \cref{sec:FCformulations} which involve Feynman-Kac-type formulas.
It is the subject of this section to introduce a refined abstract reformulation result in \cref{lemmaB2} below which involves conditional expectations and is suitable for the derivation of some of the machine learning methods for \PDEs\ based on Feynman-Kac-type formulas.
To prove \cref{lemmaB2} we need two elemantary and well-known results presented in \cref{prop31,prop31copy} below (cf., e.g., \cite[Theorem~8.20]{MR4201399}).
Furthermore, in \cref{lemmaBcor2} below we demonstrate how the reformulation result in \cref{lemmaBcor} above can be derived as a consequence of the more general result in \cref{lemmaB2} below.

\begin{proposition}\label{prop31}
{%
 Let $(\Omega, \F,\P)$  be a probability space,    let $\mathcal{G} \subseteq \F$ be a 
sigma-algebra,   let $N\in\N$,  and let 
$X \colon \Omega\to \R^N$ be a random variable with
$\E[\|X\|^2]<\infty$. Then 
\begin{enumerate}[(i)]
\item\label{prop31item1} it holds for all $\mathcal{G}$-measurable $Y\colon \Omega\to \R^N$ that
 \beq\label{prop31item1100}
  \E\big[\|X-\E[X|\mathcal{G}]\|^2\bigr]\leq \E\big[\|X-\E[X|\mathcal{G}]\|^2\bigr]+\E\big[\|\E[X|\mathcal{G}]-Y\|^2\bigr]=
   \E\bigl[\|X-Y\|^2\bigl],
   \eeq
\item\label{prop31copyitem2}  it holds for all $\mathcal{G}$-measurable $Y\colon \Omega\to \R^N$
with $\P(Y= \E[ X | \mathcal{G}]    )<1$
   that
   \beq
   \E\big[\|X-\E[X|\mathcal{G}]\|^2\bigr]<\E\bigl[\|X-Y\|^2\bigl],
   \eeq
   and
   \item\label{prop31item3}  it holds that 
\beq\begin{split}
\E\big[\|X-\E[X|\mathcal{G}]\|^2\bigr]&=\inf_{{\substack{Y\colon \Omega\to \R^N \\ 
Y \text{ is }\mathcal{G}\text{-measurable } }}} \E\bigl[\|X-Y\|^2\bigl].
\end{split}\eeq
\end{enumerate}
}
\end{proposition}
\begin{cproof}{prop31}
Throughout this proof 
let $e_1, e_2, \dots, e_N \in \R^N$ satisfy $ e_1 = ( 1, 0, \dots, 0 ) $, $ e_2 = ( 0, 1, 0, \dots, 0 ) $, $ \dots $, $ e_N = ( 0, \dots, 0, 1 ) $.
\Nobs\  
Jensen's inequality for conditional expectations
(see,  \ex  
Klenke \cite[Theorem~8.20]{MR4201399}),   the tower property for conditional expectations,   and the assumption that   $\E[\|X\|^2]<\infty$ \prove\ 
that
\beq
	\begin{split}
	&\E\bigl[\|\E[ X | \mathcal{G}]\|^2\bigr]
	=\E\!\left[ \textstyle \sum\limits_{ n = 1 }^N |  \langle e_n,   \E[ X | \mathcal{G}] \rangle |^2 \right] 
	=\E\!\left[\textstyle \sum\limits_{ n = 1 }^N \big| \E \bigl[ \langle e_n,   X\rangle| \mathcal{G} \bigr]\big|^2  \right] \\
	&\leq   \E\!\left[\textstyle \sum\limits_{ n = 1 }^N \E \bigl[ | \langle e_n,   X\rangle |^2 | \mathcal{G} \bigr]\right]  \leq  \E\bigl[|\E[ \|X\| | \mathcal{G}]|^2\bigr]
	= \E [ \E[ \| X \|^2 | \mathcal{G} ] ] 
	= \E\bigl[ \|X\|^2 \bigr] <\infty . \end{split}\eeq
	The tower property for conditional expectations  \hence  \proves\  
	that for all  $\mathcal{G}$-measurable $Y\colon \Omega\to \R^N$  it holds that
	\beq
	\begin{split}
	&\E[ \| X - Y \|^2 ] = \E[ \| ( X - \E[ X | \mathcal{G}] ) + (  \E[ X | \mathcal{G}] - Y ) \|^2 ]\\
	&= \E\bigl[ \| X - \E[ X | \mathcal{G}]\|^2 + 2\langle X - \E[ X | \mathcal{G}],  \E[ X | \mathcal{G}] - Y  \rangle + \| \E[ X | \mathcal{G}] - Y\|^2 \bigr]\\
&= \E\bigl[\|X - \E[ X | \mathcal{G}]\|^2\bigr] + 2\,  \E\bigl[\langle X - \E[ X | \mathcal{G}],  \E[ X | \mathcal{G}] - Y  \rangle \bigr] + \E\bigl[ \| \E[ X | \mathcal{G}] - Y\|^2 \bigr]\\
&=\E\bigl[\|X - \E[ X | \mathcal{G}]\|^2\bigr] + 2\, \E\Bigl[ \E\bigl[\langle X - \E[ X | \mathcal{G}],  \E[ X | \mathcal{G}] - Y  \rangle\big|\mathcal{G} \bigr] \Bigr]+ \E\bigl[ \| \E[ X | \mathcal{G}] - Y\|^2 \bigr].
\end{split}
\eeq
\Hence that 
\beq
\begin{split}
&\E[ \| X - Y \|^2 ] \\
&=\E\bigl[\|X - \E[ X | \mathcal{G}]\|^2\bigr] +
2\, \E\bigl[ \langle \E[ X - \E[X|\mathcal{G}] | \mathcal{G}] , \E[ X | \mathcal{G} ] - Y \rangle \bigr]
+ \E\bigl[ \| \E[ X | \mathcal{G}] - Y\|^2 \bigr]\\
&=\E\bigl[\|X - \E[ X | \mathcal{G}]\|^2\bigr] +
2\, \E\bigl[ \langle  \E[X|\mathcal{G}] - \E[X|\mathcal{G}] , \E[ X | \mathcal{G} ] - Y \rangle \bigr]
+ \E\bigl[ \| \E[ X | \mathcal{G}] - Y\|^2 \bigr]\\
&=\E\bigl[\|X - \E[ X | \mathcal{G}]\|^2\bigr] +\E\bigl[ \| \E[ X | \mathcal{G}] - Y\|^2 \bigr].
\end{split}
	\eeq
This \proves\  \cref{prop31item1}.
\Nobs\  \cref{prop31item1} \proves\ \cref{prop31item3,prop31copyitem2}.
\end{cproof}

\begin{corollary}\label{prop31copy}
{%
 Let $(\Omega, \F,\P)$  be a probability space,    let $\mathcal{G} \subseteq \F$ be a 
sigma-algebra,   let $N\in\N$,  and let 
$X = ( X_1, \dots, X_N ) \colon \Omega\to \R^N$ be a random variable with
$\E[\|X\|^2]<\infty$. Then 
\beq\label{Cor32mainresult}\begin{split}
\E\big[\|X-\E[X|\mathcal{G}]\|^2\bigr]= \sum_{ n = 1 }^N \left(\inf_{\substack{ Y \colon \Omega \to \R \\ Y \text{is } \mathcal{G}\text{-measurable}} } \E\bigl[ | X_n-Y |^2 \bigr]\right).
\end{split}\eeq
}
\end{corollary}
\begin{cproof}{prop31copy}
\Nobs\   \cref{prop31} (applied for every $n \in \{ 1, 2, \dots, N \}$ with $N\with 1  $, $X\with X_n$   in the notation of \cref{prop31}) \proves\ that for all $n \in \{ 1, 2, \dots, N \}$  it holds that
\begin{equation}
\E[ | X_n - \E[ X_n | \mathcal{G} ] |^2 ] =\inf_{\substack{ Y \colon \Omega \to \R \\ Y \text{is } \mathcal{G}\text{-measurable}} } \E\bigl[ | X_n-Y |^2\big].
\end{equation}
Combining this with  the fact that
$\E\big[\|X-\E[X|\mathcal{G}]\|^2\bigr]=\sum_{ n = 1 }^N  \E[ | X_n - \E[ X_n | \mathcal{G} ] |^2 ]        $
\proves\ \cref {Cor32mainresult}.
\end{cproof}

\begin{theorem}[Basic reformulation result for machine learning methods for \PDEs\ II]\label{lemmaB2}
 {%

 Let $\basicset$ and $\solutionset$  be non-empty sets with $\solutionset  \subseteq \basicset$,  let $(\Omega, \F,\P)$  be a probability space,   
 for every $k \in \{ 1, 2 \}$  let $( \mathfrak{R}_k, \mathcal{R}_k )$ be a measurable space,  for every $k \in \{ 1, 2 \} $ let $R_k \colon \Omega \to  \mathfrak{R}_k $ be a random variable,  let $N \in \N$, 
 for every $v \in \basicset$ let $U_v \colon   \mathfrak{R}_1   \to \R^N$ be  measurable,
 let $\phi\colon  \mathfrak{R}_1\times  \mathfrak{R}_2\to\R^N $ be measurable, 
 assume $\E[ \| \phi( R_1,  R_2)\|^2 ] < \infty$, 
 assume for all $v \in \basicset $ 
  that
 \beq\label{ifandonlyifinS} \P( U_v(  R_1 )=\E[ \phi( R_1,  R_2)  |  R_1] ) = 1\eeq
 if and only if $v \in \solutionset$,
 and let $ \lossexp \colon V \to [0,\infty]$  satisfy for all $v \in \basicset $ that
 \beq\label{lvdef108}  \lossexp(v) = \E[ \| \phi(  R_1,  R_2 ) - U_v( R_1 )\|^2 ].\eeq
 Then
 \begin{enumerate}[(i)]
 \item \label{lemmaB2anotheritem1}  
 for all $v \in \basicset$  it holds that
  \beq \lossexp(v) = \inf_{w \in \basicset}\lossexp(w)  \eeq
if and only if
 $ v \in \solutionset   $
and
 \item \label{lemmaB2anotheritem2}  
  there exists $v \in \basicset$ such that
\beq\lossexp(v) = \inf_{w \in \basicset}\lossexp(w).\eeq
  \end{enumerate}
 }
 \end{theorem}

\begin{cproof}{lemmaB2}
Throughout this proof  let $\mathcal{G}\subseteq \mathcal{F}$ be the sigma-algebra generated by $R_1$. 
\Nobs\  \cref{lvdef108}   \proves\  that
for all $v,w \in \basicset$ with 
$ \P(U_{v}(  R_1 )=\E[ \phi( R_1,  R_2)  |  R_1])= \P(U_{w}(  R_1 )=\E[ \phi( R_1,  R_2)  |  R_1])=1$ it holds
that
\beq\label{lv1eqlv2}\begin{split}  \lossexp(v) &=\E[ \| \phi(R_1,R_2) - U_v(  R_1 )\|^2 ]= \E[ \| \phi(R_1,R_2 ) - \E[\phi(  R_1, R_2 ) | R_1 ] \|^2 ]\\& = \E[ \| \phi(R_1,R_2) - U_w(  R_1 )\|^2 ] = \lossexp(w).\end{split}\eeq
\Moreover\   \cref{prop31copyitem2} in   \cref{prop31} (applied for every $w\in \basicset$ with $(\Omega, \F,\P)\with (\Omega, \F,\P)$,  $\mathcal{G}\with \mathcal{G}$,    $N\with N$,  $X\with
( \Omega \ni \omega \mapsto \phi(R_1(\omega),R_2(\omega)) \in \R^N )$, $Y\with ( \Omega \ni \omega \mapsto U_w(R_1(\omega))\in\R^N)$ in the notation of  \cref{prop31})  and \cref{lvdef108}
\prove\  that for all $v,w\in \basicset$  with $ \P(U_v(  R_1 )=\E[ \phi( R_1,  R_2)  |  R_1])=1$ and
$ \P(U_w(  R_1 )=\E[ \phi( R_1,  R_2)  |  R_1])<1$
it holds that  
\beq\begin{split}\label{lvsmallerthanlw} \lossexp(v)& =\E[ \| \phi(R_1,R_2) - U_v(  R_1 )\|^2 ]= \E[ \| \phi(R_1,R_2 ) - \E[\phi(  R_1, R_2 ) | R_1 ] \|^2 ] \\
&< \E[ \| \phi(R_1,R_2) - U_w(  R_1 )\|^2 ] = \lossexp(w).\end{split}\eeq
Combining this and \cref{lv1eqlv2} \proves\ that
for all $v, w \in \basicset$
with $ \P(U_v(  R_1 )=\E[ \phi( R_1,  R_2)  |  R_1])=1$
  it holds that
 \beq \lossexp(v) \leq \lossexp(w).\eeq
 \Hence that for all  $v \in \basicset$
with $ \P(U_v(  R_1 )=\E[ \phi( R_1,  R_2)  |  R_1])=1$  it holds  that
\beq\label{clossexp115}  \lossexp(v) =\inf_{w \in \basicset} \lossexp(w).\eeq
 Combining this with \cref{ifandonlyifinS} \proves\  that for all $v \in \solutionset$ it holds that 
\begin{equation}\label{allvinS116}
  \lossexp(v) = \inf_{w\in\basicset} \lossexp(w).
  \end{equation}
\Moreover\   \cref{lvsmallerthanlw} and \cref{clossexp115} \prove\ that
 for all  $v, w \in \basicset$ with
   $ \P(U_v(  R_1 )=\E[ \phi( R_1,  R_2)  |  R_1])=1$
and
$ \P(U_w(  R_1 )=\E[ \phi( R_1,  R_2)  |  R_1])<1$ 
   it holds that
  \beq\label{losslveqc}\inf_{z\in\basicset} \lossexp(z) = \lossexp(v) < \lossexp(w).\eeq
  \Moreover\  \cref{ifandonlyifinS} and the assumption that $S \neq \emptyset$  \prove\  that there exists $v \in \basicset$ such that $ \P(U_v(  R_1 )=\E[ \phi( R_1,  R_2)  |  R_1])=1$.
Combining this with \cref{losslveqc}  \proves\ that for all $w\in\basicset$  with $ \P(U_w(  R_1 )=\E[ \phi( R_1,  R_2)  |  R_1])<1$  it holds that 
 \beq\label{losslveqc2}\inf_{z\in\basicset} \lossexp(z) < \lossexp(w).\eeq
 \Hence that for all $v \in \basicset$ with  $\lossexp(v) =\inf_{z\in\basicset} \lossexp(z)$
 it holds that 
   $ \P(U_v(  R_1 )=\E[ \phi( R_1,  R_2)  |  R_1])=1$.
   This and \cref{ifandonlyifinS} \prove\ that 
   for all $v\in\basicset$ with   $\lossexp(v) =\inf_{z\in\basicset} \lossexp(z)$ it holds that $v \in \solutionset$.
   Combining this with  \cref{allvinS116}
\proves\  \cref{lemmaB2anotheritem1}.  
\Nobs\   \cref{lemmaB2anotheritem1} and 
 the assumption that $S \neq \emptyset$  \prove\  \cref{lemmaB2anotheritem2}. 
\end{cproof}

Note that the next result is identical to \cref{lemmaBcor} in \cref{sect:reformulationI}.

 \begin{corollary}[Basic reformulation result for machine learning methods for \PDEs\ I]\label{lemmaBcor2}
 {%

 Let $\solutionset$ and $\basicset$ be non-empty sets with $\solutionset  \subseteq \basicset$,  let $(\Omega, \F,\P)$  be a probability space,   
   let $( \mathfrak{R}, \mathcal{R})$ be a measurable space,   let $R \colon \Omega \to  \mathfrak{R}$ be a random variable,  let $N\in \N$,  for every $v \in \basicset$
 let $U_v \colon   \mathfrak{R}  \to \R^N$ be  measurable,
 assume for all $v \in \basicset$ that  $\P( U_v(R ) = 0 ) = 1 $ if and only if $v \in \solutionset $,   and let $\lossexp \colon \basicset \to [0,\infty]$ satisfy for all $v \in \basicset$ that $\lossexp(v) =
 \E[  \| U_v(R ) \|^2]$.
 Then
 \begin{enumerate}[(i)]
 \item 
 for all $v \in \basicset$  it holds that
  \beq \lossexp(v) = \inf_{w \in \basicset}\lossexp(w)  \eeq
if and only if
 $ v \in \solutionset   $
and
 \item
  there exists $v \in \basicset$ such that
\beq \lossexp(v) = \inf_{w \in \basicset}\lossexp(w).\eeq
  \end{enumerate}
 }
 \end{corollary}
\begin{cproof}{lemmaBcor2}
\Nobs\   \cref{lemmaB2} (applied with $\basicset\with\basicset$,  $\solutionset\with\solutionset$, 
$( \mathfrak{R}_1, \mathcal{R}_1 )\with ( \mathfrak{R}, \mathcal{R} )$,  $( \mathfrak{R}_2, \mathcal{R}_2)\with ( \mathfrak{R}, \mathcal{R} )$,  $R_1\with R$,  $R_2\with R$,
 $N\with N$,   $(U_v)_{v\in\basicset}\with  (U_v)_{v\in\basicset}$,
  $\phi\with ( \mathfrak{R}\times  \mathfrak{R}\ni (x,y)\mapsto 0\in \R^N)$
 in the notation of  \cref{lemmaB2})
\proves\  
 \cref{lemmaBcoritemi,lemmaBcoritemii}.
\end{cproof}

\subsection[Stochastic representations (Feynman-Kac formulas) for PDEs]{Linear and nonlinear stochastic representations (Feynman-Kac formulas) for PDEs}
\label{sect:FeynmanKac}

In this section we recall well-known Feynman-Kac formulas for \PDEs.
We first prove a nonlinear Feynman-Kac formula for semilinear parabolic \PDEs\ in \cref{nonlinearFKformula} and thereafter obtain as a consequence linear Feynman-Kac formulas for heat \PDEs\ in \cref{cor36,noBMversion,notimeFKformula}.

 \begin{proposition}[Nonlinear Feynman-Kac Formula]\label{nonlinearFKformula}
 {%
  Let $T\in (0,\infty)$, $d\in \N$,
		let $\mu\colon[0,T]\times \R^d \rightarrow \R^d $,  
		$\sigma\colon[0,T]\times \R^d\rightarrow \R^{d\times d} $,  and $f\colon [0,T]\times \R^d \times \R\times\R^d  \rightarrow \R$ be  continuous,  let
	$u\in C^{1,2}([0,T]\times \R^d)$ satisfy\footnote{\Nobs\  for every $d \in \N$ and every $d \times d$-matrix $ M \in \R^{ d \times d }$ 
it holds that $M^*$ is the transpose of $M$. } for all $t\in [0,T]$, $x\in \R^d$ that
	\begin{multline}\label{thm22pde0}
	\tfrac{\partial u}{\partial t}(t,x)+\tfrac12 {\rm{Tr}}\bigl((\sigma(t,x)[\sigma(t,x)]^*({\rm{Hess}}_xu)(t,x) \bigr)     +\tfrac{\partial u}{\partial x} (t,x) \, \mu(t,x)\\
+f\bigl(t,x,u(t,x), [ \sigma( t, x ) ]^*( \nabla_x u )( t, x )\bigr)    =0,
	\end{multline}
	let $(\Omega,\F, \P)$ be a probability space with a normal\footnote{\Nobs\ for every $T \in (0,\infty)$ and every filtered probability space $(\Omega,\F, (\mathbb{F}_t)_{t\in [0,T]},\P)$ it holds that $(\mathbb{F}_t)_{t\in [0,T]}$ is a normal filtration on $(\Omega,\F, \P)$ if and only if it holds for all $t \in  [0,T)$  that $ \{ A\in \mathcal{F}\colon \P(A)=0 \} \subseteq\mathbb{F}_t =(\cap_{s \in (t,T]} \mathbb{F}_s)$  (cf., \ex  Pr\'{e}v\^{o}t \& R\"ockner~\cite[Definition
2.1.11]{MR2329435}).}  filtration $(\mathbb{F}_t)_{t\in [0,T]}$,  let 
		$\BM\colon [0,T]\times\Omega \to \R^d$ be a standard $(\mathbb{F}_t)_{t\in [0,T]}$-Brownian motion,
let $X\colon [0,T]\times \Omega \to \R^d$ be an  $(\mathbb{F}_t)_{t \in [0,T]}$-adapted stochastic process   with continuous sample paths which satisfies that for all $t\in [0,T]$  it holds $\P$-a.s.\  that 
\beq\label{bsdext0}
 X_t=X_0+\int_{0}^t\mu(s, X_s)\,\mathrm{d}s+\int_{0}^t\sigma(s, X_s)\,\mathrm{d}\BM_s,
\eeq 
and 
let  $Y\colon [0,T]\times \Omega \to \R$ and $Z\colon [0,T]\times \Omega \to \R^d$ be   $(\mathbb{F}_t)_{t \in [0,T]}$-adapted stochastic processes with continuous sample paths which satisfy for all $t\in [0,T]$  that 
$Y_t=u(t,X_t) $ and  $Z_t=[\sigma(t,X_t)]^*  ( \nabla_x u )(t,X_t) $.
Then  for all $t \in [0,T]$,  $s\in [0,t]$  it holds $\P$-a.s.\ that
  \beq\label{FBSDE0forward0}
\begin{split}
 Y_t-Y_s+\int_{s}^t f(r,X_r, Y_r, Z_r)\,\mathrm{d}r-\int_{s}^t (Z_r)^* \,\mathrm{d}\BM_r =0.
 \end{split}
\eeq
  
 } 
  \end{proposition}
  \begin{cproof}{nonlinearFKformula}
\Nobs\  \cref{bsdext0} and It\^{o}'s formula (cf., \ex  Klenke~\cite[Theorem~25.27]{MR4201399})
  \prove\
 that for all   $t \in [0,T]$,  $s\in [0,t]$  it holds $\P$-a.s.\ that
 \beq\begin{split}
  Y_t &= u(t,X_t)\\
  &=u(s,X_s)+\int_s^t\tfrac{ \partial u}{\partial r}(r,X_r)
  +\tfrac{\partial u}{\partial x} (r,X_r) \, \mu(r,X_r)\, \mathrm{d}r\\
   &\quad  +\int_s^t \tfrac{\partial u}{\partial x} (r,X_r)  \,  \sigma(r, X_r)\,\mathrm{d}\BM_r
 +\tfrac12 \int_s^t {\rm{Tr}}\bigl((\sigma(r,X_r)[\sigma(r,X_r)]^*({\rm{Hess}}_xu)(r,X_r) \bigr)\, \mathrm{d}r.
\end{split} \eeq
Combining this with \cref{thm22pde0} \proves\  that
 for all $t \in [0,T]$,  $s\in [0,t]$ it holds $\P$-a.s.\
that
  \beq\begin{split}
  Y_t &=u(s,X_s)-\int_s^t  \big[\tfrac12 {\rm{Tr}}\bigl((\sigma(r,X_r)[\sigma(r,X_r)]^*({\rm{Hess}}_xu)(r,X_r) \bigr)     +\tfrac{\partial u}{\partial x} (r,X_r) \, \mu(r,X_r)\\
&\quad+f\bigl(r,X_r,u(r,X_r), [ \sigma( r, X_r ) ]^*( \nabla_x u )( r, X_r )\bigr)\big]\, \mathrm{d}r
+\int_s^t \tfrac{\partial u}{\partial x} (r,X_r) \,   \mu(r, X_r)\,\mathrm{d}r\\
  &\quad  +\int_s^t \tfrac{\partial u}{\partial x} (r,X_r)  \,  \sigma(r, X_r)\,\mathrm{d}\BM_r +\tfrac12 \int_s^t {\rm{Tr}}\bigl((\sigma(r,X_r)[\sigma(r,X_r)]^*({\rm{Hess}}_xu)(r,X_r) \bigr)\, \mathrm{d}r\\
&=Y_s-\int_s^tf\bigl(r,X_r,u(r,X_r), [ \sigma( r, X_r) ]^*( \nabla_x u )( r, X_r )\bigr)\, \mathrm{d}r+\int_s^t \tfrac{\partial u}{\partial x} (r,X_r) \, \sigma(r,X_r) \, \mathrm{d}\BM_r\\
&=Y_s-\int_s^tf\bigl(r,X_r,Y_r,Z_r\bigr)\, \mathrm{d}r+\int_s^t(Z_r)^* \, \mathrm{d}\BM_r.
\end{split}
  \eeq
  This \proves\  \cref{FBSDE0forward0}.
  \end{cproof}

\begin{lemma}[Right continuity of the induced filtration]
\label{lem:properties_filtration}
{%
  Let $T\in (0,\infty)$,  let $ \left( \Omega, \mathcal{F}, \P \right) $
be a probability space, let
$ (\mathbb{F}_t)_{t\in [0,T]} $
be a filtration on $ ( \Omega, \mathcal{F} , \P) $,  and 
for every $t \in [0,T]$ 
let  $\mathbb{G}_t \subseteq \F $  satisfy 
$\mathbb{G}_t = 
    \cap_{ s \in  (t,T]\cup\{T\} } \,
   \mathbb{F}_s$.
Then 
\begin{enumerate}[(i)]
\item 
\label{item:filtration_minus_plus}
it holds for all $ t \in [0,T] $ that
$
  \mathbb{F}_t
  \subseteq
 \mathbb{G}_t
$,
\item
\label{item:filtration_plus}
it holds for all $ s, t \in [0,T] $
with 
$
  s < t
$
that
$
  \mathbb{F}_s
  \subseteq
 \mathbb{G}_s
  \subseteq
  \mathbb{F}_t
  \subseteq
   \mathbb{G}_t
$,   and 
\item\label{lem:further_properties_filtration}
 it holds for all $ t \in [0,T] $ that
\begin{equation}
 \mathbb{G}_t
  =\cap_{ s \in  (t,T]\cup\{T\}  } \,
   \mathbb{G}_s.
\end{equation}
\end{enumerate}}
\end{lemma}
\begin{cproof}{lem:properties_filtration}
\Nobs\ the fact that for all $t\in[0,T]$ it holds that  
$\mathbb{G}_t=(\cap_{ s \in  (t,T]\cup\{T\} } \,  \mathbb{F}_s)\supseteq \mathbb{F}_t$  \proves\
\cref{item:filtration_minus_plus}.
\Nobs\  the assumption that for all $t \in [0,T]$ it holds that $\mathbb{G}_t =  \cap_{ s \in  (t,T]\cup\{T\} } \,
  \mathbb{F}_s$ \proves\  that for all $t,r \in [0,T]$ with $t<r$  it holds that
   \beq 
  \mathbb{G}_t = 
  (  \cap_{ s \in  (t,T]\cup\{T\} } \,
   \mathbb{F}_s )
   =(\cap_{ s \in (t,T] } \mathbb{F}_s)
   \subseteq ( \cap_{ s \in (t,r] } \, \mathbb{F}_s ) \subseteq \mathbb{F}_r .
  \eeq
  This and \cref{item:filtration_minus_plus} \prove\  \cref{item:filtration_plus}.
\Nobs\  \cref{item:filtration_plus} and the fact that $\mathbb{G}_T=\mathbb{F}_T$ 
 \prove\  that for all $ t \in [0,T] $,
$ r \in (t,T]\cup\{T\}$
it holds that
\begin{equation}
  (
    \cap_{ s \in (t, T]\cup\{T\} } \,
  \mathbb{G}_s
  )\subseteq( \cap_{ s \in ( t, ( r + t ) / 2 ] \cup \{ T \} } \, \mathbb{G}_s )
  \subseteq \mathbb{G}_{ ( r + t ) / 2 }
  \subseteq
  \mathbb{F}_r
  .
\end{equation}
Combining this with \cref{item:filtration_minus_plus}
\proves\ that for all $ t \in [0,T] $ it holds 
that
\begin{equation}
\label{eq:further_properties_B}
( \cap_{ s \in (t,T] \cup \{ T \} }\,\mathbb{F}_s )
 \subseteq   (   \cap_{ s \in  (t,T]\cup\{T\} } \,
  \mathbb{G}_s
  )
  \subseteq
  (
    \cap_{ r \in  (t,T]\cup\{T\} } \,
    \mathbb{F}_r
  )
  =
  \mathbb{G}_t
  .
\end{equation}
\end{cproof}

\begin{proposition}[Construction of a normal filtration]
\label{prop:construction_stochastic_basis}
{%
 Let $T\in(0,\infty)$,  let $ \left( \Omega, \mathcal{F},\P \right) $
be a probability space,  let
$ (\mathbb{F}_t)_{t\in [0,T]} $
be a filtration on $ ( \Omega, \mathcal{F} , \P) $,
for every $ t \in [0,T] $
 let $ \mathbb{G}_t\subseteq \mathcal{F}$ be  the sigma-algebra generated by  $\mathbb{F}_t\cup \{ A \in \mathcal{F} \colon
      \P( A ) = 0
    \}$,
and  for every $ t \in [0,T] $ 
 let $ \mathbb{H}_t\subseteq \mathcal{F}$ satisfy
$\mathbb{H}_t= \cap_{ s \in  (t,T]\cup\{T\} } \,
   \mathbb{G}_s$.
Then 
\begin{enumerate}[(i)]
\item\label{Gtnormalfiltration}
it holds that $ (\mathbb{H}_t)_{t\in [0,T]} $ is a normal filtration on 
$ ( \Omega, \mathcal{F} , \P) $ and 
\item  \label{item:smallest_basis}
it holds 
for all 
normal filtrations 
$ ( \mathbb{I}_t )_{ t \in [0,T] } $
on
$
  ( \Omega, \mathcal{F}, \P )
$
with 
$
  \forall \, t \in [0,T] \colon
  \mathbb{F}_t \subseteq \mathbb{I}_t
$
that
$ 
  \forall \, 
  t \in [0,T] 
  \colon
  \mathbb{H}_t
  \subseteq
  \mathbb{I}_t
$.
\end{enumerate}

}
\end{proposition}

\begin{cproof}{prop:construction_stochastic_basis}
First, \nobs  that  the assumption that
for all $ t \in [0,T] $ it holds that 
 $ \mathbb{G}_t\subseteq \mathcal{F}$ is  the sigma-algebra generated by  $\mathbb{F}_t\cup \{ A \in \mathcal{F} \colon
      \P( A ) = 0
    \}$
\proves\  that 
for all $ t \in [0,T] $
it holds that
\begin{equation}
    \left\{ 
      A \in \mathcal{F} \colon
      \P( A ) = 0
    \right\}
  \subseteq
  \mathbb{G}_t 
  .
\end{equation}
\Cref{item:filtration_minus_plus} in \cref{lem:properties_filtration}
(applied with
$(\mathbb{F}_t)_{t\in [0,T]}\with (\mathbb{G}_t)_{t\in [0,T]} $,  $ (\mathbb{G}_t)_{t\in [0,T]}\with (\mathbb{H}_t)_{t\in [0,T]} $
  in the notation of  \cref{lem:properties_filtration})
  \hence \proves\ that
for all $ t \in [0,T] $
it holds that
\begin{equation}
    \left\{ 
      A \in \mathcal{F} \colon
      \P( A ) = 0
    \right\}
  \subseteq
  \mathbb{H}_t
  .
\end{equation}
Combining this with \cref{item:filtration_plus,lem:further_properties_filtration} in \cref{lem:properties_filtration}
\proves\  \cref{Gtnormalfiltration}.
\Nobs\
for all 
normal filtrations 
$ ( \mathbb{I}_t )_{ t \in [0,T] } $  with 
$
  \forall \, t \in [0,T] \colon
  \mathbb{F}_t \subseteq \mathbb{I}_t
$ it holds for all  $t \in [0,T] $ that 
\beq\label{129propertyofHt} \bigl(\mathbb{F}_t\cup    \left\{ 
      A \in \mathcal{F} \colon
      \P( A ) = 0
    \right\}\bigr)
  \subseteq  \mathbb{I}_t  \qquad \text{and} \qquad 
   \mathbb{I}_t=\cap_{ s \in  (t,T]\cup\{T\} } \,
  \mathbb{I}_s.
   \eeq 
  This and the fact that
  for all $ t \in [0,T] $ it holds that
  $ \mathbb{G}_t\subseteq \mathcal{F}$ is  the sigma-algebra generated by   $\mathbb{F}_t\cup \{ A \in \mathcal{F} \colon
      \P( A ) = 0
    \}$
  \prove\   that 
  for all  normal filtrations 
$ ( \mathbb{I}_t )_{ t \in [0,T] } $
on
$
  ( \Omega, \mathcal{F}, \P )
$
with 
$
  \forall \, t \in [0,T] \colon
  \mathbb{F}_t \subseteq \mathbb{I}_t
$
 it holds that
$  \forall \, t \in [0,T] \colon
  \mathbb{G}_t
  \subseteq
  \mathbb{I}_t 
$.
Combining this with \cref{129propertyofHt} \proves\  that
for all 
normal filtrations 
$ ( \mathbb{I}_t )_{ t \in [0,T] } $
on
$
  ( \Omega, \mathcal{F}, \P )
$
with 
$
  \forall \, t \in [0,T] \colon
  \mathbb{F}_t \subseteq \mathbb{I}_t 
$
it holds that
\begin{equation}
  \forall \, t \in [0,T] \colon   \mathbb{H}_t=(\cap_{ s \in  (t,T]\cup\{T\} } \, \mathbb{G}_s)
  \subseteq (\cap_{ s \in  (t,T]\cup\{T\} } \,  \mathbb{I}_s)=
 \mathbb{I}_t.
\end{equation}
\end{cproof}

\begin{corollary}\label{cor36}
 {%
 Let $T,\rho \in (0,\infty)$, 
 $d\in \N$,  
 $ \varrho = \sqrt{ 2 \rho } $,
 let $u\in C^{1,2}([0,T]\times \R^d)$ 
	have at most polynomially growing partial derivatives,  assume  for all $t\in [0,T]$, $x\in \R^d$ that
	\beq
	\tfrac{\partial u}{\partial t}(t,x)= \rho \, \Delta_x u(t,x),
	\eeq
	let $(\Omega,\F, \P)$ be a probability space,  and let 
		$\BM\colon [0,T]\times\Omega \to \R^d$ be a standard Brownian motion.
Then  it holds  for all $t\in [0,T]$, $x\in\R^d$ that
  \beq\label{cor36mainresult}
\begin{split}
u(t,x)=\E\bigl[ u( 0, x+\varrho\BM_t  ) \bigr].
 \end{split}
\eeq
 } 
  \end{corollary}
  
\begin{cproof}{cor36}
 Throughout this proof let $(\mathbb{F}_t)_{t \in [0,T]}$
be  the normal filtration generated by $B$ (cf.,  \ex \cref{prop:construction_stochastic_basis}).
\cref{nonlinearFKformula} (applied for every $\mathcal{T}\in(0,T]$,  $x\in\R^d$    with $T\with \mathcal{T}$,  $d\with d$,  
$\mu\with  ([0,\mathcal{T}]\times\R^d\ni (t,v)\mapsto 0\in \R^d)$, 
$\sigma\with ([0,\mathcal{T}]\times \R^d\ni (t,v) \mapsto \varrho\I_d\in \R^{d\times d})$,  $f\with ( [0,\mathcal{T}]\times \R^d \times \R\times\R^d\ni (t,v,y,z)\mapsto 0\in\R)$,
$u\with  ([0,\mathcal{T}]\times\R^d\ni (t,v) \mapsto u(\mathcal{T}-t,v)\in\R)$, 
$(\Omega,\F, \P)\with (\Omega,\F, \P)$,
$(\mathbb{F}_t)_{t \in [0,T]}\with (\mathbb{F}_t)_{t \in [0,T]}$,
$X\with ([0,\mathcal{T}]\times \Omega\ni (t, \omega)\mapsto x+\varrho B_t(\omega)\in\R^d)$
in the notation of \cref{nonlinearFKformula}) \proves\     that   for all $\mathcal{T}\in(0,T]$,  $t\in [0,\mathcal{T}]$, $s\in [0,t]$,  $x\in\R^d$  it holds $\P$-a.s.\  that 
 \beq\label{cor36eq132}\begin{split}
  u(\mathcal{T}-t,x+\varrho \BM_t)=u(\mathcal{T}-s,x+\varrho \BM_s)-\int_s^t \varrho[( \nabla_x u )(r,x+\varrho \BM_r)]^* \, \mathrm{d}\BM_r.
\end{split}
  \eeq
  \Hence  that for all $\mathcal{T}\in (0,T]$, $x\in\R^d$ it holds $\P$-a.s.\  that
  \beq\begin{split}
  u(0,x+\varrho \BM_\mathcal{T})=u(\mathcal{T},x)-\int_0^\mathcal{T}\varrho[( \nabla_x u )(r,x+\varrho \BM_r)]^* \, \mathrm{d}\BM_r.
\end{split}
  \eeq
   \Hence  that for all $t\in [0,T]$, $x\in\R^d$ it holds $\P$-a.s.\  that
  \beq\label{cor36eq132a}\begin{split}
  u(0,x+\varrho \BM_t)=u(t,x)-\int_0^t \varrho[( \nabla_x u )(r,x+\varrho \BM_r)]^* \, \mathrm{d}\BM_r.
\end{split}
  \eeq
  \Moreover\  the assumption that  $u\in C^{1,2}([0,T]\times \R^d)$ 
	has at most polynomially growing partial derivatives
  \proves\  that  there exists $c\in (0,\infty)$ such that for all $r\in [0,T]$,   $x\in\R^d$ it holds   that 
  \beq\label{polygrowth2}
		\|( \nabla_x u )(r,x)\|\leq c(1+\|x\|^c).
		\eeq
		Combining this with the fact that  for all $c\in (0,\infty)$ it holds that
		$\sup_{r\in [0,T]}\E[\| \BM_r\|^c ]<\infty$  \proves\  that for all $t\in [0,T]$, $x\in\R^d$  it holds   that 
		\beq\label{integrabilitymartingale132}
\E\!\left[\int_0^t  \| \varrho ( \nabla_x u )(r,x+\varrho \BM_r)\|^2 \, \mathrm{d}r\right]<\infty.
\eeq
This \proves\  that   for all $t\in [0,T]$, $x\in \R^d$ it holds that
\beq\label{eq143expsmallerthaninfinity}
\E\!\left[\left|\int_0^t \varrho[( \nabla_x u )(r,x+\varrho \BM_r)]^* \, \mathrm{d}\BM_r \right|^2\right]<\infty.
\eeq
This and  \cref{cor36eq132a} \prove\ that  for all $t\in [0,T]$,   $x\in\R^d$ it holds that
\beq \E\bigl[ |u( 0, x+\varrho\BM_t  ) |\bigr]<\infty.\eeq
Combining this with \cref{cor36eq132a},  \cref{integrabilitymartingale132},  and \cref{eq143expsmallerthaninfinity}  \proves\ that
for all $t\in [0,T]$, $x\in \R^d$ it holds that 
\beq \E\bigl[ u( 0, x+\varrho\BM_t  ) \bigr]=u(t,x).\eeq
This \proves\ \cref{cor36mainresult}.
\end{cproof}

\begin{corollary}\label{noBMversion}
{%

Let $T,\rho \in (0,\infty)$, 
 $d\in \N$,  
 $ \varrho = \sqrt{ 2 \rho } $,
 let $u\in C^{1,2}([0,T]\times \R^d)$ 
	have at most polynomially growing partial derivatives, 
	let $\varphi\colon \R^d\to \R$ be a function,
 assume  for all $t\in [0,T]$, $x\in \R^d$ that $ u(0,x) = \varphi(x) $ and 
	\beq\label{eq124}
	\tfrac{\partial u}{\partial t}(t,x)= \rho \, \Delta_x u(t,x),
	\eeq
	let $(\Omega,\F, \P)$ be a probability space, 
	let
$ \mathbb{B}\colon \Omega \to \R^d$ 
be a  standard normal random variable,   
 let  $\xi  \colon \Omega \to \R^d $ and $\tau\colon \Omega\to[0,T]$ be  random variables,
assume that $ \mathbb{B} $,   $ \xi $,        and $\tau$  are independent,  and let $\mathcal{G}\subseteq \mathcal{F}$ be the sigma-algebra generated by $\tau$ and $\xi$.
Then it holds $\P$-a.s.\ 
that
 \beq\label{utauxi125}
\begin{split}
u(\tau,\xi)=\E\bigl[ \varphi( \varrho \sqrt{\tau}  \mathbb{B} + \xi )|\mathcal{G} \bigr].
 \end{split}
\eeq
}
\end{corollary}
\begin{cproof}{noBMversion}
Throughout this proof  
let $\probspacetilde$ be a probability space 
and let $\BM\colon [0,T]\times\tildeOmega \to \R^d$
be a standard Brownian motion.
\Nobs\  \cref{cor36}
(applied   with $T\with T$,  $d\with d$,  $\rho\with\rho$,   $u\with u$,  $(\Omega, \F, \P ) \with \probspacetilde$,  
 $\BM\with\BM$
 in the notation of  \cref{cor36}),
 the fact that  for all $x\in\R^d$ it holds that  $u(0,x) = \varphi(x) $,  and \cref{eq124} \prove\  that
 for all $t\in [0,T]$, $x\in \R^d$ it holds that
 \beq
\begin{split}
u(t,x)=\int_{\tildeOmega} u\bigl( 0, x+\varrho\BM_{t}(\omega)\bigr) \,\tilde{\P}(d\omega)=\int_{\tildeOmega} \varphi( x+\varrho\BM_{t}(\omega)) \,\tilde{\P}(d\omega).
 \end{split}
\eeq

Combining this and the fact that 
 for all $t \in [0, T]$,  $A \in \mathcal{B}( \R^d )$  it holds that $\tilde{\P}(   \BM_{t} \in A ) = \P( \sqrt{t} \mathbb{B} \in A )$ 
 \proves\  that  for all $t\in[0,T]$, $x\in\R^d$ it holds  that
 \beq\label{eq127}
\begin{split}
u(t,x)=
 \int_{ \tilde{\Omega} } \varphi( x+\varrho\BM_{t}(\omega)) \,\tilde{\P}(d\omega)= \int_{ \Omega }  \varphi( x+\sqrt{t} \mathbb{B}(\omega)) \,\P(d\omega)  =
\E\bigl[ \varphi(x+\varrho \sqrt{t} \mathbb{B} ) \bigr].
 \end{split}
\eeq
\Moreover\  the assumption that $u\in C^{1,2}([0,T]\times \R^d)$ and the assumption that $\mathcal{G}$ is the sigma-algebra generated by $\tau$ and $\xi$ \prove\  that 
$\Omega \ni \omega\mapsto u(\tau(\omega), \xi(\omega)) \in  \R$
 is $\mathcal{G}$/$\mathcal{B}(\R)$-measurable.
		\Moreover\  the factorization lemma for conditional expectations,    \ex in    \cite[Corollary 2.9]{jentzen2018strong} (applied with $(\Omega,\F, \P)\with (\Omega,\F, \P)$,   $\mathcal{G}\with \mathcal{G}$,
$(\mathbb{X},\mathcal{X})\with (\R^d,\mathcal{B}(\R^d))$,  $(\mathbb{Y},\mathcal{Y})\with ([0,T]\times\R^d,\mathcal{B}([0,T]\times\R^d))$, 
$X\with \mathbb{B}$,  $Y\with (\Omega\ni\omega\mapsto (\tau(\omega),\xi(\omega))\in [0,T]\times\R^d)$,
  $\Phi\with (\R^d\times[0,T]\times \R^d\ni(x,t,y)\mapsto  \varphi( \varrho \sqrt{t}x+ y )\in\R)$, $\phi\with u$ in the notation of \cite[Corollary 2.9]{jentzen2018strong}) and \cref{eq127}
  \prove\    that
 for all $A\in \mathcal{G}$ it holds  that 
\beq
\E\bigl[u(\tau,\xi)\id_{A}\big]=\E \bigl[\varphi(\xi + \varrho \sqrt{\tau} \mathbb{B} )\id_{A}\bigr].
\eeq
Combining this with the fact that $\Omega \ni \omega\mapsto u(\tau(\omega), \xi(\omega)) \in  \R$
 is $\mathcal{G}$/$\mathcal{B}(\R)$-measurable \proves\  \cref{utauxi125}.
\end{cproof}

\begin{corollary}\label{notimeFKformula}
{%

 Let $T,\rho \in (0,\infty)$, 
 $d\in \N$,  
 $ \varrho = \sqrt{ 2T \rho } $, 
 let $u\in C^{1,2}([0,T]\times \R^d)$ 
	have at most polynomially growing partial derivatives, 
	let $\varphi\colon \R^d\to \R$ be a function,  assume 
 for all $t\in [0,T]$, $x\in \R^d$ that $ u(0,x) = \varphi(x) $ and 
	\beq\label{ }
	\tfrac{\partial u}{\partial t}(t,x)= \rho \, \Delta_x u(t,x),
	\eeq
	let $(\Omega,\F, \P)$ be a probability space, 
	let
$ \mathbb{B}\colon \Omega \to \R^d$ 
be a  standard normal random variable,   
 let  $\xi  \colon \Omega \to \R^d $ be a  random variable, 
assume that $ \mathbb{B} $    and $ \xi $ are independent,   and let $\mathcal{G}\subseteq\mathcal{F}$ be the sigma-algebra generated by $\xi$.
Then it holds $\P$-a.s.\ 
that
 \beq\label{eq130}
\begin{split}
u(T,\xi)=\E\bigl[ \varphi( \varrho \mathbb{B} + \xi )|\mathcal{G} \bigr].
 \end{split}
\eeq
}

\end{corollary}
\begin{cproof}{notimeFKformula}
\Nobs\ 
\cref{noBMversion} (applied with  $T\with T$, $d\with d$,  $x\with x$,  $(\Omega,\F, \P)\with (\Omega,\F, \P)$, 
$\mathbb{B}\with \mathbb{B}$, $\xi\with\xi$,
 $\tau\with (\Omega \ni \omega \mapsto T \in [0,T])$
in the notation of \cref{noBMversion})
\proves\  \cref{eq130}.
\end{cproof}

\subsection{Deep Kolmogorov methods}
\label{sect:DKM}

In this section we discuss the deep Kolmogorov methods introduced in \cite{Beck2021} for solving linear Kolmogorov \PDEs.
Specifically, we present two reformulation results upon which deep Kolmogorov methods are based.
The first reformulation result in \cref{sec3thm31b} casts terminal values of a heat \PDEs\ as solutions of infinite-dimensional stochastic optimization problems.
The second reformulation result in \cref{sec3thm31a} casts the entire solution of a heat \PDEs\ as a solution of infinite-dimensional stochastic optimization problem.
Obtaining deep Kolmogorov methods from these reformulation results consists of replacing these infinite-dimensional stochastic optimization problems by finite-dimensional stochastic optimization problems over the parameters of \ANNs\ and thereafter applying an \SGD-type method to solve the finite-dimensional stochastic optimization problems (see, e.g., \cite{Beck2021} and \cite[Section 2]{beck2020overview}).

\begin{theorem}[Deep Kolmogorov methods for terminal values of heat \PDEs]
  \label{sec3thm31b}
{%
 Let
$ T, \rho \in (0,\infty) $,  $ d \in \N $, 
$ \varrho = \sqrt{ 2 T\rho } $,  let 
$ \varphi \colon \R^d \to \R $ be a function,
 let $u\in C^{1,2}([0,T]\times\R^d,\R)$ 
	have at most polynomially growing partial derivatives, 
assume 
 for all $t\in [0,T]$, $x\in \R^d$ that
$ u(0,x) = \varphi(x) $ and 
\begin{equation}
\label{eq:differentialu0b}
  \tfrac{ \partial u}{\partial t}(t,x) 
  = 
  \rho \, \Delta_x u(t,x),
\end{equation}
let 
$ (\Omega, \F, \P ) $ be a probability space, 
let
$ \mathbb{B}\colon \Omega \to \R^d $ and $\xi  \colon \Omega \to \R^d $ be standard normal random variables, 
assume that $ \mathbb{B}$ and $\xi$ are    independent, 
let $\basicset= C(\R^d,\R)$,
and  let $\lossexp\colon \basicset \to [0,\infty]$ satisfy for all $v\in \basicset$  that
\beq
\lossexp(v)=
\E\br[\big]{ \abs{ \varphi( \varrho \mathbb{B} + \xi ) - v(\xi) }^2 }.
\eeq
Then 
\begin{enumerate}[(i)]
 	\item\label{thm31bitem1} 
 	for all $v\in \basicset$  it holds that
 	 \beq\label{thm31bitem1a}
\lossexp(v)=\inf_{ w \in \basicset}\lossexp(w)
\eeq
  if and only if 	 it holds for all $x\in  \R^d$ that 
$ v(x) 
  = u(T,x)$ and
   	\item\label{thm31bitem2} 
	there exists $v\in \basicset$ such that 
	\beq   \lossexp(v)=\inf_{ w \in \basicset}\lossexp(w).           \eeq
 	\end{enumerate}
}
\end{theorem}
\begin{cproof}{sec3thm31b}
Throughout this proof   let $\lambda\colon \mathcal{B}(\R^d) \to [0,\infty]$ be the Lebesgue-Borel  measure,
let 
 $\solutionset\subseteq   V$ satisfy 
 \beq\label{eq148defofsolutionset} \begin{split}
 \solutionset=\bigl\{ v\in    V \colon
 \textstyle\sup_{x\in \R^d}|u(T,x)-v(x)|=0    \bigr\},
\end{split}\eeq
and let $\mathcal{G}\subseteq\mathcal{F}$ be the sigma-algebra generated by  $\xi$.
 \cref{notimeFKformula}
(applied   with $u\with u$,  $\varphi\with\varphi$,  $(\Omega, \F, \P ) \with (\Omega, \F, \P ) $,  
 $\mathbb{B}\with \mathbb{B}$,
 $\xi\with\xi$,  $\mathcal{G}\with \mathcal{G}$
 in the notation of  \cref{notimeFKformula})  and \cref{eq148defofsolutionset} 
 \prove\ that 
for all $v\in \solutionset$ 
it holds $\P$-a.s.\   that
\beq
\begin{split}
v(\xi)=\E\bigl[\varphi( \varrho\sqrt{T} \mathbb{B} + \xi )|\mathcal{G} \bigr].
 \end{split}
\eeq
\Hence that for all $v\in \solutionset$ 
it holds  that 
 \beq \P(v(\xi)-\E[ \varphi( \varrho\sqrt{T} \mathbb{B} + \xi )  |  \mathcal{G}] = 0 ) = 1.\eeq
 \Moreover\ \cref{notimeFKformula} \proves\  that  for all  $v\in \basicset  $  with $ \P(v(\xi)-\E[ \varphi( \varrho\sqrt{T} \mathbb{B} + \xi )  |  \mathcal{G}] = 0 ) = 1$  it holds $\P$-a.s.\ that 
\beq  u(T,\xi)=v(\xi).\eeq
Combining this with the  assumption that 
$\xi  \colon \Omega \to \R^d$ 
is a  standard normal random variable  \proves\ that
  for all  $v\in \basicset  $  with $ \P(v(T,\xi)-\E[ \varphi( \varrho\sqrt{T} \mathbb{B} + \xi )  |  \mathcal{G}] = 0 ) = 1$  it holds that
 \beq  \lambda\bigl( \{ x \in \R^d : u(T,x) \neq v(x) \} \bigr) = 0.
 \eeq
 The fact that  $ \R^d \ni x \mapsto u(T,x)-v(x)\in\R$ is continuous 
\proves\ that for all 
$v\in \basicset  $  with $ \P(v(T,\xi)-\E[ \varphi( \varrho\sqrt{T} \mathbb{B} + \xi )  |  \mathcal{G}] = 0 ) = 1$  it holds that
$\sup_{ x \in \R^d} | u(T,x) - v(x) | = 0.$
          \cref{lemmaB2} (applied with $(U_v)_{v\in \basicset}\with (v)_{v\in\basicset}$,  $R_1\with \xi $,  $R_2\with \mathbb{B}$,
          $\mathfrak{R}_1\with  \R^d$,  $\mathfrak{R}_2\with   \R^d$,
          $\phi\with (( \xi ,   \mathbb{B} )\mapsto  \varphi( \varrho\sqrt{T} \mathbb{B} + \xi )   $
          in the notation of \cref{lemmaB2}) \hence  \proves\  
\cref{thm31aitem1,thm31aitem2}.
\end{cproof}

\begin{theorem}[Deep Kolmogorov methods for full solutions of heat \PDEs]
  \label{sec3thm31a}
{%
 
Let
$ d \in \N $, 
$ T, \rho \in (0,\infty) $, 
$ \varrho = \sqrt{ 2 \rho } $, 
let 
$ \varphi \colon \R^d \to \R $ be a function,
 let $u\in C^{1,2}([0,T]\times\R^d,\R)$ 
	have at most polynomially growing partial derivatives, 
assume 
 for all $t\in [0,T]$, $x\in \R^d$ that
$ u(0,x) = \varphi(x) $ and 
\begin{equation}
\label{eq:differentialu0}
  \tfrac{ \partial u}{\partial t}(t,x) 
  = 
  \rho \, \Delta_x u(t,x),
\end{equation}
let 
$ (\Omega, \F, \P ) $ be a probability space, 
let
$ \mathbb{B}\colon \Omega \to \R^d$ and $\xi  \colon \Omega \to \R^d $ 
be  standard normal random variables,   
 let $\tau\colon \Omega\to[0,T]$ be a continuous uniformly
  distributed random variable,
assume that $ \mathbb{B} $,   $\tau$,        and $ \xi $ are independent,  let $\basicset= C^{1,2}([0,T]\times\R^d,\R)$,
and  let $\lossexp\colon \basicset \to [0,\infty]$ satisfy for all $v\in \basicset$  that
\beq\label{defoflossfunction32}
\lossexp(v)=
\E\br[\big]{ \abs{ \varphi( \varrho\sqrt{\tau} \mathbb{B} + \xi ) - v(\tau,\xi) }^2 }.
\eeq
Then 
\begin{enumerate}[(i)]
 	\item\label{thm31aitem1} 
 	for all $v\in \basicset$  it holds that
 	 \beq\label{thm31aitem1a}
\lossexp(v)=\inf_{ w \in \basicset}\lossexp(w)
\eeq
  if and only if 	 
    it holds  for all $t \in  [0,T]$,  $x\in \R^d$ that
$ v(0,x) = \varphi(x) $ and 
\begin{equation}
  \tfrac{ \partial v}{\partial t}(t,x) 
  = 
  \rho \, \Delta_x v(t,x),
\end{equation}
and
   	\item\label{thm31aitem2} 
	there exists $v\in \basicset$ such that 
	\beq   \lossexp(v)=\inf_{ w \in \basicset}\lossexp(w).           \eeq
 	\end{enumerate}
}
\end{theorem}
\begin{cproof}{sec3thm31a}
Throughout this proof  
for every $v\in \basicset  $ 
let 
 $F_v \colon [0,T]\times\R^d  \to  \R $ satisfy for all  $t\in [0,T]$, $x\in\R^d$
that
\beq F_v(t,x)=|  v(0,x) - \varphi(x) |^2+| \tfrac{ \partial v}{\partial t}(t,x)-\Delta_x v(t,x) |^2,\eeq
let
 $\solutionset\subseteq   V$ satisfy 
 \beq \begin{split}
 \solutionset=\{ v\in    V \colon
 \textstyle\sup_{(t,x)\in  [0,T]\times\R^d} F_{v}(t,x)=0    \},
\end{split}\eeq
and let $\mathcal{G}\subseteq\mathcal{F}$ be the sigma-algebra generated by $\tau$ and $\xi$.
 \cref{noBMversion}
(applied for every $v\in \solutionset$  with $u\with v$,  $\varphi\with\varphi$,  $(\Omega, \F, \P ) \with (\Omega, \F, \P ) $,  
 $\mathbb{B}\with \mathbb{B}$,
$\tau\with\tau$,    $\xi\with\xi$,    $\mathcal{G}\with \mathcal{G} $
 in the notation of  \cref{noBMversion})
 \proves\ that 
for all $v\in \solutionset$ 
it holds $\P$-a.s.\   that
\beq
\begin{split}
v(\tau,\xi)=\E\bigl[\varphi( \varrho\sqrt{\tau} \mathbb{B} + \xi )|\mathcal{G} \bigr].
 \end{split}
\eeq
\Hence that for all $v\in \solutionset$ 
it holds  that 
 $ \P(v(\tau,\xi)-\E[ \varphi( \varrho\sqrt{\tau} \mathbb{B} + \xi )  |  \mathcal{G}] = 0 ) = 1$.
 Combining  \cref{noBMversion} and the assumption that for all $t\in [0,T]$, $x\in\R^d$ it holds  that 
 $ u(0,x) = \varphi(x) $ and 
$ \tfrac{ \partial u}{\partial t}(t,x) 
  = 
  \rho \, \Delta_x u(t,x)$
\proves\ that  it holds $\P$-a.s.\ that 
\beq
u(\tau,\xi)=\E\bigl[\varphi( \varrho\sqrt{\tau} \mathbb{B} + \xi )|\mathcal{G} \bigr].
\eeq
 \Hence that for all  $v\in \basicset  $  with $ \P(v(\tau,\xi)-\E[ \varphi( \varrho\sqrt{\tau} \mathbb{B} + \xi )  |  \mathcal{G}] = 0 ) = 1$  it holds $\P$-a.s.\ that 
$u(\tau,\xi)=v(\tau,\xi)$.
Combining this with the  assumption that 
$\xi  \colon \Omega \to \R^d $ 
is a  standard normal random variable and  the  assumption
that  $\tau\colon \Omega\to[0,T]$ is a continuous uniformly
  distributed random variable \proves\ that
  for all  $v\in \basicset  $  with $ \P(v(\tau,\xi)-\E[ \varphi( \varrho\sqrt{\tau} \mathbb{B} + \xi )  |  \mathcal{G}] = 0 ) = 1$  it holds for  a.e.\  $(t,x)\in [0,T]\times\R^d$  that 
$u(t,x)=v(t,x)$.
The assumption that $
  u 
  \in \basicset
$ 
\hence \proves\ that
  for all  $v\in \basicset  $,  $(t,x)\in [0,T]\times\R^d$  with $ \P(v(\tau,\xi)-\E[ \varphi( \varrho\sqrt{\tau} \mathbb{B} + \xi )  |  \mathcal{G}] = 0 ) = 1$  it holds   that 
$u(t,x)=v(t,x)$.
          \cref{lemmaB2} (applied with $U_v\with v$,  $R_1\with (\tau,\xi) $,  $R_2\with \mathbb{B}$,
          $\mathfrak{R}_1\with  [0,T]\times\R^d$,  $\mathfrak{R}_2\with   \R^d$,
          $\phi\with ( (\tau,\xi),   \mathbb{B} )\mapsto  \varphi( \varrho\sqrt{\tau} \mathbb{B} + \xi )   $
          in the notation of  \cref{lemmaB2}) \hence  \proves\  
\cref{thm31aitem1} and 
\cref{thm31aitem2}.
\end{cproof}

\subsection{Deep BSDE methods}\label{sec23}

In this section we discuss the deep \BSDE\ methods introduced in \cite{MR3847747,MR3736669} for the approximation of semilinear parabolic \PDEs.
First, we present some well-known uniqueness results for solutions of \BSDEs\ in \cref{sec:uniquenessBSDE}.
Thereafter, we introduce a reformulation result for semilinear parabolic \PDEs\ in \cref{sec:reformulationPDEasOptimization} which casts terminal values of semilinear parabolic \PDEs\ as solutions of infinite-dimensional stochastic optimization problems involving \BSDEs.
Obtaining deep \BSDE\ methods from this reformulation result then consists of 
  discretizing the involved \BSDEs,
  replacing the unknown functions in the \BSDEs\ by \ANNs,
  and
  applying \SGD-type methods to the resulting finite-dimensional stochastic optimization problems.

\subsubsection{Uniqueness for solutions of BSDEs}
\label{sec:uniquenessBSDE}

\begin{lemma}[Uniqueness for solutions for BSDEs]\label{uniqueofbsde}
{%
 Let $T,M\in (0,\infty)$, $d\in \N$,
	let $(\Omega,\F, \P)$ be a probability space with a normal filtration $(\mathbb{F}_t)_{t\in [0,T]}$, let $\xi\colon \Omega\to \R^d $  be  $\mathbb{F}_0$-measurable with $\E[\|\xi\|^2]< \infty $,
	let $\BM\colon [0,T]\times\Omega \to \R^d$ be a standard $(\mathbb{F}_t)_{t\in [0,T]}$-Brownian motion,
		let $\mu\colon[0,T]\times \R^d \rightarrow \R^d $,
		$\sigma\colon[0,T]\times \R^d\rightarrow \R^{d\times d} $,
		$g\colon \R^d \to \R $,
		 and $f\colon [0,T]\times \R^d \times \R\times\R^d  \rightarrow \R$ be Lipschitz continuous, 
	let $X\colon [0,T]\times \Omega \to \R^d$  be an $(\mathbb{F}_t)_{t \in [0,T]}$-adapted stochastic process with continuous sample paths (w.c.s.p.),  which satisfies that for every $t \in [0,T]$ it holds $\P$-a.s.\ that
\beq X_t=\xi+\int_{0}^t\mu(s, X_s)\,\mathrm{d}s+\int_{0}^t\sigma(s, X_s)\,\mathrm{d}\BM_s,
	\eeq
	and assume $\int_0^T \E\big[ \|X_s\|^2 \big] \,\mathrm{d}s < \infty$,
	for every $k \in \{ 1, 2 \}$ let $Y^k \colon [0,T]\times \Omega \to \R$ and $Z^k \colon [0,T]\times \Omega \to \R^d$  be  {$(\mathbb{F}_t)_{t \in [0,T]}$-adapted} stochastic processes with continuous sample paths,  which satisfy that for every $t \in [0,T]$ it holds $\P$-a.s.\ that
\beq\label{fbsdek12} Y^k_t-g(X_T)-\int_{t}^T f(s,X_s, Y^k_s, Z^k_s)\,\mathrm{d}s+\int_{t}^T (Z_s^k)^* \,\mathrm{d}\BM_s =0,\eeq  
and  assume $\int_0^T \E\big[ |Y_s|^2 \big] \,\mathrm{d}s < \infty$.
Then
\beq
  \P\bigl( Y^1 = Y^2 \bigr)= \P\bigl( Z^1 = Z^2 \bigr)= 1.
\eeq
}	
\end{lemma} 
 
 \begin{cproof}{uniqueofbsde}
\Nobs\ \cref{fbsdek12} \proves\
  that for all $ s \in [0,T] $ it holds $ \P $-a.s.\ that
\begin{align}\begin{split}
& Y^{1}_s-Y_s^2 \\
&=
\int_{s}^{T}\Big[ f(t,X_t, Y^1_t, Z^1_t)-f(t,X_t, Y^2_t, Z^2_t)\Bigr]
\,\mathrm{d}t
-\int_{s}^{T}\bigl[
  Z^{1}_t-Z_t^2\bigr] \,\mathrm{d}\BM_t.
\end{split}\label{k01}\end{align}
\Nobs\ Hutzenthaler et al.~\cite[Lemma 3.1]{MR4488351}
 (applied  with $Y\with Y^{1} -Y^2$, \\$A\with f(\cdot,X_{\cdot}, Y^1_{\cdot}, Z^1_{\cdot})-f(\cdot,X_{\cdot}, Y^2_{\cdot}, Z^2_{\cdot})$,
$Z \with   Z^{1} -Z^2$ in the notation of (46) in \cite{MR4488351}) 
\proves\ that for all $s\in[0,T]$, $\lambda\in (0,\infty)$  it holds that
\beq\begin{split}
& \E\Big[e^{\lambda s}|Y^{1}_s -Y^2_s|^2+\int_{s}^Te^{\lambda t}\|Z_t^{1} -Z_t^2\|^2\big |\mathbb{F}_s \,\mathrm{d}t\Bigr]\\ &\quad\leq \E\Big[\int_s^T\frac{e^{\lambda t}}{\lambda}\|f(t,X_{t}, Y^1_{t}, Z^1_{t})-f(t,X_{t}, Y^2_{t}, Z^2_{t})\|^2\,\mathrm{d}t\big |\mathbb{F}_s\Bigr].
\end{split}\eeq
Taking the expectation and for all $s\in[0,T]$, $\lambda\in (0,\infty)$
it holds  that
\beq\label{eq:b01}\begin{split}
& \E\Big[e^{\lambda s}|Y^{1}_s -Y^2_s|^2+\int_{s}^Te^{\lambda t}\|Z_t^{1} -Z_t^2\|^2 \,\mathrm{d}t\Bigr]\\ &\quad\leq \E\Big[\int_s^T\frac{e^{\lambda t}}{\lambda}\|f(t,X_{t}, Y^1_{t}, Z^1_{t})-f(t,X_{t}, Y^2_{t}, Z^2_{t})\|^2\,\mathrm{d}t\Bigr]\\
&\quad \leq 
\E\Big[\int_s^T\frac{Le^{\lambda t}}{\lambda}\bigl(|Y^{1}_t -Y^2_t|^2+\|Z_t^{1} -Z_t^2\|^2\bigr) \,\mathrm{d}t\Bigr],
\end{split}\eeq
where $L$ is the Lipschitz constant of $f$.
Let $\lambda >2LT+1$ and we immediately have for all $s\in [0,T]$,  $N\in \N$ it holds that 
\beq\begin{split}
\E\bigl[e^{\lambda s}|Y^{1}_s -Y^2_s|^2\bigr]&\leq \E\left[\int_s^T\frac{e^{\lambda t}}{2T}|Y_t^{1} -Y_t^2|^2 \,\mathrm{d}t\right]=\frac{1}{2T}\int_s^T\E\bigl[e^{\lambda t}|Y_t^{1} -Y_t^2|^2\bigr] \,\mathrm{d}t.
\end{split}\eeq
Induction \hence shows that  for all $N\in \N$,  $t_0\in [0,T]$  it holds  that
\beq\begin{split}
& \E\bigl[e^{\lambda t_0}|Y^{1}_{t_0} -Y^2_{t_0}|^2\bigr]
\leq \Big[\frac{1}{2T}\Bigr]^2\int_{t_0}^T\int_{t_1}^T \E\bigl[e^{\lambda t_2}|Y^{1}_{t_2} -Y^2_{t_2}|^2\bigr] \,\mathrm{d}t_2 \,\mathrm{d}t_1\\
&\leq \Big[\frac{1}{2T}\Bigr]^{N+1}\int_{t_0}^T\int_{t_1}^T\int_{t_2}^T\ldots\int_{t_{N-1}}^T\int_{t_{N}}^T  \E\bigl[e^{\lambda s}|Y^{1}_{s} -Y^2_{s}|^2\bigr] \,\mathrm{d}s \,\mathrm{d}t_{N} \ldots \, \mathrm{d}t_{3}  \,\mathrm{d}t_2 \,\mathrm{d}t_1\\
&\leq \Big[\frac{1}{2T}\Bigr]^{N+1}\int_{t_0}^T\int_{t_1}^T\int_{t_2}^T\ldots\int_{t_{N-1}}^T\int_{t_{0}}^T  \E\big[e^{\lambda s}|Y^{1}_{s} -Y^2_{s}|^2\bigr] \,\mathrm{d}s \,\mathrm{d}t_{N} \ldots \, \mathrm{d}t_{3}  \,\mathrm{d}t_2 \,\mathrm{d}t_1\\
&= \Big[\frac{1}{2T}\Bigr]^{N+1}
\left[\int_{t_0}^T\E \bigl[e^{\lambda s}|Y^{1}_{s} -Y^2_{s}|^2\bigr] \,\mathrm{d}s\right]
\left[\int_{t_0}^T\int_{t_1}^T\int_{t_2}^T\ldots\int_{t_{N-1}}^T   \,\mathrm{d}t_{N} \ldots \, \mathrm{d}t_{3}  \,\mathrm{d}t_2 \,\mathrm{d}t_1\right]\\
&\leq  \Big[\frac{1}{2T}\Bigr]^{N+1}\left[\int_{t_0}^T\E \bigl[e^{\lambda s}|Y^{1}_{s} -Y^2_{s}|^2\bigr] \,\mathrm{d}s\right]\frac{(T-t_0)^N}{N!}\\
&\leq  \Big[\frac{1}{2T}\Bigr]^{N+1}\frac{(T-t_0)^N}{N!}e^{\lambda T}\int_{t_0}^T\E \bigl[|Y^{1}_{s} -Y^2_{s}|^2\bigr] \,\mathrm{d}s.
\end{split}
\eeq
\Hence that  for all $t\in [0,T]$ it holds that 
\beq
\E\bigl[e^{\lambda t}|Y^{1}_{t} -Y^2_{t}|^2\bigr]=0.
\eeq
\Hence that for all $t\in [0,T]$ it holds that
\beq
  \P\bigl( Y_t^1 = Y_t^2 \bigr) = 1.
\eeq
The fact the every $k \in \{ 1, 2 \}$ it holds that $Y^k$  is a stochastic process with continuous sample paths \hence \proves\   that 
 \beq
  \P\bigl( Y^1 = Y^2 \bigr) = 1.
\eeq
Combining this with  \cref{eq:b01} \proves\ that it holds that 
\beq
\E\Big[\int_{s}^Te^{\lambda t}\|Z_t^{1} -Z_t^2\|^2 \,\mathrm{d}t\Bigr]=0, 
\eeq
\hence for a.e.\  $t\in [0,T]$  it holds that
\beq
  \P\bigl( Z_t^1 = Z_t^2 \bigr) = 1.
\eeq
The fact that  $Z$ is a stochastic process with continuous  sample paths \proves\ that
for all $t\in [0,T]$  it holds that
\beq
  \P\bigl( Z_t^1 = Z_t^2 \bigr) = 1,
\eeq
and 
\beq  \P\bigl( Z^1 = Z^2 \bigr) = 1. \eeq
 \end{cproof}

 \begin{lemma}\label{uniqueytvv}
{%
 Let $v\in C_b(\R^d,\R)$, $V\in C_b([0,T]\times \R^d,\R^d)$,  let $f\colon [0,T]\times \R^d \times \R\times\R^d  \rightarrow \R$ be Lipschitz continuous,  let $(\Omega,\F, (\mathbb{F}_t)_{t\in [0,T]},\P)$  be a filtered probability space, 
let $X\colon [0,T]\times \Omega \to \R^d$  be an $(\mathbb{F}_t)_{t \in [0,T]}$-adapted stochastic process with continuous sample paths,
and
for every $k \in \{ 1, 2 \}$ let $Y^k \colon [0,T]\times \Omega \to \R$ be a  stochastic process with continuous sample paths which satisfies that  for all $t\in[0,T]$ it holds $\P$-a.s.\ that
\beq\label{YtvV0}
Y^k_t=v(X_0)-\int_{0}^t f\bigl(s,X_s, Y_s^k,V(s,X_s)\bigr)\,\mathrm{d}s+\int_{0}^t [V(s,X_s)]^* \,\mathrm{d}\BM_s.\eeq
Then $\P( Y^1=Y^2)=1      $.

}
\end{lemma}
\begin{cproof}{uniqueytvv}
\Nobs\ \cref{YtvV0}  \proves\ that  for all $t\in [0,T]$ 
it holds $\P$-a.s.\ that
\beq\label{YtvV2}
\begin{split}
\left| Y_t^1-Y_t^2\right|&=\left|-\int_{0}^t\big[ f\bigl(s,X_s, Y_s^1,V(s,X_s)\bigr)-f\bigl(s,X_s, Y_s^2,V(s,X_s)\bigr)\big]\,\mathrm{d}s \right| 
\\
&\leq  L\int_{0}^t \big| Y_s^1-Y_s^2\big| \,\mathrm{d}s.
\end{split}
\eeq
Combining this with  Gronwall's inequality and the fact that it holds  $\P$-a.s.\ that $ Y_0^1=Y_0^2$ \proves\ 
 that $\P( Y_t^1=Y_t^2)=1      $.
Combining this with the fact that  for every $k \in \{ 1, 2 \}$  $Y^k$ has  continuous sample paths  \proves\ 
that
 \beq
  \P\bigl( Y^1 = Y^2 \bigr) = 1.\eeq

\end{cproof}

\subsubsection{Reformulating semilinear PDEs as infinite-dimensional stochastic optimization problems}
\label{sec:reformulationPDEasOptimization}

\begin{theorem}\label{bsdethm}
	{%
 
	Let $T\in (0,\infty)$, $d\in \N$,
	let $\mu\colon[0,T]\times \R^d \rightarrow \R^d $,  
		$\sigma\colon[0,T]\times \R^d\rightarrow \R^{d\times d} $, 
		$g\colon \R^d\to \R$,
		 and $f\colon [0,T]\times \R^d \times \R\times\R^d  \rightarrow \R$ be Lipschitz continuous,  let
	$u\in C^{1,2}([0,T]\times \R^d)$ satisfy for all $t\in [0,T]$, $x\in \R^d$ that
	$u(T,x)=g(x)$ and
	\begin{multline}\label{thm22pde}
	\tfrac{\partial u}{\partial t}(t,x)+\tfrac12 {\rm{Tr}}\bigl((\sigma(t,x)[\sigma(t,x)]^*({\rm{Hess}}_xu)(t,x) \bigr)     +\tfrac{\partial u}{\partial x} (t,x) \, \mu(t,x)\\
+f\bigl(t,x,u(t,x), [ \sigma( t, x ) ]^*( \nabla_x u )( t, x )\bigr)    =0,
	\end{multline}
	let $(\Omega,\F, (\mathbb{F}_t)_{t\in [0,T]},\P)$  be a filtered probability space,  let 
		$\BM\colon [0,T]\times\Omega \to \R^d$ be a standard $(\mathbb{F}_t)_{t\in [0,T]}$-Brownian motion,  let $X\colon [0,T]\times \Omega \to \R^d$  be an   $(\mathbb{F}_t)_{t \in [0,T]}$-adapted stochastic processes with continuous sample paths which satisfies that  for all $t\in [0,T]$ it holds $\P$-a.s.\ that
\beq\label{FBSDE}
\begin{split}
 X_t=X_0+\int_{0}^t\mu(s, X_s)\,\mathrm{d}s+\int_{0}^t\sigma(s, X_s)\,\mathrm{d}\BM_s,
 \end{split}
\eeq
and for\footnote{\Nobs\ for all metric spaces $( E, d )$ and $( \mathcal{E}, \delta )$ it holds that $C_b( E, \mathcal{E} )$ is the set of all bounded continuous functions from $E$ to $ \mathcal{E}$.} every $v\in C_b(\R^d,\R)$, $V\in C_b([0,T]\times \R^d,\R^d)$ let $Y^{v,V}\colon [0,T]\times \Omega \to \R $ be a stochastic process with continuous sample paths which satisfies that for all $t \in [0,T]$ it holds $\P$-a.s.\ that
\beq\label{YtvV}
Y_t^{v,V}=v(X_0)-\int_{0}^t f\bigl(s,X_s, Y_s^{v,V},V(s,X_s)\bigr)\,\mathrm{d}s+\int_{0}^t [V(s,X_s)]^* \,\mathrm{d}\BM_s.\eeq
Then 
\begin{enumerate}[(i)]
\item\label{thm24item1} for all $v\in C_b(\R^d,\R)$, $V\in C_b([0,T]\times \R^d,\R^d)$ it holds that 
\beq\label{ivVmini}
\E\bigl[| Y_T^{v,V}-g(X_T)|^2\bigr]=
\inf_{ \substack{ w \in C_b(\R^d,\R), \\ W \in C_b([0,T]\times \R^d,\R^d) } }
 \E\bigl[| Y_T^{w,W}-g(X_T)|^2\bigr]
\eeq
if and only if 
\beq
\P\Bigl(v(X_0)=u(0,X_0)\Bigr)=\P\Bigl( \forall \, t \in [0,T]\colon V(t,X_t)=[\sigma(t,X_t)]^*  ( \nabla_x u )(t,X_t) \Bigr)=1
\eeq 
and 
\item\label{thm24item2} there exist $v\in C_b(\R^d,\R)$, $V\in C_b([0,T]\times \R^d,\R^d)$ such that 
\beq\label{ivVmini0}
\E\bigl[| Y_T^{v,V}-g(X_T)|^2\bigr]=
\inf_{ \substack{ w \in C_b(\R^d,\R), \\ W \in C_b([0,T]\times \R^d,\R^d) } }
 \E\bigl[| Y_T^{w,W}-g(X_T)|^2\bigr]=0.
\eeq
\end{enumerate}

	 \begin{cproof}{bsdethm}
Throughout this proof 
let 
\begin{multline}
\mathcal{S}=\{(v,V) \in C_b(\R^d,\R)\times C_b([0,T]\times \R^d,\R^d)\colon \\  \P\Bigl(v(X_0)=u(0,X_0)\Bigr)=\P\Bigl( \forall \, t \in [0,T]\colon V(t,X_t)=[\sigma(t,X_t)]^*  ( \nabla_x u )(t,X_t) \Bigr)=1 \},         \end{multline}
let $c\in [0,\infty]$ satisfy 
\beq\label{defofc} c=\inf_{ \substack{ w \in C_b(\R^d,\R), \\ W \in C_b([0,T]\times \R^d,\R^d) } }
 \E\bigl[| Y_T^{w,W}-g(X_T)|^2\bigr],\eeq
 let 
$\mathscr{v} \in C_b(\R^d,\R)$,   $\mathscr{V} \in C_b([0,T]\times \R^d,\R^d)$ satisfy for all 
$t\in [0,T]$,  $x\in \R^d$ that 
\beq \label{eq31} \mathscr{v}(x) = u(0,x)\qquad \text{and} \qquad
\mathscr{V} (t,x) = [ \sigma(t,x) ]^* (\nabla_xu)(t,x)\eeq
 and let 
	 	$\mathcal{Y}\colon   [0,T]\times \Omega \to \R    $ and $\mathcal{Z}\colon   [0,T]\times \Omega \to \R^d   $ satisfy for all $t\in [0,T]$ that
	 	\beq\label{eq32}
	 	 \mathcal{Y}_t=u(t,X_t)\qquad \text{and} \qquad \mathcal{Z}_t=\mathscr{V}(t,X_t).
	 	 \eeq
	\Nobs\  	 $(\mathscr{v}, \mathscr{V})\in  \mathcal{S}   $.  \Hence that $\mathcal{S}$ is not empty.
	 	  \Nobs\ \cref{thm22pde}, \cref{FBSDE},  \cref{eq32},  and  \cref{nonlinearFKformula} \prove\  that 
	 for all $s, t\in [0,T]$  it holds $\P$-a.s.\ that
  \beq\label{FBSDE0forward2}
\begin{split}
  \mathcal{Y}_t- \mathcal{Y}_s+\int_{s}^t f(r,X_r,  \mathcal{Y}_r, \mathcal{Z}_r)\,\mathrm{d}r-\int_{s}^t (\mathcal{Z}_r)^* \,\mathrm{d}\BM_r =0.
 \end{split}
\eeq
\Hence that 
\beq\label{ytvv2}
\begin{split}
 \mathcal{Y}_t-u(0,X_0)+\int_{0}^t f(s,X_s, \mathcal{Y}_s, \mathcal{Z}_s)\,\mathrm{d}s-\int_{0}^t (\mathcal{Z}_s)^* \,\mathrm{d}\BM_s =0.
 \end{split}
\eeq
\Moreover\ 	combining \cref{FBSDE0forward2} with the fact that $u(T,x)=g(x)$ \proves\ that  for all $t\in [0,T]$   it holds $\P$-a.s.\  that 
\beq\label{ytvv22}
\begin{split}
\mathcal{Y}_t- g(X_T)-\int_t^T f(s,X_s, \mathcal{Y}_s, \mathcal{Z}_s)\,\mathrm{d}s+\int_t^T (\mathcal{Z}_s)^* \,\mathrm{d}\BM_s =0.
 \end{split}
\eeq
\Moreover\   \cref{YtvV} \proves\  that  for all 
$(v,V)\in  \mathcal{S}  $
	 it holds
	 	 $\P$-a.s.\    that
\beq\label{ytvv3}
\begin{split}
 Y_t^{v,V}-u(0,X_0)+\int_{0}^t f(s,X_s, Y_s^{v,V}, \mathcal{Z}_s)\,\mathrm{d}s-\int_{0}^t (\mathcal{Z}_s)^* \,\mathrm{d}\BM_s =0.
 \end{split}
\eeq
Combining this with \cref{ytvv2} and \cref{uniqueytvv}  \proves\  that
 for all 
$(v,V)\in  \mathcal{S}$  it holds that
	 	 \beq\label{eq35} \E\bigl[| Y_T^{v,V}-g(X_T)|^2\bigr]=\E\bigl[| \mathcal{Y}_T-g(X_T)|^2\bigr]=\E\bigl[| u(T,X_T)-g(X_T)|^2\bigr]=0.
	 	 \eeq
\Nobs\  \cref{eq35},  \cref{ivVmini} and the fact that $\mathcal{S}$ is not empty  \prove\ that 
for all $v\in C_b(\R^d,\R)$, $V\in C_b([0,T]\times \R^d,\R^d)$  with
$\E\bigl[| Y_T^{v,V}-g(X_T)|^2\bigr]=c$
it holds that 
\beq
\E\bigl[| Y_T^{v,V}-g(X_T)|^2\bigr]=0.
\eeq
 \Hence that 
 for all $v\in C_b(\R^d,\R)$, $V\in C_b([0,T]\times \R^d,\R^d)$  with
$\E\bigl[| Y_T^{v,V}-g(X_T)|^2\bigr]=c$  that
 $\P(Y_T^{v,V}=g(X_T))=1$.   Combining this with
	 	\cref{YtvV}   \proves\  that   for all  $t\in [0,T]$,   $v\in C_b(\R^d,\R)$,  $V\in C_b([0,T]\times \R^d,\R^d)$  with
$\E\bigl[| Y_T^{v,V}-g(X_T)|^2\bigr]=c$ it holds    $\P$-a.s.\  that 
	 	 \beq
\begin{split}
Y_t^{v,V}=g(X_T)+\int_{t}^T f\bigl(s,X_s, Y_s^{v,V},V(s,X_s)\bigr)\,\mathrm{d}s-\int_{t}^T [V(s,X_s)]^* \,\mathrm{d}\BM_s.
 \end{split}
\eeq
	 		      \Cref{uniqueofbsde} and  \cref{ytvv22}  \hence   \prove\  that   for all $t\in [0,T]$,    $v\in C_b(\R^d,\R)$, $V\in C_b([0,T]\times \R^d,\R^d)$  with
$\E\bigl[| Y_T^{v,V}-g(X_T)|^2\bigr]=c$ it holds    $\P$-a.s.\  that     
	 	 $\bigl(Y^{v,V},V(\cdot,X_{\cdot})\bigr)=\bigl(\mathcal{Y}, \mathcal{Z} \bigr) $.
	\Hence that 	
	for all $s\in [0,T]$,  $v\in C_b(\R^d,\R)$, $V\in C_b([0,T]\times \R^d,\R^d)$  with
$\E\bigl[| Y_T^{v,V}-g(X_T)|^2\bigr]=c$ it holds  $\P$-a.s.\  that  $v(X_0)=u(0,X_0)$, $V(s,X_s)=\mathcal{Z}_s=[\sigma(s,X_s)]^*(\nabla_xu)(s,X_s)$,   for all $s\in [0,T]$. 	
\Hence that for all    $v\in C_b(\R^d,\R)$, $V\in C_b([0,T]\times \R^d,\R^d)$  with
$\E\bigl[| Y_T^{v,V}-g(X_T)|^2\bigr]=c$,   it holds that $(v,V)\in \mathcal{S}$.
 \Cref{lemmaBcor} (applied with  $S\with \mathcal{S}$,  $V\with C_b(\R^d,\R)\times C_b([0,T]\times \R^d,\R^d)$,   $(U_v)_{v\in V}\with (U_v)_{v\in V}$,    $(\Omega, \F,\P) \with (\Omega, \F,\P)$ in the notation of  \cref{lemmaBcor})    \proves\  \cref{thm24item1,thm24item2}.
 	 \end{cproof}	}
\end{theorem}

\begin{corollary}\label{cor25}
{%

Let $T\in (0,\infty)$, $d\in \N$,
		let $\mu\colon[0,T]\times \R^d \rightarrow \R^d $,  
		$\sigma\colon[0,T]\times \R^d\rightarrow \R^{d\times d} $,  $g\colon \R^d\to \R$,   and $f\colon [0,T]\times \R^d \times \R\times\R^d  \rightarrow \R$ be Lipschitz continuous,  let
	$u\in C^{1,2}([0,T]\times \R^d)$ satisfy for all $t\in [0,T]$, $x\in \R^d$ that
	$u(T,x)=g(x)$ and
	\begin{multline}\label{thm22pdec}
	\tfrac{\partial u}{\partial t}(t,x)+\tfrac12 {\rm{Tr}}\bigl((\sigma(t,x)[\sigma(t,x)]^*({\rm{Hess}}_xu)(t,x) \bigr)     +\tfrac{\partial u}{\partial x} (t,x) \, \mu(t,x)\\
+f\bigl(t,x,u(t,x), [ \sigma( t, x ) ]^*( \nabla_x u )( t, x )\bigr)    =0,
	\end{multline}
let $(\Omega,\F, (\mathbb{F}_t)_{t\in [0,T]},\P)$  be a filtered probability space,  let 
		$\BM \colon [0,T]\times\Omega \to \R^d$ be a standard $(\mathbb{F}_t)_{t\in [0,T]}$-Brownian motion,	  let $X\colon [0,T]\times \Omega \to \R^d$  be an   $(\mathbb{F}_t)_{t \in [0,T]}$-adapted stochastic process with continuous sample paths which satisfies that  for all $t\in [0,T]$ it holds $\P$-a.s.\ that
\beq\label{FBSDEc}
\begin{split}
 X_t=X_0+\int_{0}^t\mu(s, X_s)\,\mathrm{d}s+\int_{0}^t\sigma(s, X_s)\,\mathrm{d}\BM_s,
 \end{split}
\eeq
 assume for all $\varepsilon \in (0,\infty)$, $x \in \R^d$ that $\P( \| X_0 - x \| < \varepsilon ) > 0$,
and  for  every $v\in C_b(\R^d,\R)$, $V\in C_b([0,T]\times \R^d,\R^d)$  let $Y^{v,V}\colon [0,T]\times \Omega \to \R $ be a stochastic process with continuous sample paths which satisfies that for all $t \in [0,T]$ it holds $\P$-a.s.\ that
\beq\label{YtvVc}
Y_t^{v,V}=v(X_0)-\int_{0}^t f\bigl(s,X_s, Y_s^{v,V},V(s,X_s)\bigr)\,\mathrm{d}s+\int_{0}^t [V(s,X_s)]^* \,\mathrm{d}\BM_s.\eeq
Then
\begin{enumerate}[(i)]
\item \label{cor25item1}   there exist $v\in C_b(\R^d,\R)$, $V\in C_b([0,T]\times \R^d,\R^d)$ which satisfy
\beq\label{ivVmini0c}
\E\bigl[| Y_T^{v,V}-g(X_T)|^2\bigr]=
\inf_{ \substack{ w \in C_b(\R^d,\R), \\ W \in C_b([0,T]\times \R^d,\R^d) } }
 \E\bigl[| Y_T^{w,W}-g(X_T)|^2\bigr]=0
\eeq
and
\item  \label{cor25item2} it holds for all $ x \in \R^d $ that $ v(x) = u(0,x) $.
\end{enumerate}

}

\end{corollary}
\begin{cproof}{cor25}
\Nobs\  \cref{thm24item2} in \cref{bsdethm} \proves\  that there exist\\ $v\in C_b(\R^d,\R)$, $V\in C_b([0,T]\times \R^d,\R^d)$ which satisfy
\beq\label{ivVmini0c1}
\E\bigl[| Y_T^{v,V}-g(X_T)|^2\bigr]=
\inf_{ \substack{ w \in C_b(\R^d,\R), \\ W \in C_b([0,T]\times \R^d,\R^d) } }
 \E\bigl[| Y_T^{w,W}-g(X_T)|^2\bigr]=0.
\eeq
This \proves\  \cref{cor25item1}. 
\Nobs\  \cref{thm24item1} in \cref{bsdethm} and  \cref{ivVmini0c1}  \prove\ that   
\beq
\P\bigl(v(X_0)=u(0,X_0)\bigr)=1.
\eeq 
Combining this and \cref{sec2lem6} (applied with $f\with ( \R^d \ni x \mapsto v(x) - u(0,x) \in \R )$ in the notation of \cref{sec2lem6}) \proves\ \cref{cor25item2}. 
\end{cproof}

\section{Operator learning methods}
\label{section:NeurOP}

\providecommandordefault{\nrEvalRuns}{1000}
\providecommandordefault{\nrEvalRunsOL}{\nrEvalRuns}

\newcommand{\OLBurgersT}{1}
\newcommand{\OLBurgersSpaceSize}{2\pi}
\newcommand{\OLBurgersLaplaceFactor}{\frac{1}{10}}

\newcommand{\OLBurgersdecayRate}{6}
\newcommand{\OLBurgersoffset}{10}
\newcommand{\OLBurgersvar}{10^6}
\newcommand{\OLBurgersinnerdecay}{2}
\newcommand{\OLBurgersInitDistr}{\mathcal{N}(0, \OLBurgersvar(\OLBurgersoffset \id_{\initialValues} - \Delta_x)^{-\OLBurgersdecayRate})}

\newcommand{\OLBurgersSpaceStep}{128}
\newcommand{\OLBurgersBatchSize}{1024}
\newcommand{\OLBurgersEvalSteps}{400}
\newcommand{\OLBurgersImprTol}{0.96}
\newcommand{\OLBurgersNrOLRuns}{5}
\newcommand{\OLBurgersRefAlgOne}{Spectral/Crank-Nicolson}
\newcommand{\OLBurgersRefAlgTwo}{explicit midpoint}
\newcommand{\OLBurgersNrTrainSamples}{2^{18}}
\newcommand{\OLBurgersNrValidSamples}{2^{14}}
\newcommand{\OLBurgersNrTestSamples}{2^{14}}
\newcommand{\OLBurgersNrTrainSpaceDiscr}{512}
\newcommand{\OLBurgersNrValidSpaceDiscr}{1024}
\newcommand{\OLBurgersNrTestSpaceDiscr}{1024}
\newcommand{\OLBurgersNrTrainTimeSteps}{1000}
\newcommand{\OLBurgersNrValidTimeSteps}{1500}
\newcommand{\OLBurgersNrTestTimeSteps}{1500}

\newcommand{\OLAConedSpaceStep}{64}
\newcommand{\OLAConedParams}{@@}
\newcommand{\OLAConedBatchSize}{128}
\newcommand{\OLAConedEvalSteps}{400}
\newcommand{\OLAConedImprTol}{0.97}
\newcommand{\OLAConedNrOLRuns}{3}
\newcommand{\OLAConedRefAlgOne}{Spectral/Crank-Nicolson}
\newcommand{\OLAConedRefAlgTwo}{explicit midpoint}
\newcommand{\OLAConedNrTrainSamples}{2^{18}}
\newcommand{\OLAConedNrValidSamples}{2^{14}}
\newcommand{\OLAConedNrTestSamples}{2^{14}}
\newcommand{\OLAConedNrTrainSpaceDiscr}{128}
\newcommand{\OLAConedNrValidSpaceDiscr}{256}
\newcommand{\OLAConedNrTestSpaceDiscr}{256}
\newcommand{\OLAConedNrTrainTimeSteps}{1000}
\newcommand{\OLAConedNrValidTimeSteps}{1400}
\newcommand{\OLAConedNrTestTimeSteps}{1400}

\newcommand{\OLACtwodSpaceStep}{64}
\newcommand{\OLACtwodParams}{@@}
\newcommand{\OLACtwodBatchSize}{64}
\newcommand{\OLACtwodEvalSteps}{400}
\newcommand{\OLACtwodImprTol}{0.97}
\newcommand{\OLACtwodNrOLRuns}{1}
\newcommand{\OLACtwodRefAlgOne}{Spectral/Crank-Nicolson}
\newcommand{\OLACtwodRefAlgTwo}{explicit midpoint}
\newcommand{\OLACtwodNrTrainSamples}{2^{10}}
\newcommand{\OLACtwodNrValidSamples}{2^{9}}
\newcommand{\OLACtwodNrTestSamples}{2^{9}}
\newcommand{\OLACtwodNrTrainSpaceDiscr}{128}
\newcommand{\OLACtwodNrValidSpaceDiscr}{256}
\newcommand{\OLACtwodNrTestSpaceDiscr}{256}
\newcommand{\OLACtwodNrTrainTimeSteps}{1000}
\newcommand{\OLACtwodNrValidTimeSteps}{1400}
\newcommand{\OLACtwodNrTestTimeSteps}{1400}

\newcommand{\OLACtwodT}{3}
\newcommand{\OLACtwodSpaceSize}{1}
\newcommand{\OLACtwodLaplaceFactor}{\frac{2}{1000}}

\newcommand{\OLACtwoddecayRate}{4}
\newcommand{\OLACtwodoffset}{\sqrt{5000}}
\newcommand{\OLACtwodvar}{25*10^6}
\newcommand{\OLACtwodinnerdecay}{1}
\newcommand{\OLACtwodInitDistr}{\mathcal{N}(0, \OLACtwodvar(\OLACtwodoffset \id_{\initialValues} - \Delta_x)^{-\OLACtwoddecayRate} - 0.8 \id_{\initialValues})}

\newcommand{\OLReactDiffT}{1}
\newcommand{\OLReactDiffSpaceSize}{2}
\newcommand{\OLReactDiffLaplaceFactor}{\frac{5}{100}}
\newcommand{\OLReactDiffReactionRate}{2}

\newcommand{\OLReactDiffdecayRate}{4}
\newcommand{\OLReactDiffoffset}{100}
\newcommand{\OLReactDiffvar}{10^8}
\newcommand{\OLReactDiffStartVar}{0.2}
\newcommand{\OLReactDiffinnerdecay}{1}
\newcommand{\OLReactDiffInitDistr}{\mathcal{N}(0, \OLReactDiffvar(\OLReactDiffoffset \id_{\initialValues} - \Delta_x)^{-\OLReactDiffdecayRate} - 0.8 \id_{\initialValues})}

\newcommand{\OLReactDiffSpaceStep}{128}
\newcommand{\OLReactDiffParams}{(0.1, 1.3)^2}
\newcommand{\OLReactDiffBatchSize}{256}
\newcommand{\OLReactDiffILRsteps}{50}
\newcommand{\OLReactDiffEvalSteps}{400}
\newcommand{\OLReactDiffImprTol}{0.97}
\newcommand{\OLReactDiffNrOLRuns}{3}
\newcommand{\OLReactDiffNrGridRuns}{16}
\newcommand{\OLReactDiffNrOptRuns}{-}
\newcommand{\OLReactDiffRefAlgOne}{\FDM/Crank-Nicolson}
\newcommand{\OLReactDiffRefAlgTwo}{explicit midpoint}
\newcommand{\OLReactDiffNrTrainSamples}{2^{18}}
\newcommand{\OLReactDiffNrValidSamples}{2^{14}}
\newcommand{\OLReactDiffNrTestSamples}{2^{14}}
\newcommand{\OLReactDiffNrTrainSpaceDiscr}{512}
\newcommand{\OLReactDiffNrValidSpaceDiscr}{1024}
\newcommand{\OLReactDiffNrTestSpaceDiscr}{1024}
\newcommand{\OLReactDiffNrTrainTimeSteps}{1000}
\newcommand{\OLReactDiffNrValidTimeSteps}{1500}
\newcommand{\OLReactDiffNrTestTimeSteps}{1500}

\newcommand{\OLReactDiffDMarch}{(\OLReactDiffSpaceStep, 512, 1024, 512, \OLReactDiffSpaceStep)}

\makeatletter
\let\scp\@undefined 
\makeatother
\DeclarePairedDelimiterXPP{\scp}[2]{}{\langle}{\rangle}{}{#1,#2}

In this section we provide an introduction to operator learning methods.
Roughly speaking, the field of operator learning is concerned with the approximation of mappings which map functions to functions, a type of mappings usually called \emph{operators}.
In the context of \PDEs, operators naturally arise, for example, as mappings from boundary conditions or initial values of \PDEs\ to corresponding \PDE\ solutions or terminal values.
To approximately learn such operators it is necessary to develop new types of trainable models which take functions as inputs and outputs. 
We refer to such models as \emph{neural operators}.

In \cref{sect:data_driven_OL} we present a basic method for training such neural operators using evaluations of the target operator.
In \cref{sect:OL_architectures} we discuss precise definitions of different types of neural operators.
We first consider neural operators which rely on interpolation to move between discrete data and functions
and are based on classical \ANNs\ 
such as 
	fully connected feed-forward \ANNs\ (see \cref{sect:ANN_NO}) 
	and 
	\CNNs\ (see \cref{sect:CNN_NO}).
We then present a selection of recent and advanced neural operators that are directly formulated in function space without making use of interpolation 
such as 
	\IKNOs\ (see \cref{sect:IKNO}),
	\FNOs\ (see \cref{sect:FNO}),
	and
	\DeepONets\ (see \cref{sect:deepONets}).
In \cref{sect:PINO} we introduce the \PINO\ methodology for operators related to \PDEs, which combines the data-driven operator learning approach in \cref{sect:data_driven_OL} with the \PINNs\ methodology in \cref{secpinn}.
In \cref{sect:OL_refs} we survey additional operator learning methods from the literature.
Lastly, in \cref{sect:OL_simul} we present numerical simulations which illustrate the performance of the neural operators discussed in this section.

\subsection{The canonical data-driven operator learning framework}
\label{sect:data_driven_OL}

In \cref{algo:operator_learning} below we illustrate how a neural operator can be trained using the \SGD\ optimization method when evaluations of the target operator are available as training data.

\begingroup
\providecommandordefault{\indim}{d}
\providecommandordefault{\outdim}{\mathbf{d}}
\providecommandordefault{\B}{\mathcal{M}}
\providecommandordefault{\inDom}{D}
\providecommandordefault{\outDom}{\mathbf{D}}
\providecommandordefault{\Grid}{G}
\providecommandordefault{\In}{\mathcal{I}}
\providecommandordefault{\Out}{\mathcal{O}}
\providecommandordefault{\S}{\mathcal{S}}
\providecommandordefault{\nrparams}{\fd}
\providecommandordefault{\i}{i}
\providecommandordefault{\x}{x}
\providecommandordefault{\Loss}{\mathscr{L}}
\providecommandordefault{\g}{g}
\providecommandordefault{\II}{\mathbb{I}}
\providecommandordefault{\XX}{\mathbb{X}}
\providecommandordefault{\b}{m}
\providecommandordefault{\Grad}{\mathscr{G}}

\begin{algo}[General data-driven operator learning]
\label{algo:operator_learning}
Let 
	$\indim, \outdim, \B \in \N$,
	$\inDom \subseteq \R^\indim$,	
	$\outDom \subseteq \R^\outdim$,
let
	$\In \subseteq \cL^0(\inDom; \R)$,
	$\Out \subseteq \cL^0(\outDom; \R)$
be vector spaces,
let $\lossmetric{\cdot} \colon \EndValues \to [0, \infty)$ be measurable,
let 
	$\S \colon \In \to \Out$, 
	$\neuralOp \colon \R^\nrparams \times \In \to \Out$, and
	$\Loss \colon \R^\nrparams \times \In \to \R$ 
satisfy for all
	$\theta \in \R^\nrparams$,
	$\i \in \In$
that
\begin{equation}
\label{algo:operator_learning:eq0}
\begin{split} 
	\Loss(\theta, \i)
=
	\lossmetric{(\neuralOp_{\theta}(\i)) - (\S(\i))}^2,
\end{split}
\end{equation}
let $\Grad \colon \R^\nrparams \times \In \to \R^{\nrparams}$ satisfy for all
	$\theta \in \R^\nrparams$,
	$\i \in \In$
with 
	$ \Loss(\cdot, \i) \colon \R^{ \nrparams} \to \R$ is differentiable at $\theta$
that
$\Grad(\theta, \i) = (\nabla_\theta \Loss)( \theta , \i)$,
let $(\Omega, \cF, \P)$ be a probability space,
let $\II_{n, \b} \colon \Omega \to \In$, $n, \b \in \N$, be random variables,
let $(\gamma_n )_{n \in \N} \subseteq \R$, 
and let $\Theta\colon \N_0 \times \Omega \to \R^{ \nrparams}$ satisfy for all $n \in \N_0$ that 
\begin{equation}
\label{algo:operator_learning:eq1}	
	\Theta_{ n + 1 } 
=
	\Theta_n - \frac{\gamma_n}{\B} 
	\br*{
		\sum_{\b = 1}^{\B}\Grad( \Theta_n , \II_{n, \b})
	}.
\end{equation}
\end{algo}

\begin{remark}[Explanations for \cref{algo:operator_learning}]
{\sl 
We think of $\S$ as the target operator %
that we want to approximate,
we think of $\neuralOp$ as the neural operator that we want to use to approximate $\S$,
we think of $\lossmetric{\cdot}$ as a norm or seminorm on the output space $\Out$ used to measure the distance between training samples and the output of the neural operator (cf., e.g., \cref{OLReactionDiffusion:eq3,OLAC:eq3,OLBurgers:eq3} for concrete examples), 
we think of $(\II_{n, \b})_{(n, \b) \in \N^2}$ as the training samples for the optimization procedure,
and we think of $\Theta$ as an \SGD\ process for the loss $\Loss$ with generalized gradient $\Grad$, learning rates $(\gamma_n )_{n \in \N}$, constant batch size $\B$, initial values $\Theta_0$, and data $(\II_{n, \b})_{(n, \b) \in \N \times \{1, 2, \ldots, \B\}}$.
For sufficiently large $n \in \N$ we hope that 
$
	\neuralOp_{\Theta_n} \approx \S
$.

For simplicity in \cref{algo:operator_learning} we choose to describe the plain-vanilla \SGD\ method with constant learning rates for the learning process in \cref{algo:operator_learning:eq1}.
In practice, however, typically other more advanced \SGD-type optimization methods are employed.
For example, in our numerical simulations in \cref{sect:OL_simul} we use an adaptive trainig procedure based on the Adam optimizer which is described in \cite[Appendix A.2]{adanns2023}.
Moreover, in typical applications the target operator $\S$ is not known exactly and needs to be replaced with approximations of it to compute reference solutions in the loss function in \cref{algo:operator_learning:eq0}.
}
\end{remark}

\endgroup

\subsection{Neural operator architectures}
\label{sect:OL_architectures}

In this section we introduce various neural operator architectures. 
In the treatment of several types of neural operators, we will for simplicity restrict ourselves to the case of periodic functions. %
In these situations, we will make use of the modulo operation in \cref{def:mod} below and of multi-linear periodic interpolation in \cref{def:interpolator} below.

\renewcommand{\mod}[2]{\operatorname{mod}_{#1}(#2)\cfadd{def:mod}}

\cfclear
\begin{definition}[Modulo operators]
\label{def:mod}
Let $a \in \N$.
Then we denote by $\operatorname{mod}_{a} \colon \R \to [0,a)$
the function which satisfies for all 
	$b \in \R$
that
\begin{equation}
\label{T_B_D}
\begin{split} 
	\mod{a}{b}
=	
	\min \pr*{
		\{
			b - ka \in \R \colon k \in \Z
		\}
		\cap 
		[0,a)
	}.
\end{split}
\end{equation}

\end{definition}

\begingroup
\providecommandordefault{\a}{\matrixdim}
\providecommandordefault{\I}{\mathcal{I}}
\providecommandordefault{\x}{x}
\providecommandordefault{\xalt}{\mathbf{x}}
\cfclear
\begin{lemma}[Properties of multi-linear periodic interpolators]
\label{interpolation}
Let
	$\tensordim \in \N$,
	$\a_1, \a_2, \ldots, \a_\tensordim \in \N$,
	$
		\xalt = (\xalt_{\matrixindex_1, \matrixindex_2, \ldots, \matrixindex_\tensordim})_{
				(\matrixindex_1, \matrixindex_2, \ldots, \matrixindex_\tensordim) \in 
				(\bigtimes_{\dimindex = 1}^\tensordim  \{0, 1, 2, \ldots, \a_\dimindex\} )
		} 
		\in 
		\R^{(\a_1+1) \times (\a_2+1) \times \ldots \times (\a_\tensordim+1)} 
	$
and let 
	$\I \colon[0,1]^\tensordim \to \R$
satisfy for all 
	$
		y = (y_1, y_d, \ldots, y_\tensordim) 
	\in 
		[0,1]^\tensordim
	$
that
\begin{equation}
\label{interpolation:eq1}
\begin{split} 
	\I (y)
=
	\sum_{
		\substack{
			(\matrixindex_1, \matrixindex_2, \ldots, \matrixindex_\tensordim) 
			\in 
			(\bigtimes_{\dimindex = 1}^\tensordim  \{0, 1, \ldots, \a_\dimindex\} )
		}
	}
		\br*{
			\prod_{\dimindex = 1}^\tensordim 
				\pr{
					1
					-
					\abs{
						\a_\dimindex y_{\dimindex} - \matrixindex_{\dimindex}
					}
				}
			\indicator{[-1,1]}(\a_\dimindex y_{\dimindex} - \matrixindex_{\dimindex})
		}
		\xalt_{
			\matrixindex_1, \matrixindex_2, \ldots, \matrixindex_\tensordim
		}.
\end{split}
\end{equation}
Then
\begin{enumerate}[label=(\roman*)]
\item \label{interpolation:item1}
it holds for all
	$
		(\matrixindex_1, \matrixindex_2, \ldots, \matrixindex_\tensordim) 
		\in 
		(\bigtimes_{\dimindex = 1}^\tensordim  \{0, 1, \ldots, \a_\dimindex\})
	$
that
$
	I(\frac{\matrixindex_1}{\a_1}, \frac{\matrixindex_2}{\a_2}, \ldots, \frac{\matrixindex_\tensordim}{\a_\tensordim})
=
	\xalt_{
		\matrixindex_1, \matrixindex_2, \ldots, \matrixindex_\tensordim
	}
$
and
\item \label{interpolation:item2}
it holds that
$\I \in C([0,1]^\tensordim, \R)$.
\end{enumerate}
\end{lemma}

\begin{proof}[Proof of \cref{interpolation}]
\Nobs \eqref{interpolation:eq1} implies \cref{interpolation:item1,interpolation:item2}.
The proof of \cref{interpolation} is thus complete.
\end{proof}

\cfclear
\cfadd{interpolation}
\begin{definition}[Multi-linear periodic interpolators]
\label{def:interpolator}
Let
	$\tensordim \in \N$,
	$\a_1, \a_2, \ldots, \a_\tensordim \in \N$.
Then we say that $\I$ is the periodic multi-linear interpolator on $[0,1]^\tensordim$ 
	with spatial discretization $(\a_1, \a_2, \ldots, \a_\tensordim)$
if and only if it holds that
	$\I \colon \R^{\a_1 \times \a_2 \times \ldots \times \a_\tensordim} \to C([0,1]^\tensordim, \R)$ 
is the function which satisfies for all
	$
		\x = (\x_{\matrixindex_1, \matrixindex_2, \ldots, \matrixindex_\tensordim})_{
				(\matrixindex_1, \matrixindex_2, \ldots, \matrixindex_\tensordim) \in 
				(\bigtimes_{\dimindex = 1}^\tensordim  \{0, 1, 2, \ldots, \a_\dimindex-1\} )
		} 
		\in 
		\R^{\a_1 \times \a_2 \times \ldots \times \a_\tensordim} 
	$,
	$
		\xalt = (\xalt_{\matrixindex_1, \matrixindex_2, \ldots, \matrixindex_\tensordim})_{
				(\matrixindex_1, \matrixindex_2, \ldots, \matrixindex_\tensordim) \in 
				(\bigtimes_{\dimindex = 1}^\tensordim  \{0, 1, 2, \ldots, \a_\dimindex\} )
		} 
		\in 
		\R^{(\a_1+1) \times (\a_2+1) \times \ldots \times (\a_\tensordim+1)} 
	$,
	$
		y = (y_1, y_d, \ldots, y_\tensordim) 
	\in 
		[0,1]^\tensordim
	$
with
$
	\forall \, 
		(\matrixindex_1, \matrixindex_2, \ldots, \matrixindex_\tensordim) 
		\in 
		(\bigtimes_{\dimindex = 1}^\tensordim  \{0, 1, 2, \ldots, \a_\dimindex\} )
	\colon \allowbreak
		\xalt_{\matrixindex_1, \matrixindex_2, \ldots, \matrixindex_\tensordim}
		=
		x_{\mod{\a_1}{\matrixindex_1}, \mod{\a_2}{\matrixindex_2}, \ldots, \mod{\a_\tensordim}{\matrixindex_\tensordim}}
$
that
\begin{equation}
\label{T_B_D}
\begin{split} 
	(\I(x)) (y)
=
	\sum_{
		\substack{
			(\matrixindex_1, \matrixindex_2, \ldots, \matrixindex_\tensordim) 
			\in 
			(\bigtimes_{\dimindex = 1}^\tensordim  \{0, 1, \ldots, \a_\dimindex\} )
		}
	}
		\br*{
			\prod_{\dimindex = 1}^\tensordim 
				\pr{
					1
					-
					\abs{
						\a_\dimindex y_{\dimindex} - \matrixindex_{\dimindex}
					}
				}
			\indicator{[-1,1]}(\a_\dimindex y_{\dimindex} - \matrixindex_{\dimindex})
		}
		\xalt_{
			\matrixindex_1, \matrixindex_2, \ldots, \matrixindex_\tensordim
		}
\end{split}
\end{equation}
\cfload.
\end{definition}

\endgroup

\subsubsection{Fully-connected feed-forward neural operators}
\label{sect:ANN_NO}

In this section we introduce basic neural operator models based on vanilla fully-connected feed-forward \ANNs.
In order to use fully-connected feed-forward \ANNs\ to define models which take functions as inputs and outputs, we will employ evaluations on grids and multi-linear interpolation to move between discrete data and functions.
We will first introduce fully-connected feed-forward \ANN\ models which take multi-dimensional matrices as inputs and outputs (see \cref{def:FCFFANNModels} below) and then use these models to define corresponding
neural operator models (see \cref{def:FCFFANNOperators} below).
Roughly speaking, we think of the multi-dimensional matrices that serve as inputs and outputs of the fully-connected feed-forward \ANN\ models as discretizations of functions on a grid.

\cfclear
\begingroup
\providecommand{\a}{}
\renewcommand{\a}{\matrixdim}
\providecommandordefault{\Flatten}{\mathcal{F}}
\providecommandordefault{\A}{A}
\begin{definition}[Fully-connected feed-forward \ANN\ models]
\label{def:FCFFANNModels}
Let
	$\tensordim, L \in \N$,
	$l_0,l_1,\ldots, l_L \in \N$, 
	$\a_1, \a_2, \ldots, \a_\tensordim \in \N$,
	$\activation \in C(\R, \R)$
satisfy 
	$l_0 = \a_1 \a_2 \cdots \a_\tensordim  = l_L$
and let 
	$\Flatten \colon \R^{\a_1 \times \a_2 \times \ldots \times \a_\tensordim} \to \R^{\a_1  \a_2  \cdots \a_\tensordim}$
satisfy for all
	$
		\A
	=
		(\A_{\matrixindex_1, \matrixindex_2, \ldots, \matrixindex_\tensordim})_{
				(\matrixindex_1, \matrixindex_2, \ldots, \matrixindex_\tensordim) \in 
				(\bigtimes_{\dimindex = 1}^\tensordim  \{1, 2, \ldots, \a_\dimindex\} )
			} 
	 \in 
	 	\R^{\a_1 \times \a_2 \times \ldots \times \a_\tensordim}
	 $,
	$x = (x_1, \ldots, x_{\a_1  \a_2  \cdots \a_\tensordim}) \in \R^{\a_1  \a_2  \cdots \a_\tensordim}$
with 
	$\forall \, (\matrixindex_1, \matrixindex_2, \ldots, \matrixindex_\tensordim) \in (\bigtimes_{\dimindex = 1}^\tensordim  \{1, 2, \ldots, \a_\dimindex\} )\colon$
\begin{equation}
\label{T_B_D}
\begin{split} 
	\A_{\matrixindex_1, \matrixindex_2, \ldots, \matrixindex_\tensordim}
=
	x_{
		(\matrixindex_1-1) \a_2 \a_3 \cdots \a_\tensordim +
		(\matrixindex_2-1) \a_3 \a_4 \cdots \a_\tensordim +
		\ldots +
		(\matrixindex_{\tensordim-1}-1) \a_\tensordim +
		\matrixindex_{\tensordim}
	}
\end{split}
\end{equation}
that 
$
	\Flatten(\A) = x
$.
Then we call
	$\neuralOp$
the fully-connected feed-forward \ANN\ model 
	with input-output shape  $(\a_1, \a_2, \ldots, \a_\tensordim)$, 
	architecture $(l_0,l_1,\ldots, l_L )$, and
	activation function $\activation$
if and only if it holds that
\begin{equation}
\label{T_B_D}
\begin{split} \textstyle
		\neuralOp
	\colon
	\pr*{
		\bigtimes_{{\layerindex} = 1}^L 
	      	\R^{l_{{\layerindex}} \times l_{{\layerindex}-1}} \times \R^{l_{{\layerindex}}}
	}
	\times
		\R^{\a_1 \times \a_2 \times \ldots \times \a_\tensordim}
	\to
		\R^{\a_1 \times \a_2 \times \ldots \times \a_\tensordim}
\end{split}
\end{equation}
is the function which satisfies for all
	$
		\Phi 
	=
		(
			(
				W_{\layerindex}, B_{\layerindex}
			)
		)_{\layerindex \in \{1, 2, \ldots, L\}}
	\in \allowbreak
		\pr[\big]{
			\bigtimes_{{\layerindex} = 1}^L \R^{l_{{\layerindex}} \times l_{{\layerindex}-1}} \times \R^{l_{{\layerindex}}}
		}
	$,
	$\A \in \R^{\a_1 \times \a_2 \times \ldots \times \a_\tensordim}$,
	$x_0 \in \R^{l_0}$,
	$x_1 \in \R^{l_1}$,
	$\dots$,
	$x_L \in \R^{l_L}$
with
	$\forall \, {\layerindex} \in \{1,2,\dots,L-1\} \colon$
\begin{equation}
\label{FCFFANNModels:eq1}
\begin{split} 
	x_0
=
	\Flatten(\A),
\qquad
	x_{{\layerindex}} 
=
	\multdim_{
		a
	}
	(
		W_{{\layerindex}} x_{\layerindex-1} + B_{\layerindex}
	),
\qandq 
	x_{{L}}
=
	W_{{L}} x_{L-1} + B_{L}
\end{split}
\end{equation}
that
$
	\neuralOp_{\Phi}(\A)
=
	\Flatten^{-1}(x_L)
$
\cfload.
\end{definition}

\begin{remark}[Explanations for \cref{def:FCFFANNModels}]
\label{rem:FCFFANNModels}
In \eqref{FCFFANNModels:eq1}
we think of $\Phi$ as the parameters of the \ANN,
we think of $\A$ as the input of the \ANN,
we think of $x_0$ as a reshaped version of the input $\A$,
for every $\layerindex \in \{1,2,\dots,L-1\}$ we think of $x_\layerindex$ as the output of the $\layerindex$-th layer of the \ANN\ given the input $\A$,
and
we think of the output $\Flatten^{-1}(x_L)$ of the \ANN\ as a reshaped version of the output of $x_L$ of the last layer of the \ANN.
\end{remark}

\endgroup

\cfclear
\begingroup
\providecommandordefault{\a}{\matrixdim}
\providecommandordefault{\neuralOplocal}{\mathcal{N}}
\providecommandordefault{\I}{\mathcal{I}}
\begin{definition}[Periodic fully-connected feed-forward neural operators]
\label{def:FCFFANNOperators}
Let
	$\tensordim, L \in \N$,
	$l_0,l_1,\ldots, l_L \in \N$, 
	$\a_1, \a_2, \ldots, \a_\tensordim \in \N$,
	$\activation \in C(\R, \R)$
satisfy 
	$l_0 = \a_1 \a_2 \cdots \a_\tensordim  = l_L$,
let\cfadd{def:interpolator}
	$\I \colon \R^{\a_1 \times \a_2 \times \ldots \times \a_\tensordim} \to C([0,1]^\tensordim, \R)$ 
be the periodic multi-linear interpolator on $[0,1]^\tensordim$
with spatial discretization $(\a_1, \a_2, \ldots, \a_\tensordim)$, 
and
let\cfadd{def:FCFFANNModels}
	$\neuralOplocal$
be the fully-connected feed-forward \ANN\ model 
	with input-output shape  $(\a_1, \a_2, \ldots, \a_\tensordim)$, 
	architecture $(l_0,l_1,\ldots, l_L )$, and
	activation function $\activation$
\cfload.
Then we call
	$\neuralOp$
the periodic fully-connected feed-forward neural operator on $[0,1]^\tensordim$ with
	spatial discretization $(\a_1, \a_2, \ldots, \a_\tensordim)$,
	architecture $(l_0,l_1,\ldots, l_L )$, and
	activation function $\activation$
if and only if it holds that
\begin{equation}
\begin{split} \textstyle
		\neuralOp
	\colon
	\pr[\big]{
		\bigtimes_{{\layerindex} = 1}^L 
	      	\R^{l_{{\layerindex}} \times l_{{\layerindex}-1}} \times \R^{l_{{\layerindex}}}
	}
	\times
		C([0,1]^{\tensordim}, \R)
	\to
		C([0,1]^{\tensordim}, \R)
\end{split}
\end{equation}
is the function which satisfies for all
	$f \in C([0,1]^\tensordim, \R)$,
	$
			\Phi 
		=
			(
				(
					W_{\layerindex}, B_{\layerindex}
				)
			)_{\layerindex \in \{1, 2, \ldots, L\}}
		\in \allowbreak
				(\bigtimes_{{\layerindex} = 1}^L \R^{l_{{\layerindex}} \times l_{{\layerindex}-1}} \times \R^{l_{{\layerindex}}})
		$
that
\begin{equation}
	\neuralOp_{\Phi}(f)
=
	\I\pr[\Big]{
		\neuralOplocal_{\Phi}\pr[\Big]{
			\pr[\big]{
				f(
					\nicefrac{\matrixindex_1}{\a_1},
					\nicefrac{\matrixindex_2}{\a_2}, \ldots,
					\nicefrac{\matrixindex_\tensordim}{\a_\tensordim}
				)
			}_{
				(\matrixindex_1, \matrixindex_2, \ldots, \matrixindex_\tensordim) \in 
				(\bigtimes_{\dimindex = 1}^\tensordim  \{0, 1, \ldots, \a_\dimindex-1\} )
			} 
		}
	}.
\end{equation}
\end{definition}
\endgroup

\subsubsection{Architectures based on convolutional neural networks}
\label{sect:CNN_NO}

In this section, we illustrate with two basic examples how \CNNs\ can be used to construct neural operators.
Since \CNNs\ make use of the spatial structure of their input, they are a natural approach in the context of operator learning, in particular when the action of the neural operator should be based on the values of the input function sampled on a grid.
For more background on \CNNs, cf., e.g., \cite[Chapter 9]{Goodfellow2016}.

\paragraph{Simple periodic convolutional neural operators}
\label{sect:periodic_cnn}

In this section we present a simple example of a neural operator architecture based on discrete periodic convolutions (cf.\ \cref{def:convolution}).
Roughly speaking, the discrete periodic convolution operation allows to define \CNNs\ with the property that the shape of the feature maps stays the same on each layer (cf.\ \cref{rem:perdiodicCNN}).
Like in the case of fully-connected feed-forward \ANNs\ in \cref{sect:ANN_NO},
we will first introduce \CNN\ models which take multi-dimensional matrices as inputs and outputs (see \cref{def:perdiodicCNN} below) and then use these models to define corresponding neural operator models (see \cref{def:periodicCNO} below).
Such simple periodic neural operators were employed by \cite{Khoo2021} as a part of their proposed operator learning methodology.

\begingroup
\providecommand{\a}{}
\renewcommand{\a}{a}
\begin{definition}[Discrete periodic convolutions]
\label{def:convolution}
Let
	$d \in \N$,
	$\a_1, \a_2, \ldots, \a_\tensordim,  w_1, w_2, \ldots, w_\tensordim \allowbreak \in \N_0$,
	$
		A = 
		(A_{\matrixindex_1, \matrixindex_2, \ldots, \matrixindex_\tensordim})_{
			(\matrixindex_1, \matrixindex_2, \ldots, \matrixindex_\tensordim) \in 
			(\bigtimes_{\dimindex = 1}^\tensordim  \{0, 1, \ldots, \a_\dimindex-1\} )
		} 
		\in \R^{\a_1 \times \a_2 \times \ldots \times \a_\tensordim}
	$,
	$
		W = \allowbreak 
		(W_{\matrixindexalt_1, \matrixindexalt_2, \ldots, \matrixindexalt_\tensordim})_{
			(\matrixindexalt_1, \matrixindexalt_2, \ldots, \matrixindexalt_\tensordim) \in 
			(\bigtimes_{\dimindex = 1}^\tensordim  \{-w_\dimindex, -w_\dimindex + 1, \ldots, w_\dimindex\})
		} 
		\in \R^{(2w_1+1) \times (2w_2+1) \times \ldots \times (2w_\tensordim+1)}
	$
satisfy for all 
	$ \dimindex \in \{1, 2, \ldots, \tensordim \}$
that
\begin{equation}
\begin{split} 
	\a_\dimindex \geq 2 w_\dimindex + 1.
\end{split}
\end{equation}
Then
we denote by 
$
	A \convolution W =
		((A \convolution W)_{\matrixindex_1, \matrixindex_2, \ldots, \matrixindex_\tensordim})_{
			(\matrixindex_1, \matrixindex_2, \ldots, \matrixindex_\tensordim) \in 
			(\bigtimes_{\dimindex = 1}^\tensordim  \{0, 1, \ldots, \a_\dimindex-1\})
		} 
	\in \R^{\a_1 \times \a_2 \times \ldots \times \a_\tensordim}
$ 
the tensor which satisfies for all
	$(\matrixindex_1, \matrixindex_2, \ldots, \matrixindex_\tensordim) \in \bigtimes_{\dimindex = 1}^\tensordim  \{0, 1, \ldots, \a_\dimindex-1\}$
that
\begin{equation}
\begin{split} 
	(A \convolution W)_{\matrixindex_1, \matrixindex_2, \ldots, \matrixindex_\tensordim}
=
	\sum_{\matrixindexalt_1 = -w_1}^{w_1} 
	\sum_{\matrixindexalt_2 = -w_2}^{w_2} 
	\ldots 
	\sum_{\matrixindexalt_\tensordim = -w_\tensordim}^{w_\tensordim}
		A_{
			\mod{\a_1}{\matrixindex_1 + \matrixindexalt_1}, 
			\mod{\a_2}{\matrixindex_2 + \matrixindexalt_2}, \ldots,
			\mod{\a_\tensordim}{\matrixindex_\tensordim + \matrixindexalt_\tensordim} 
		} 
		W_{\matrixindexalt_1, \matrixindexalt_2, \ldots, \matrixindexalt_\tensordim}
\end{split}
\end{equation}\cfclear\cfadd{def:mod}\cfload.
\end{definition}
\endgroup

\begingroup
\providecommand{\d}{}
\renewcommand{\d}{\matrixdim}
\begin{definition}[One tensor]
	\label{def:onetensor}
	Let
		$\tensordim \in \N$,
		$\d_1, \d_2, \ldots, \d_\tensordim \in \N$.
	Then we denote by 
	$
		\onetensor^{\d_1, \d_2, \ldots, \d_\tensordim} 
	= 
		(\onetensor^{\d_1, \d_2, \ldots, \d_\tensordim}_{i_1, i_2, \ldots, i_\tensordim})_{
			(i_1, i_2, \ldots, i_\tensordim) 
			\in 
			(\bigtimes_{\dimindex = 1}^\tensordim \{1, 2, \dots, \d_\dimindex\})
		} 
		\in \R^{\d_1 \times \d_2 \times \ldots \times \d_\tensordim}
	$
	the tensor which satisfies for all
		$i_1 \in \{1, 2, \dots, \d_1 \}$,
		$i_2 \in \{1, 2, \dots, \d_2 \}$,
		$\dots$,
		$i_\tensordim \in \{1, 2, \dots, \d_\tensordim\}$
	that
	\begin{equation}
	\begin{split} 
		\onetensor^{\d_1, \d_2, \ldots, \d_\tensordim}_{i_1, i_2, \ldots, i_\tensordim}
	=
		1.
	\end{split}
	\end{equation}
\end{definition}
\endgroup

\cfclear
\begingroup
\providecommand{\a}{}
\renewcommand{\a}{\matrixdim}
\begin{definition}[Simple periodic \CNN\ models]
\label{def:perdiodicCNN}
Let
	$\tensordim, L \in \N$,
	$l_0,l_1,\ldots, l_L \in \N$, 
	$\a_1, \a_2, \ldots, \a_\tensordim \in \N$,
	$(w_{{\layerindex}, \dimindex})_{({\layerindex}, \dimindex) \in \{1, 2, \ldots, L\} \times \{1, 2, \ldots, \tensordim\}} \subseteq \N_0$,
	$\activation \in C(\R, \R)$
satisfy for all 
	$\layerindex \in \{1, 2, \ldots, L\}$,
	$ \dimindex \in \{1, 2, \ldots, \tensordim \}$
that
	$l_0 = 1 = l_L$	
and
$
	\a_\dimindex \geq 2 w_{\layerindex, \dimindex} + 1
$.
Then we call
	$\neuralOp$
the simple periodic \CNN\ model with
	input-output shape $(\a_1, \a_2, \ldots, \a_\tensordim)$,
	channel structure $(l_0,l_1,\ldots, l_L )$,
	kernel sizes $(w_{{\layerindex}, \dimindex})_{({\layerindex}, \dimindex) \in \{1, 2, \ldots, L\} \times \{1, 2, \ldots, \tensordim\}}$, and
	activation function $\activation$
if and only if it holds that
\begin{equation}
\label{T_B_D}
\begin{split} \textstyle
		\neuralOp
	\colon
	\pr[\Big]{
		\bigtimes_{{\layerindex} = 1}^L 
	      	\pr[\big]{\R^{ 
	      		(2w_{{\layerindex}, 1}+1) \times (2w_{{\layerindex}, 2}+1) \times \ldots \times (2w_{{\layerindex}, \tensordim}+1)
	      	}}^{l_{{\layerindex}} \times l_{{\layerindex}-1}} \times \R^{l_{{\layerindex}}}
	}
	\times
		\R^{\a_1 \times \a_2 \times \ldots \times \a_\tensordim}
	\to
		\R^{\a_1 \times \a_2 \times \ldots \times \a_\tensordim}
\end{split}
\end{equation}
is the function which satisfies for all
	$
		\Phi 
	=
		(
			(
				(W_{\layerindex,n,m})_{(n,m) \in \{1, 2, \ldots, l_\layerindex\} \times \{1, 2, \ldots, l_{\layerindex-1}\}}
				, \allowbreak 
				(B_{\layerindex,n})_{n \in \{1, 2, \ldots, l_\layerindex\}}
			)
		)_{\layerindex \in \{1, 2, \ldots, L\}}
	\in \allowbreak
			\bigtimes_{{\layerindex} = 1}^L ((\R^{w_{{\layerindex}, 1} \times w_{{\layerindex}, 2} \times \ldots \times w_{{\layerindex}, \tensordim}})^{l_{{\layerindex}} \times l_{{\layerindex}-1}} \times \R^{l_{{\layerindex}}})
	$,
	$x_0 = (x_{0,1}, \ldots, x_{0,l_{0}})  \in (\R^{ \a_{1} \times \a_{2} \times \ldots \times \a_{ \tensordim}})^{l_0}$,
	$x_1 = (x_{1,1}, \ldots, x_{1,l_{1}}) \in (\R^{ \a_{1} \times \a_{2} \times \ldots \times \a_{ \tensordim}})^{l_1}$,
	$\dots$,
	$x_{L-1} = (x_{{L-1},1}, \ldots, x_{{L-1},l_{{L-1}}}) \in (\R^{\a_{1} \times \a_{2} \times \ldots \times \a_{ \tensordim}})^{l_{L-1}}$
with
	$\forall \, {\layerindex} \in \{1,2,\dots,L-1\}$, $n \in \{1,2,\dots,l_{{\layerindex}}\} \colon \allowbreak
	x_{{\layerindex},n} 
=
	\multdim_{
		\activation
	}
	\bigl( \allowbreak
		B_{\layerindex, n} \onetensor^{\a_{1}, \a_{2}, \ldots, \a_{\tensordim}} + \sum_{m = 1}^{l_{{\layerindex}-1}} x_{\layerindex-1,m} \convolution W_{{\layerindex},n,m}
	\bigr)
$
that
\begin{equation} 
\label{perdiodicCNN:eq1}\textstyle
	\neuralOp_{\Phi}(x_{0,1})
=
	B_{L, 1} \onetensor^{\a_{1}, \a_{2}, \ldots, \a_{\tensordim}} 
	+ 
	\sum_{m = 1}^{l_{{L}-1}} x_{L-1,m} \convolution W_{{L},1,m}
\end{equation}
\cfload.
\end{definition}

\begin{remark}[Explanations for \cref{def:perdiodicCNN}]
\label{rem:perdiodicCNN}
In \eqref{perdiodicCNN:eq1} 
we think of $\Phi$ as the parameters of the \CNN,
we think of $x_0$ as the input of the \CNN,
and for every $\layerindex \in \{1,2,\dots,L-1\}$ we think of $x_\layerindex$ as the output of the $\layerindex$-th layer of the \CNN\ given the input $x_0$ consisting of the $l_{{\layerindex}}$ feature maps $x_{\layerindex,1}, x_{\layerindex,2}, \ldots, x_{\layerindex,l_{\layerindex}} \in \R^{\a_{1} \times \a_{2} \times \ldots \times \a_{ \tensordim}}$.
\end{remark}
\endgroup

\cfclear
\begingroup
\providecommand{\a}{}
\renewcommand{\a}{\matrixdim}
\providecommandordefault{\neuralOplocal}{\mathcal{N}}
\providecommandordefault{\interpol}{\mathcal{I}}
\begin{definition}[Simple periodic convolutional neural operators]
\label{def:periodicCNO}
Let
	$\tensordim, L \in \N$,
	$l_0,l_1,\ldots, l_L \in \N$, 
	$\a_1, \a_2, \ldots, \a_\tensordim \in \N$,
	$(w_{{\layerindex}, \dimindex})_{({\layerindex}, \dimindex) \in \{1, 2, \ldots, L\} \times \{1, 2, \ldots, \tensordim\}} \subseteq \N_0$,
	$\activation \in C(\R, \R)$
satisfy for all 
	$\layerindex \in \{1, 2, \ldots, L\}$,
	$ \dimindex \in \{1, 2, \ldots, \tensordim \}$
that
	$l_0 = 1 = l_L$	
and
$
	\a_\dimindex \geq 2 w_{\layerindex, \dimindex} + 1
$,
let\cfadd{def:interpolator}
	$\interpol \colon \R^{\a_1 \times \a_2 \times \ldots \times \a_\tensordim} \to C([0,1]^\tensordim, \R)$ 
be the periodic multi-linear interpolator on $[0,1]^\tensordim$
with spatial discretization $(\a_1, \a_2, \ldots, \a_\tensordim)$,
and
let\cfadd{def:perdiodicCNN}
	$\neuralOplocal$
the simple periodic \CNN\ model with
	input-output shape $(\a_1, \a_2, \ldots, \a_\tensordim)$,
	channel structure $(l_0,l_1,\ldots, l_L )$,
	kernel sizes $(w_{{\layerindex}, \dimindex})_{({\layerindex}, \dimindex) \in \{1, 2, \ldots, L\} \times \{1, 2, \ldots, \tensordim\}}$, and
	activation function $\activation$
\cfload.
Then we call
	$\neuralOp$
the simple periodic convolutional neural operator on $[0,1]^\tensordim$ with
	spatial discretization $(\a_1, \a_2, \ldots, \a_\tensordim)$,
	channel structure $(l_0,l_1,\ldots, l_L )$,
	kernel sizes $(w_{{\layerindex}, \dimindex})_{({\layerindex}, \dimindex) \in \{1, 2, \ldots, L\} \times \{1, 2, \ldots, \tensordim\}}$, and
	activation function $\activation$
if and only if it holds that
\begin{equation}
\label{T_B_D}
\begin{split} \textstyle
		\neuralOp
	\colon
	\pr[\Big]{
		\bigtimes_{{\layerindex} = 1}^L 
	      	\pr[\big]{\R^{ 
	      		(2w_{{\layerindex}, 1}+1) \times (2w_{{\layerindex}, 2}+1) \times \ldots \times (2w_{{\layerindex}, \tensordim}+1)
	      	}}^{l_{{\layerindex}} \times l_{{\layerindex}-1}} \times \R^{l_{{\layerindex}}}
	}
	\times
		C([0,1]^\tensordim, \R)
	\to
		C([0,1]^\tensordim, \R)
\end{split}
\end{equation}
is the function which satisfies for all
	$f \in C([0,1]^\tensordim, \R)$,
	$
		\Phi 
	\in \allowbreak
		\pr[\big]{
			\bigtimes_{{\layerindex} = 1}^L 
				(\R^{
					w_{{\layerindex}, 1} \times 
					w_{{\layerindex}, 2} \times 
					\ldots \times w_{{\layerindex}, \tensordim}}
				)^{l_{{\layerindex}} \times l_{{\layerindex}-1}} 
				\times 
				\R^{l_{{\layerindex}}}
		}
	$
that
\begin{equation}
	\neuralOp_{\Phi}(f)
=
	\interpol\pr[\Big]{
		\neuralOplocal_{\Phi}\pr[\Big]{
			\pr[\big]{
				f(
					\nicefrac{\matrixindex_1}{\a_1},
					\nicefrac{\matrixindex_2}{\a_2}, \ldots,
					\nicefrac{\matrixindex_\tensordim}{\a_\tensordim}
				)
			}_{
				(\matrixindex_1, \matrixindex_2, \ldots, \matrixindex_\tensordim) \in 
				(\bigtimes_{\dimindex = 1}^\tensordim  \{0, 1, \ldots, \a_\dimindex-1\} )
			} 
		}
	}.
\end{equation}
\end{definition}
\endgroup

\paragraph{Simple convolutional neural operators with encoder-decoder architecture}
\label{sect:dec_enc_cnn}

A common technique for designing \CNNs\ is to use so-called encoder-decoder architectures. 
This means that the network architecture is divided into two segments: an encoder segment followed by a decoder segment. 
Roughly speaking, in the encoder part, the shape of the feature maps is successively contracted while the number of channels (i.e.\ the number of feature maps) is increased. 
In the decoder part, this operation is reversed, in the sense that successively the shape of the feature maps is expanded while the number of channels is decreased. 
For an explanation of the encoder-decoder architecture in a mathematical context, see \cref{rem:DecEncCNN}.

In this section we introduce a simple way to achieve an encoder-decoder architecture %
by employing full stride convolutions (cf.\ \cref{def:convolution_stride}) to contract the shape of the feature maps, and 
utilizing transposed convolutions (cf.\ \cref{def:deconvolution_trans} and \cite{Zeiler2010}) to expand the shape of the feature maps.
Similar to \cref{sect:periodic_cnn,sect:ANN_NO}, we will first introduce \CNN\ models which take multi-dimensional matrices as inputs and outputs (see \cref{def:DecEncCNN} below) and then use these models to define corresponding neural operator models (see \cref{def:DecEncCNO} below).
This approach is very similar to the \CNNs\ used in the operator learning method in \cite{Guo2016}.

\newcommand{\convdim}{b}
\newcommand{\strideconvolution}{\cfadd{def:convolution_stride}{\,\underline{\ast}\,}}

\cfclear
\begingroup
\providecommand{\a}{}
\renewcommand{\a}{a}
\providecommand{\b}{}
\renewcommand{\b}{b}
\begin{definition}[Discrete convolutions with full stride]
	\label{def:convolution_stride}
	Let
		$d \in \N$,
		$\a_1, \a_2, \ldots, \a_\tensordim,  \allowbreak
		\b_1, \b_2, \ldots, \b_\tensordim, \allowbreak
		w_1, w_2, \ldots, w_\tensordim \allowbreak \in \N$,
		$
			A = 
			(A_{\matrixindex_1, \matrixindex_2, \ldots, \matrixindex_\tensordim})_{
				(\matrixindex_1, \matrixindex_2, \ldots, \matrixindex_\tensordim) \in 
				(\bigtimes_{\dimindex = 1}^\tensordim  \{0, 1, \ldots, \a_\dimindex - 1\} )
			} 
			\in \R^{\a_1 \times \a_2 \times \ldots \times \a_\tensordim}
		$,
		$
			W =  \linebreak
			(W_{\matrixindexalt_1, \matrixindexalt_2, \ldots, \matrixindexalt_\tensordim})_{
				(\matrixindexalt_1, \matrixindexalt_2, \ldots, \matrixindexalt_\tensordim) \in 
				(\bigtimes_{\dimindex = 1}^\tensordim  \{0, 1, \ldots, w_\dimindex-1\})
			} 
			\in \R^{w_1 \times w_2 \times \ldots \times w_\tensordim}
		$
	satisfy for all 
		$ \dimindex \in \{1, 2, \ldots, \tensordim \}$
	that
	\begin{equation}
	\label{T_B_D}
	\begin{split} 
		\b_\dimindex = \a_\dimindex / w_\dimindex .
	\end{split}
	\end{equation}
	Then
	we denote by 
	$
		A \strideconvolution W =
			((A \strideconvolution W)_{\matrixindex_1, \matrixindex_2, \ldots, \matrixindex_\tensordim})_{
				(\matrixindex_1, \matrixindex_2, \ldots, \matrixindex_\tensordim) \in 
				(\bigtimes_{\dimindex = 1}^\tensordim  \{0, 1, \ldots, \b_\dimindex-1\})
			} 
		\in \R^{\b_1 \times \b_2 \times \ldots \times \b_\tensordim}
	$ 
	the tensor which satisfies for all
		$(\matrixindex_1, \matrixindex_2, \ldots, \matrixindex_\tensordim) \in \times_{\dimindex = 1}^\tensordim  \{0, 1, \ldots, \b_\dimindex-1\}$
	that
	\begin{equation}
	\begin{split} 
		(A \strideconvolution W)_{\matrixindex_1, \matrixindex_2, \ldots, \matrixindex_\tensordim}
	=
		\sum_{\matrixindexalt_1 = 0}^{w_1-1} 
		\sum_{\matrixindexalt_2 = 0}^{w_2-1} 
		\ldots 
		\sum_{\matrixindexalt_\tensordim = 0}^{w_\tensordim-1}
			A_{
				{\matrixindex_1 w_1 + \matrixindexalt_1}, 
				{\matrixindex_2 w_2 + \matrixindexalt_2}, \ldots,
				{\matrixindex_\tensordim w_\tensordim + \matrixindexalt_\tensordim} 
			} 
			W_{\matrixindexalt_1, \matrixindexalt_2, \ldots, \matrixindexalt_\tensordim}.
	\end{split}
	\end{equation}
\end{definition}
\endgroup

\newcommand{\deconvolution}{\cfadd{def:deconvolution_trans}{\,\overline{\ast}\,}}

\cfclear
\begingroup
\providecommand{\a}{}
\renewcommand{\a}{a}
\providecommand{\b}{}
\renewcommand{\b}{b}
\begin{definition}[Discrete transposed convolutions with full stride]
	\label{def:deconvolution_trans}
	Let
		$d \in \N$,
		$\a_1, \a_2, \ldots, \a_\tensordim,  \allowbreak
		\b_1, \b_2, \ldots, \b_\tensordim, \allowbreak
		w_1, w_2, \ldots, w_\tensordim \allowbreak \in \N$,
		$
			B = 
			(B_{\matrixindex_1, \matrixindex_2, \ldots, \matrixindex_\tensordim})_{
				(\matrixindex_1, \matrixindex_2, \ldots, \matrixindex_\tensordim) \in 
				(\bigtimes_{\dimindex = 1}^\tensordim  \{0, 1, \ldots, \a_\dimindex-1\} )
			} 
			\in \R^{\b_1 \times \b_2 \times \ldots \times \b_\tensordim}
		$,
		$
			W =  \linebreak
			(W_{\matrixindexalt_1, \matrixindexalt_2, \ldots, \matrixindexalt_\tensordim})_{
				(\matrixindexalt_1, \matrixindexalt_2, \ldots, \matrixindexalt_\tensordim) \in 
				(\bigtimes_{\dimindex = 1}^\tensordim  \{0, 1, \ldots, w_\dimindex-1\})
			} 
			\in \R^{w_1 \times w_2 \times \ldots \times w_\tensordim}
		$
	satisfy for all 
		$ \dimindex \in \{1, 2, \ldots, \tensordim \}$
	that
	\begin{equation}
	\label{T_B_D}
	\begin{split} 
		\a_\dimindex = \b_\dimindex  w_\dimindex .
	\end{split}
	\end{equation}
	Then
	we denote by 
	$
		B \deconvolution W =
			((B \deconvolution W)_{\matrixindex_1, \matrixindex_2, \ldots, \matrixindex_\tensordim})_{
				(\matrixindex_1, \matrixindex_2, \ldots, \matrixindex_\tensordim) \in 
				(\bigtimes_{\dimindex = 1}^\tensordim  \{0, 1, \ldots, \a_\dimindex-1\})
			} 
		\in \R^{\a_1 \times \a_2 \times \ldots \times \a_\tensordim}
	$ 
	the tensor which satisfies for all
		$(\matrixindex_1, \matrixindex_2, \ldots, \matrixindex_\tensordim) \in \times_{\dimindex = 1}^\tensordim  \{0, 1, \ldots, \a_\dimindex-1\}$,
		$(\matrixindexalt_1, \matrixindexalt_2, \ldots, \matrixindexalt_\tensordim) \in \times_{\dimindex = 1}^\tensordim  \{0, 1, \ldots, \b_\dimindex-1\}$
	that
	\begin{equation}
	\begin{split} 
		(B \deconvolution W)_{
			{\matrixindex_1 w_1 + \matrixindexalt_1}, 
			{\matrixindex_2 w_2 + \matrixindexalt_2}, \ldots,
			{\matrixindex_\tensordim w_\tensordim + \matrixindexalt_\tensordim} 
		} 
	=
		B_{\matrixindex_1, \matrixindex_2, \ldots, \matrixindex_\tensordim}
		W_{\matrixindexalt_1, \matrixindexalt_2, \ldots, \matrixindexalt_\tensordim}.
	\end{split}
	\end{equation}
\end{definition}
\endgroup

\cfclear
\begingroup
\providecommand{\a}{}
\renewcommand{\a}{\matrixdim}
\providecommand{\Wup}{}
\renewcommand{\Wup}{\mathbf{W}}
\providecommand{\Wdown}{}
\renewcommand{\Wdown}{{W}}
\providecommand{\Bup}{}
\renewcommand{\Bup}{\mathbf{B}}
\providecommand{\Bdown}{}
\renewcommand{\Bdown}{{B}}
\providecommand{\Phiup}{}
\renewcommand{\Phiup}{\mathbf{\Phi}}
\providecommand{\Phidown}{}
\renewcommand{\Phidown}{{\Phi}}
\providecommand{\xup}{}
\renewcommand{\xup}{\mathbf{x}}
\providecommand{\xdown}{}
\renewcommand{\xdown}{{x}}
\begin{definition}[Simple \CNN\ models with encoder-decoder architecture]
\label{def:DecEncCNN}
Let
	$\tensordim, L \in \N$,
	$l_0,l_1,\ldots, l_L \in \N$, 
	$(\a_{{\layerindex}, \dimindex})_{({\layerindex}, \dimindex) \in \{0, 1, \ldots, L\} \times \{1, 2, \ldots, \tensordim\}} \subseteq \N$,	
	$(w_{{\layerindex}, \dimindex})_{({\layerindex}, \dimindex) \in \{1, 2, \ldots, L\} \times \{1, 2, \ldots, \tensordim\}} \subseteq \N$,
	$\activation \in C(\R, \R)$
satisfy for all 
	$\layerindex \in \{1, 2, \ldots, L\}$,
	$ \dimindex \in \{1, 2, \ldots, \tensordim \}$
that
	$l_0 = 1$	
and
\begin{equation}
\label{T_B_D}
\begin{split} 
	\a_{\layerindex, \dimindex}
=
	\a_{\layerindex-1, \dimindex }/
	w_{\layerindex-1, \dimindex }.
\end{split}
\end{equation}
Then we call
	$\neuralOp$
the simple \CNN\ model with encoder-decoder architecture
	with 
	input-output shape $(\a_{0, 1}, \a_{0, 2}, \ldots, \a_{0, \tensordim})$,
	channel structure $(l_0,l_1,\ldots, l_{L-1}, l_L, l_{L-1}, \ldots, l_1, l_0)$, 
	kernel sizes \linebreak $(w_{{\layerindex}, \dimindex})_{({\layerindex}, \dimindex) \in \{1, 2, \ldots, L\} \times \{1, 2, \ldots, \tensordim\}} $,
	and
	activation function $\activation$
if and only if it holds that
\begin{multline}
\label{T_B_D} \textstyle
	\neuralOp
\colon
	\pr[\bigg]{
		\pr[\Big]{
			\bigtimes_{{\layerindex} = 1}^L 
		      	(\R^{ 
		      		w_{{\layerindex}, 1} \times w_{{\layerindex}, 2} \times \ldots \times w_{{\layerindex}, \tensordim}
		      	})^{l_{{\layerindex}} \times l_{{\layerindex}-1}} \times \R^{l_{{\layerindex}}}
		}
		\times \\ \textstyle
		\pr[\Big]{
			\bigtimes_{{\layerindex} = 1}^L
		      	(\R^{ 
		      		w_{{\layerindex}, 1} \times w_{{\layerindex}, 2} \times \ldots \times w_{{\layerindex}, \tensordim}
		      	})^{l_{{\layerindex}-1} \times l_{{\layerindex}}} 
		      	\times 
		      	\R^{l_{{\layerindex}-1}}
		}
	}
	\times
		\R^{\a_{0, 1} \times \a_{0, 2} \times \ldots \times \a_{0, \tensordim}}
	\to
		\R^{\a_{0, 1} \times \a_{0, 2} \times \ldots \times \a_{0, \tensordim}}
\end{multline}
is the function which satisfies for all
	$
		\Phidown
	=
		(
			(
				(\Wdown_{\layerindex,n,m})_{(n,m) \in \{1, 2, \ldots, l_\layerindex\} \times \{1, 2, \ldots, l_{\layerindex-1}\}}
				, \allowbreak 
				(\Bdown_{\layerindex,n})_{n \in \{1, 2, \ldots, l_\layerindex\}}
			)
		)_{\layerindex \in \{1, 2, \ldots, L\}}
	\in \allowbreak
		\bigtimes_{{\layerindex} = 1}^L\pr{
		      	(\R^{ 
   		      		w_{{\layerindex}, 1} \times w_{{\layerindex}, 2} \times \ldots \times w_{{\layerindex}, \tensordim}
   		      	})^{l_{{\layerindex}} \times l_{{\layerindex}-1}} 
   		      	\times 
   		      	\R^{l_{{\layerindex}}}
	      }
	$,
	$
		\Phiup
	=
		(
			(
				(\Wup_{\layerindex,n,m})_{(n,m) \in \{1, 2, \ldots, l_{\layerindex-1}\} \times \{1, 2, \ldots, l_{\layerindex}\}}
				, \allowbreak 
				(\Bup_{\layerindex,n})_{n \in \{1, 2, \ldots, l_{\layerindex-1}\}}
			)
		)_{\layerindex \in \{1, 2, \ldots, L\}}
	\in \allowbreak
		\bigtimes_{{\layerindex} = 1}^L\pr{
		      	(\R^{ 
		      		w_{{\layerindex}, 1} \times w_{{\layerindex}, 2} \times \ldots \times w_{{\layerindex}, \tensordim}
		      	})^{l_{{\layerindex}-1} \times l_{{\layerindex}}} 
		      	\times 
		      	\R^{l_{{\layerindex}-1}}
    }
	$,
	$\xdown_0 = (\xdown_{0,1}, \ldots, \xdown_{0,l_{0}})  \in (\R^{ \a_{0, 1} \times \a_{0, 2} \times \ldots \times \a_{0,  \tensordim}})^{l_0}$,
	$\xdown_1 = (\xdown_{1,1}, \ldots, \xdown_{1,l_{1}}) \in (\R^{ \a_{1, 1} \times \a_{1, 2} \times \ldots \times \a_{1,  \tensordim}})^{l_1}$,
	$\dots$,
	$\xdown_{L} = (\xdown_{{L},1}, \ldots, \xdown_{{L},l_{{L}}}) \in (\R^{\a_{L, 1} \times \a_{L, 2} \times \ldots \times \a_{L,  \tensordim}})^{l_{L}}$,
	$\xup_{L} = (\xup_{{L},1}, \ldots, \xup_{{L},l_{{L}}}) \in (\R^{\a_{L, 1} \times \a_{L, 2} \times \ldots \times \a_{L,  \tensordim}})^{l_{L}}$,
	$\xup_{L-1} = (\xup_{{L-1},1}, \ldots, \xup_{{L-1},l_{{L-1}}}) \in (\R^{\a_{L-1, 1} \times \a_{L-1, 2} \times \ldots \times \a_{L-1,  \tensordim}})^{l_{L-1}}$,
	$\ldots$ ,
	$\xup_{0} \allowbreak = \linebreak  (\xup_{{0},1}, \ldots, \xup_{{0},l_{{0}}}) \in (\R^{\a_{0, 1} \times \a_{0, 2} \times \ldots \times \a_{0,  \tensordim}})^{l_{0}}$
with
	$\forall \, {\layerindex} \in \{1,2,\dots,L\}, n \in \{1,2,\dots,l_{{\layerindex}}\}  \colon$
\begin{equation}
\label{DecEncCNN:eq1}
\textstyle
	\xdown_{{\layerindex},n} 
=
	\multdim_{
		\activation
	}
	\pr*{ 
		\Bdown_{\layerindex, n} \onetensor^{\a_{\layerindex, 1}, \a_{\layerindex, 2}, \ldots, \a_{\layerindex, \tensordim}} + \sum_{m = 1}^{l_{{\layerindex}-1}} \xdown_{\layerindex-1,m} \strideconvolution \Wdown_{{\layerindex},n,m}
	},
\qquad
	\xup_{L} = \xdown_{L},
\end{equation}
\begin{equation}
\label{DecEncCNN:eq2}
\begin{split} \textstyle
\andq
		\xup_{{\layerindex-1},n} 
	=
		\multdim_{
			\activation \indicator{(1,L]}(\layerindex)
			+
			\id_\R \indicator{\{1\}}(\layerindex) 
		}
		\pr*{ 
			\Bup_{\layerindex, n} \onetensor^{\a_{\layerindex-1, 1}, \a_{\layerindex-1, 2}, \ldots, \a_{\layerindex-1, \tensordim}} 
			+ 
			\sum_{m = 1}^{l_{{\layerindex}}} \xup_{\layerindex,m} \deconvolution \Wup_{{\layerindex},n,m}
		} 
\end{split}
\end{equation}
that
$
	\neuralOp_{(\Phidown, \Phiup)}(\xdown_{0, 1})
=
	\xup_{0, 1}
$
\cfload.
\end{definition}

\begin{remark}[Explanations for \cref{def:DecEncCNN}]
\label{rem:DecEncCNN}
In \eqref{DecEncCNN:eq1} and \eqref{DecEncCNN:eq2} 
we think of $\Phidown$ as the parameters of encoder \CNN,
we think of $\Phiup$ as the parameters of decoder \CNN,
we think of $\xdown_0$ as the input of the encoder \CNN,
for every 
	$\layerindex \in \{1,2,\dots,L-1\}$ 
we think of $\xdown_\layerindex$ as the output of the $\layerindex$-th layer of the encoder \CNN\ 
	given the input $\xdown_0$ 
	consisting of the $l_{{\layerindex}}$ feature maps 
	$
		\xdown_{\layerindex,1}, \xdown_{\layerindex,2}, \ldots, \xdown_{\layerindex,l_{\layerindex}} 
	\in 
		\R^{ \a_{\layerindex, 1} \times \a_{\layerindex, 2} \times \ldots \times \a_{\layerindex,  \tensordim}}
	$
with successively contracting shape $(\a_{\layerindex, 1} , \a_{\layerindex, 2} , \ldots , \a_{\layerindex,  \tensordim})$,
we think of $\xdown_{L} = \xup_{L}$ as the output of the encoder \CNN\ and simultaneously the input of the decoder \CNN,
and for every 
	$\layerindex \in \{L-1, L-2, \ldots, 1\}$ 
we think of $\xup_\layerindex$ as the output of the $\layerindex$-th layer of the encoder \CNN\ 
	given the input $\xup_L$ 
	consisting of the $l_{{\layerindex}}$ feature maps 
	$
		\xup_{\layerindex,1}, \xup_{\layerindex,2}, \ldots, \xup_{\layerindex,l_{\layerindex}} 
	\in 
		\R^{ \a_{\layerindex, 1} \times \a_{\layerindex, 2} \times \ldots \times \a_{\layerindex,  \tensordim}}
	$
with successively expanding shape $(\a_{\layerindex, 1} , \a_{\layerindex, 2} , \ldots , \a_{\layerindex,  \tensordim})$.
\end{remark}
\endgroup

\cfclear
\begingroup
\providecommand{\a}{}
\renewcommand{\a}{\matrixdim}
\providecommand{\Wup}{}
\renewcommand{\Wup}{\mathbf{W}}
\providecommand{\Wdown}{}
\renewcommand{\Wdown}{{W}}
\providecommand{\Bup}{}
\renewcommand{\Bup}{\mathbf{B}}
\providecommand{\Bdown}{}
\renewcommand{\Bdown}{{B}}
\providecommand{\Phiup}{}
\renewcommand{\Phiup}{\mathbf{\Phi}}
\providecommand{\Phidown}{}
\renewcommand{\Phidown}{{\Phi}}
\providecommand{\xup}{}
\renewcommand{\xup}{\mathbf{x}}
\providecommand{\xdown}{}
\renewcommand{\xdown}{{x}}
\providecommandordefault{\neuralOplocal}{\mathcal{N}}
\providecommandordefault{\interpol}{\mathcal{I}}
\begin{definition}[Simple convolutional neural operator with encoder-decoder architecture]
	\label{def:DecEncCNO}
Let
	$\tensordim, L \in \N$,
	$l_0,l_1,\ldots, l_L \in \N$, 
	$(\a_{{\layerindex}, \dimindex})_{({\layerindex}, \dimindex) \in \{0, 1, \ldots, L\} \times \{1, 2, \ldots, \tensordim\}} \subseteq \N$,
	$(w_{{\layerindex}, \dimindex})_{({\layerindex}, \dimindex) \in \{1, 2, \ldots, L\} \times \{1, 2, \ldots, \tensordim\}} \subseteq \N$,
	$\activation \in C(\R, \R)$
satisfy for all 
	$\layerindex \in \{1, 2, \ldots, L\}$,
	$ \dimindex \in \{1, 2, \ldots, \tensordim \}$
that
	$l_0 = 1$	
and
\begin{equation}
\label{T_B_D}
\begin{split} 
	\a_{\layerindex, \dimindex}
=
	\a_{\layerindex-1, \dimindex }/
	w_{\layerindex-1, \dimindex },
\end{split}
\end{equation}
let\cfadd{def:interpolator}
	$\interpol \colon  \R^{\a_{0,1} \times \a_{0,2} \times \ldots \times \a_{0,\tensordim}} \to C([0,1]^\tensordim, \R)$ 
be the periodic multi-linear interpolator on $[0,1]^\tensordim$
with spatial discretization $(\a_{0,1} \times \a_{0,2} \times \ldots \times \a_{0,\tensordim})$,
and let\cfadd{def:DecEncCNN}
	$\neuralOplocal$
be the simple \CNN\ model with encoder-decoder architecture with
	input-output shape $(\a_{0, 1}, \a_{0, 2}, \ldots, \a_{0, \tensordim})$,
	channel structure $(l_0,l_1,\ldots, l_{L-1}, l_L, l_{L-1}, \ldots, l_1, l_0)$, 
	kernel sizes $(w_{{\layerindex}, \dimindex})_{({\layerindex}, \dimindex) \in \{1, 2, \ldots, L\} \times \{1, 2, \ldots, \tensordim\}} $,
	and
	activation function $\activation$
\cfload.
Then we call
	$\neuralOp$
the simple convolutional neural operator with encoder-decoder architecture on $[0,1]^\tensordim$ with
	spatial discretization $(\a_{0, 1}, \a_{0, 2}, \ldots, \a_{0, \tensordim})$,
	channel structure $(l_0,l_1,\allowbreak \ldots, l_{L-1}, l_L, l_{L-1}, \ldots, l_1, l_0)$, 
	kernel sizes $(w_{{\layerindex}, \dimindex})_{({\layerindex}, \dimindex) \in \{1, 2, \ldots, L\} \times \{1, 2, \ldots, \tensordim\}} $,
	and
	activation function $\activation$
if and only if it holds that
\begin{multline}
\label{T_B_D}\textstyle
	\neuralOp
\colon
	\pr[\bigg]{
		\pr[\Big]{
			\bigtimes_{{\layerindex} = 1}^L 
		      	(\R^{ 
		      		w_{{\layerindex}, 1} \times w_{{\layerindex}, 2} \times \ldots \times w_{{\layerindex}, \tensordim}
		      	})^{l_{{\layerindex}} \times l_{{\layerindex}-1}} \times \R^{l_{{\layerindex}}}
		}
		\times \\ \textstyle
		\pr[\Big]{
			\bigtimes_{{\layerindex} = 1}^L
		      	(\R^{ 
		      		w_{{\layerindex}, 1} \times w_{{\layerindex}, 2} \times \ldots \times w_{{\layerindex}, \tensordim}
		      	})^{l_{{\layerindex}-1} \times l_{{\layerindex}}} 
		      	\times 
		      	\R^{l_{{\layerindex}-1}}
		}
	}
	\times
		C([0,1]^\tensordim, \R)
	\to
		C([0,1]^\tensordim, \R)
\end{multline}
is the function which satisfies for all
	$f \in C([0,1]^\tensordim, \R)$,
	$
		\Phidown
	\in \allowbreak
	\pr[\big]{
		\bigtimes_{{\layerindex} = 1}^L
		      	(\R^{ 
   		      		w_{{\layerindex}, 1} \times w_{{\layerindex}, 2} \times \ldots \times w_{{\layerindex}, \tensordim}
   		      	})^{l_{{\layerindex}} \times l_{{\layerindex}-1}} 
   		      	\times 
   		      	\R^{l_{{\layerindex}}}
	 }
	$,
	$
		\Phiup
	\in \allowbreak
	\pr[\big]{
		\bigtimes_{{\layerindex} = 1}^L
		      	(\R^{ 
		      		w_{{\layerindex}, 1} \times w_{{\layerindex}, 2} \times \ldots \times w_{{\layerindex}, \tensordim}
		      	})^{l_{{\layerindex}-1} \times l_{{\layerindex}}} 
		      	\times 
		      	\R^{l_{{\layerindex}-1}}
    }
	$
that 
\begin{equation}
	\neuralOp_{(\Phidown, \Phiup)}(f)
=
	\interpol\pr[\Big]{
		\neuralOplocal_{(\Phidown, \Phiup)}\pr[\Big]{
			\pr[\big]{
				f(
					\nicefrac{\matrixindex_1}{\a_{0, 1}},
					\nicefrac{\matrixindex_2}{\a_{0, 2}}, \ldots,
					\nicefrac{\matrixindex_\tensordim}{\a_{0, \tensordim}}
				)
			}_{
				(\matrixindex_1, \matrixindex_2, \ldots, \matrixindex_\tensordim) \in 
				(\bigtimes_{\dimindex = 1}^\tensordim  \{0, 1, \ldots, \a_{0,\dimindex}-1\} )
			} 
		}
	}.
\end{equation}
\end{definition}

\endgroup

\subsubsection{Integral kernel neural operators}
\label{sect:IKNO}

In this section we introduce \IKNOs, an early architecture for operator learning developed in \cite{Li2020}.
The novel feature of \IKNOs\ compared to the neural operators introduced in the previous sections is that their definition is decoupled from a specific discretization of the function space. 

The \IKNOs\ introduced in this section are not immediately implementable numerically as they require the computation of an integral in \eqref{GKN:eq1}.
When numerically evaluating \IKNOs, \cite{Li2020} propose to approximate the integral using a discretization technique based on message passing graph networks, which can be applied with different discretization grids, and thus call their developed neural operator \GKNO.
In \cite{Li2020a} this discretization methodology is further improved with a multi-scale graph structure and the resulting neural operator is named \MGNO.

\newcommand{\map}{\mathbb{M}}
\newcommand{\spacedim}{d}
\newcommand{\Domain}{D}

\newcommand{\lift}{\mathscr{P}}
\newcommand{\liftdim}{\mathfrak{d}_{\mathscr{P}}}
\newcommand{\project}{\mathscr{Q}}
\newcommand{\projectdim}{\mathfrak{d}_{\mathscr{Q}}}
\newcommand{\intkernel}{\mathscr{K}}
\newcommand{\intkerneldim}{\mathfrak{d}_{\mathscr{K}}}

\newcommand{\nrchannels}{n}
\newcommand{\length}{L}
\newcommand{\meas}{\nu}

\begingroup
\providecommandordefault{\inputFunction}{f}
\cfclear
\begin{definition}[Integral kernel neural operators]
\label{def:GKN}
Let 
	$\spacedim, \nrchannels, \liftdim, \intkerneldim, \projectdim, \length \in \N$,
	$\Domain \in \Borel(\R^{\spacedim})$,
let 
$
		\lift 
	\colon 
		\R^{\liftdim} \times \R^{\spacedim} \times \R
	\to
		\R^\nrchannels
$,
$
		\intkernel 
	\colon 
		\R^{\intkerneldim} \times \R^{\spacedim} \times \R^{\spacedim} \times \R \times \R
	\to
		\R^{\nrchannels \times \nrchannels}
$, 
$
		\project 
	\colon 
		\R^{\projectdim} \times \R^{\nrchannels}
	\to
		\R
$,
and
$\activation \colon \R \to \R$
be measurable functions,
and
let $\meas_x \colon \Borel(\Domain) \to [0,\infty)$, $x \in \Domain$, be measures which satisfy for all 
	$B \in \Borel(\Domain)$
that
	$ \Domain \ni x \mapsto \meas_x(B) \in [0,\infty)$
is measurable.
Then we call 
$\neuralOp$
a integral kernel neural operator
	on the domain $\Domain$,
	with
	lifting model $\lift$,
	kernel model $\intkernel$,
	projection model $\project$,
	integration measures $(\meas_x)_{x \in \Domain}$,
	length $\length$, and 
	activation function $\activation$
if and only if it holds that
\begin{equation}
\label{T_B_D}
\begin{split} 
		\neuralOp
	\colon
		\R^{\liftdim} \times 
		(\R^{\nrchannels \times \nrchannels})^\length \times 
		(\R^{\intkerneldim})^\length \times 
		\R^{\projectdim} \times
		\mathcal{L}^0(\Domain; \R)
	\to
		\mathcal{L}^0(\Domain; \R)
\end{split}
\end{equation}
is a function which satisfies for all
	$P \in \R^{\liftdim}$,
	$W = (W_1, W_2, \ldots, W_L) \in (\R^{\nrchannels \times \nrchannels})^\length$,
	$K = (K_1, K_2, \allowbreak \ldots, K_L) \in (\R^{\intkerneldim})^\length$,
	$Q \in \R^{\projectdim}$,
	$\inputFunction \in \mathcal{L}^0(\Domain; \R)$,
	$v_0, v_1, \ldots, v_\length \in \mathcal{L}^0(\Domain; \R^\nrchannels)$
with 
	$\forall \, x \in \Domain, l \in \{1, 2,\allowbreak \ldots, \length\} \colon$
\begin{equation}
\label{T_B_D}
\begin{split} 
	v_0(x)
=
	\lift_P(x, \inputFunction(x)),
\qquad
	\int_{D}
		\abs*{\intkernel_{K_l}\pr[\big]{x, y, \inputFunction(x), \inputFunction(y)} v_{l-1}(y)}
	\meas_x(\mathrm{d} y) < \infty,
\end{split}
\end{equation}
\begin{equation}
\label{GKN:eq1}
\begin{split} 
\andq
	v_l(x)
=
	\multdim_{\activation} \pr*{
		W_l v_{l-1}(x)
		+
		\int_{D}
			\intkernel_{K_l}\pr[\big]{x, y, \inputFunction(x), \inputFunction(y)} v_{l-1}(y)
		\meas_x(\mathrm{d} y)
	}
\end{split}
\end{equation}
that
\begin{equation}
\label{T_B_D}
\begin{split} 
	\neuralOp_{(P, W, K, Q)}(\inputFunction)
=
	\project_Q(v_\length)
\end{split}
\end{equation}
\cfload.
\end{definition}

\begin{remark}[Choice of $(\meas_x)_{x \in \Domain}$ in \cref{def:GKN}]
Assume the setting in \cref{def:GKN} and let $r \in (0,\infty)$.
In \cite{Li2020} it is suggested to choose the integration measures $(\meas_x)_{x \in \Domain}$ for every 
	$x \in \Domain$
as the Lebesgue measure 
	$\meas_x(\mathrm{d}y) = \mathbbm{1}_{\{z \in \Domain \colon \norm{z-x} \leq r\}} \mathrm{d}y$
supported on a Ball with radius $r$ around $x$.
However, \cite{Li2020} also indicates that this choice can be refined further if knowledge of the target operator is available.
\end{remark}

\subsubsection[Fourier neural operators (FNOs)]{Fourier neural operators}
\label{sect:FNO}

In this section we introduce \FNOs\ as derived in \cite{Li2021,Kovachki2021a}. 
Roughly speaking, \cite{Li2021,Kovachki2021a} motivate \FNOs\ by considering \IKNOs\ with kernel models
(cf.\ $\intkernel$ in \cref{def:GKN})
which only depend on the distance of the two input points.
The authors propose a parametrization of the kernels' action in Fourier space, thereby transforming the integral operation in state space into a multiplication operation in Fourier space.

We introduce \FNOs\ in two steps. 
First, we illustrate the idea of \FNOs\ in \cref{def:FNO_FT} by presenting a version of \FNOs\ using Fourier transforms and inverse Fourier transforms  in $L^2$-spaces, which are not exactly implementable on a computer.
Secondly, we present an implementable version of \FNOs\ based on discrete Fourier transforms and discrete inverse Fourier transforms in \cref{def:FNO}, 
We begin with the standard definitions of Fourier transforms and inverse Fourier transforms in \cref{def:FT,def:IFT} below. 
For their well-definedness see, e.g., \cite[Section 3.4]{Einsiedler2017}.

\newcommand{\ft}{\mathtt{FT}\cfadd{def:FT}}
\newcommand{\ift}{\mathtt{IFT}\cfadd{def:IFT}}

\newcommand{\fourierkernel}{\mathscr{R}}
\newcommand{\fourierkerneldim}{{\mathfrak{d}_\mathscr{R}}}

\begin{definition}[Fourier transforms]
\label{def:FT}
Let $\spacedim, \nrchannels \in \N$, $f = (f_j)_{j \in \{1, 2, \ldots, \nrchannels\}}  \in L^2([0,1]^\spacedim; \C^\nrchannels)$.
Then we denote by $\ft(f) =  (\ft_{j}(f))_{j \in \{1, 2, \ldots, \nrchannels\}} \in L^2(\Z^\spacedim; \C^\nrchannels)$ the function which satisfies for all
	$k \in \Z^\spacedim$,
	$j \in  \{1, 2, \ldots, \nrchannels\}$
that
\begin{equation}
\label{T_B_D}
\begin{split} 
	(\ft_j(f))(k)
=
	\int_{[0,1]^\spacedim}
		f_j(x)
		\exp(-2\pi i \scp{k}{x}) 
	\, \mathrm{d} x.
\end{split}
\end{equation}
\end{definition}

\begingroup
\providecommandordefault{\v}{c}
\begin{definition}[Inverse Fourier transforms]
\label{def:IFT}
Let $\spacedim, \nrchannels \in \N$, $\v = (\v_j)_{j \in \{1, 2, \ldots, \nrchannels\}}  \in L^2(\Z^\spacedim; \C^\nrchannels)$.
Then we denote by $\ift(f) =  (\ift_{j}(f))_{j \in \{1, 2, \ldots, \nrchannels\}} \in L^2([0,1]^\spacedim; \C^\nrchannels)$ the equivalence class of function which satisfies for all
	$j \in  \{1, 2, \ldots, \nrchannels\}$
that
\begin{equation}
\label{T_B_D}
\begin{split} 
	\ift_j(f)
=
	\sum_{k \in \Z^\spacedim}
		\v_j(k)
		\exp(2\pi i \scp{k}{(\cdot)}).
\end{split}
\end{equation}
\end{definition}
\endgroup

\begingroup
\cfclear
\providecommandordefault{\inputFunction}{f}
\providecommandordefault{\Domain}{[0,1]^\spacedim}
\providecommandordefault{\inputspace}{L^0(\Domain; \R)}
\begin{definition}[\FNOs]
\label{def:FNO_FT}
Let 
	$\spacedim, \nrchannels, \liftdim, \fourierkerneldim, \projectdim, \length \in \N$
and
let 
$
		\lift 
	\colon 
		\R^{\liftdim} \times \R^{\spacedim} \times \R
	\to
		\C^\nrchannels
$,
$
		\fourierkernel 
	\colon 
		\R^{\fourierkerneldim} \times \Z^{\spacedim}
	\to
		\C^{\nrchannels \times \nrchannels}
$,
$
		\project 
	\colon 
		\R^{\projectdim} \times \C^{\nrchannels}
	\to
		\R
$,
and 
$\activation \colon \C \to \C$
be measurable functions.
Then we call 
$\neuralOp$
an \FNO\ with 
	lifting model $\lift$,
	Fourier kernel model $\fourierkernel$,
	projection model $\project$,
	length $\length$, and 
	activation function $\activation$
if and only if it holds that
\begin{equation}
\label{T_B_D}
\begin{split} 
		\neuralOp
	\colon
		\R^{\liftdim} \times 
		(\C^{\nrchannels \times \nrchannels})^\length \times 
		(\R^{\fourierkerneldim})^\length \times 
		\R^{\projectdim} \times
		\inputspace
	\to
		\inputspace
\end{split}
\end{equation}
is a function which satisfies for all
	$P \in \R^{\liftdim}$,
	$W = (W_1, W_2, \ldots, W_L) \in (\C^{\nrchannels \times \nrchannels})^\length$,
	$R = (R_1, R_2, \ldots, R_L) \in (\R^{\fourierkerneldim})^\length$,
	$Q \in \R^{\projectdim}$,
	$\inputFunction \in \inputspace$,
	$v_0, v_1, \ldots, v_\length \in L^2([0,1]^\spacedim; \C^\nrchannels)$
with 
	$\forall \, x \in \Domain, l \in \{1, 2, \allowbreak \ldots, \length\} \colon$
\begin{equation}
\label{T_B_D}
\begin{split} 
	v_0(x)
=
	\lift_P(x, \inputFunction(x)),
\qquad
	\fourierkernel_{R_l} \ft(v_{l-1}) \in L^2(\Z^\spacedim; \C^\nrchannels),
\end{split}
\end{equation}
\begin{equation}
\label{FNO_FT:eq1}
\begin{split} 
\andq
	v_l(x)
=
	\multdim_{\activation}
	\pr[\Big]{
		W_l v_{l-1}(x)
		+
		\br*{
			\ift\pr[\big]{ 
				\fourierkernel_{R_l} \ft(v_{l-1})
			}
		}(x)
	}
\end{split}
\end{equation}
that
\begin{equation}
\label{T_B_D}
\begin{split} 
	\neuralOp_{(P, W, R, Q)}(\inputFunction)
=
	\project_Q(v_\length)
\end{split}
\end{equation}
\cfload.
\end{definition}
\endgroup

\begin{definition}[Discretized Fourier transform]
\label{def:DFT}
Let $\spacedim, \nrspacediscr, \nrchannels \in \N$, 
$f = (f_j)_{j \in \{1, 2, \ldots, \nrchannels\}}  \in \mathcal{L}^0([0,1]^\spacedim; \C^\nrchannels)$.
Then we denote by $\dft_\nrspacediscr(f) =  (\dft_{\nrspacediscr, j}(f))_{j \in \{1, 2, \ldots, \nrchannels\}} \in \mathcal{L}^0(\{0, 1, \ldots, \nrspacediscr - 1\}^\spacedim; \C^\nrchannels)$ the function which satisfies for all
	$k \in \{0, 1, \ldots, \nrspacediscr - 1\}^\spacedim$,
	$j \in  \{1, 2, \ldots, \nrchannels\}$
that
\begin{equation}
\label{T_B_D}
\begin{split} 
	(\dft_{\nrspacediscr, j}(f))(k)
=
	\frac{1}{\nrspacediscr^\spacedim}
	\sum_{r \in \{0, 1, \ldots, \nrspacediscr - 1\}^\spacedim}
		f_j(\tfrac{r}{N})
		\exp
		\pr*{
			-2\pi i \tfrac{\scp{k}{r}{}}{\nrspacediscr}
		}
	.
\end{split}
\end{equation}
\end{definition}

\begin{definition}[Inverse discretized Fourier transform]
\label{def:IDFT}
Let 
	$\spacedim, \nrspacediscr, \nrchannels \in \N$, 
	$v = (v_j)_{j \in \{1, 2, \ldots, \nrchannels\}}  \in \mathcal{L}^0(\{0, 1, \ldots, \nrspacediscr - 1\}^\spacedim; \C^\nrchannels)$.
Then we denote by $\idft_{\nrspacediscr}(f) =  (\idft_{\nrspacediscr, j}(f))_{j \in \{1, 2, \ldots, \nrchannels\}} \in \mathcal{L}^0([0,1]^\spacedim; \C^\nrchannels)$ the function which satisfies for all
	$x \in [0,1]^\spacedim$,
	$j \in  \{1, 2, \ldots, \nrchannels\}$
that
\begin{equation}
\label{T_B_D}
\begin{split} 
	\idft_{\nrspacediscr, j}(f)(x)
=
	\sum_{k \in \{0, 1, \ldots, \nrspacediscr - 1\}^\spacedim}
		v_j(k)
		\exp(2\pi i \scp{k}{x}).
\end{split}
\end{equation}
\end{definition}

\cfclear
\begingroup
\providecommandordefault{\inputFunction}{f}
\providecommandordefault{\Domain}{[0,1]^\spacedim}
\begin{definition}[Discretized \FNOs]
\label{def:FNO}
Let 
	$\spacedim, \nrspacediscr, \nrchannels, \liftdim, \projectdim, \length \in \N$,
let 
$
		\lift 
	\colon 
		\R^{\liftdim} \times \R^{\spacedim} \times \R
	\to
		\C^\nrchannels
$,
$
		\project 
	\colon 
		\R^{\projectdim} \times \C^{\nrchannels}
	\to
		\R
$,
and 
$\activation \colon \C \to \C$ be measurable functions.
Then we call 
$\neuralOp$
the discretized \FNO\ with 
	spatial discretization $\nrspacediscr$,
	lifting model $\lift$,
	projection model $\project$,
	length $\length$, and 
	activation function $\activation$
if and only if it holds that
\begin{equation}
\label{FNO:eq1}
\begin{split} 
		\neuralOp
	\colon
		\R^{\liftdim} \times 
		(\C^{\nrchannels \times \nrchannels})^\length \times 
		((\C^{\nrchannels \times \nrchannels})^{\{0, 1, \ldots, \nrspacediscr - 1\}^\spacedim})^\length \times 
		\R^{\projectdim} \times
		\mathcal{L}^0([0,1]^\spacedim; \R)
	\to
		\mathcal{L}^0([0,1]^\spacedim; \R)
\end{split}
\end{equation}
is the function which satisfies for all
	$P \in \R^{\liftdim}$,
	$W = (W_1, W_2, \ldots, W_L) \in (\C^{\nrchannels \times \nrchannels})^\length$,
	$R = \allowbreak (R_1, R_2, \ldots, R_L) \in ((\C^{\nrchannels \times \nrchannels})^{\{0, 1, \ldots, \nrspacediscr - 1\}^\spacedim})^\length$,
	$Q \in \R^{\projectdim}$,
	$\inputFunction \in \mathcal{L}^0([0,1]^\spacedim; \R)$,
	$v_0, v_1, \ldots, v_\length \in \mathcal{L}^0([0,1]^\spacedim; \C^\nrchannels)$
with 
	$\forall \, x \in \Domain, l \in \{1, 2, \ldots, \length\} \colon$
\begin{equation}
\label{FNO:eq2}
\begin{split} 
	v_0(x)
=
	\lift_P(x, \inputFunction(x))
\quad\text{and}\quad
	v_l(x)
=
	\multdim_{\activation}
	\pr*{
		W_l v_{l-1}(x)
		+
		\br*{
			\idft_{\nrspacediscr}\pr[\big]{ 
					R_l \dft_{\nrspacediscr}(v_{l-1})
			}
		}(x)
	}
\end{split}
\end{equation}
that
\begin{equation}
\label{FNO:eq3}
\begin{split} 
	\neuralOp_{(P, W, R, Q)}(\inputFunction)
=
	\project_Q(v_\length)
\end{split}
\end{equation}
\cfload.
\end{definition}
\endgroup

\begin{remark}%

In the following we provide some explanations and remarks for \cref{def:FNO}.

\begin{enumerate}[label=\alph*)]
\item 
Note that in \eqref{FNO:eq2} we have that
	$R_l \in (\C^{\nrchannels \times \nrchannels})^{\{0, 1, \ldots, \nrspacediscr - 1\}^\spacedim}$
is a function from ${\{0, 1, \ldots, \nrspacediscr - 1\}^\spacedim}$ to $\C^{\nrchannels \times \nrchannels}$ and
	$\dft_{\nrspacediscr}(v_{l-1}) \in \mathcal{L}^0(\{0, 1, \ldots, \nrspacediscr - 1\}^\spacedim; \C^\nrchannels)$
is a function from ${\{0, 1, \ldots, \nrspacediscr - 1\}^\spacedim}$ to $\C^{\nrchannels}$.
Hence 
	$R_l\dft_{\nrspacediscr}(v_{l-1})$ 
is a function from ${\{0, 1, \ldots, \nrspacediscr - 1\}^\spacedim}$ to $\C^{\nrchannels}$. 

\item
Note that in \cref{def:FNO} instead of using an abstract kernel model as in \cref{def:FNO_FT}, we explicitly specify the action of the parameters $(R_1, R_2, \ldots, R_L)$ on the Fourier modes in \eqref{FNO:eq2} (cf.\ \eqref{FNO_FT:eq1} in \cref{def:FNO}). 
In \cite[top of page 23]{Kovachki2021a} and \cite[top of page 6]{Li2021} experiments with a more general kernel model for \FNOs\ have been performed, which yielded either same or worse accuracy for more computational effort.
Thus, the version presented in \cref{def:FNO} is the way in which \FNOs\ are typically employed.

\item 
In application scenarios of \FNOs\ when one is only interested in function values on an equidistant grid, the $\dft$ and $\idft$ can be implemented based on the \FFT\ and \IFFT, which significantly decreases the computational effort.

\item
A discretized \FNO\ with spatial discretization $\nrspacediscr \in \N$ can also be used when the input function is given through discretized values on a grid with $M \in \N$ gird points in each dimension and $M \geq \nrspacediscr$. In this case in \eqref{FNO:eq2} $\dft_{\nrspacediscr}$ should be replaced with $\dft_M$ composed with an operation to truncate the excess frequencies.

\item 
For brevity and simplicity we have presented \FNOs\ in \cref{def:FNO} with complex valued hidden functions (cf.\ $v_0, v_1, \ldots, v_\length$ in \eqref{FNO:eq1}). 
For reasons of computational efficiency or when there is a theoretical motivation, one might want to ensure that the hidden functions are real-valued. To achieve this, a few modifications to \cref{def:FNO} can be made:

\begin{itemize}
    \item Modify the codomain of $\lift$ to $\R^\nrchannels$,
    \item change the domain and codomain of $\act$ to $\R$, and 
    \item in the definition of $\neuralOp$, impose additional restrictions on $R = \allowbreak (R_1, R_2, \ldots, R_L) \in \linebreak ((\C^{\nrchannels \times \nrchannels})^{\{0, 1, \ldots, \nrspacediscr - 1\}^\spacedim})^\length$ such that for all 
    	$l \in \{1, 2, \ldots, \length\}$, 
    	$k = (k_1, k_2, \ldots, k_\spacedim) \in \{0, 1, \ldots, \nrspacediscr - 1\}^\spacedim$ 
    we have that 
    \begin{equation}
    \label{T_B_D}
    \begin{split} 
     	R_l(k) 
        = 
        \br{
            R_l(\mod{\nrspacediscr}{-k_1}, \mod{\nrspacediscr}{-k_2}, \ldots, \mod{\nrspacediscr}{-k_\spacedim})
        }^\ast.
    \end{split}
    \end{equation}
\end{itemize}
These modifications make it possible to utilize the real \FFT\ and inverse real \FFT\ in applications based on equidistant grids.
\end{enumerate}

\end{remark}

\subsubsection{Deep Operator Networks (DeepONets)}
\label{sect:deepONets}

In this section we discuss \DeepONets\ which were introduced in \cite{Lu2021}. 
\DeepONets\ are motivated by the fact that they
possess universal approximation properties for operators (cf.\ \cite[Theorem 1]{Lu2021})
similar to the universal approximation properties of fully-connected feedforward \ANNs\ for functions on finite dimensional spaces (cf.\ \cite{Cybenko1989,Hornik1989}).

\newcommand{\branch}{\mathscr{B}}
\newcommand{\trunk}{\mathscr{T}}

\newcommand{\branchdim}{{\mathfrak{d}_{\mathscr{B}}}}
\newcommand{\trunkdim}{{\mathfrak{d}_{\mathscr{T}}}}

\newcommand{\ONetwidth}{p}
\newcommand{\nrsensors}{m}

\begingroup
\providecommandordefault{\x}{x}
\begin{definition}[Unstacked \DeepONets]
\label{def:DeepONets}
Let 
	$\spacedim, \branchdim, \trunkdim, \nrsensors, \ONetwidth \in \N$,
let $\Domain$ be a set,
let $\x_1, \x_2, \ldots, \x_\nrsensors \in \Domain$
and
let 
$
		\branch 
	\colon 
		\R^{\branchdim} \times \R^{\nrsensors}
	\to
		\R^{\ONetwidth}
$ 
and
$
		\trunk 
	\colon 
		\R^{\trunkdim} \times \R^{\spacedim}
	\to
		\R^{\ONetwidth}
$ 
be measurable functions.
Then we call 
$\neuralOp$
the unstacked \DeepONet\ neural operator
	with sensors $\x_1, \x_2, \ldots, \x_\nrsensors$,
	branch model $\branch$, and
	trunk model $\trunk$
if and only if it holds that
\begin{equation}
\label{T_B_D}
\begin{split} 
		\neuralOp
	\colon
		\R^{\branchdim} \times
		\R^{\trunkdim} \times
		\mathcal{L}^0(\Domain; \R)
	\to
		\mathcal{L}^0(\R^\spacedim; \R)
\end{split}
\end{equation}
is the function which satisfies for all
	$B \in \R^{\branchdim}$,
	$T = \R^{\trunkdim}$,
	$u \in \mathcal{L}^0(\Domain; \R)$,
	$y \in \R^\spacedim$
that
\begin{equation}
\label{T_B_D}
\begin{split} 
	(\neuralOp_{(B, T)}(u))(y)
=
	\scp[\big]{
		\branch_B(u(\x_1), u(\x_2), \ldots, u(\x_\nrsensors))
	}{
		\trunk_T(y)
	}.
\end{split}
\end{equation}
\end{definition}

\begin{remark}[Stacked \DeepONets]
In \cite{Lu2021} two versions of \DeepONets\ are introduced: Stacked \DeepONets\ and unstacked \DeepONets. 
Stacked \DeepONets\ are a special case of the unstacked \DeepONets\ defined in \cref{def:DeepONets} when, roughly speaking, each output component of the branch model $\branch$ only depends on a distinct set of the $\branchdim$ model parameters.
\end{remark}

\endgroup

\subsection{Physics-informed neural operators}
\label{sect:PINO}

The \PINOs\ methodology was introduced in \cite{Li2021} to adapt the \PINNs\  methodology of \cite{Karniadakis21,Raissi19} to the task of operator learning.
Roughly speaking, the \PINO\ methodology expands the general framework for operator learning described in \cref{algo:operator_learning} by incorporating an additional loss term corresponding to the residual of the underlying \PDE. %

In this section we formulate the \PINO\ methodology in the context of two frameworks. Firstly, in \cref{algo:PINO_general}, we consider the case of a general abstract \PDE, possibly with boundary conditions (cf.\ \cref{sect:PINNsBVP}).
Secondly, in \cref{algo:PINO_time}, we specify the \PINO\ methodology to a more specific case of a time-dependent initial value \PDE\ with boundary conditions (cf.\ \cref{sect:PINNsIVP}).

\begingroup
\providecommandordefault{\d}{\matrixdim}
\providecommandordefault{\dparam}{\mathtt{d}}
\providecommandordefault{\dcond}{\mathbf{d}}
\providecommandordefault{\dsol}{d}
\providecommandordefault{\p}{p}
\providecommandordefault{\Dparam}{\mathtt{D}}
\providecommandordefault{\Dcond}{\mathbf{D}}
\providecommandordefault{\Dsol}{D}
\providecommandordefault{\In}{\mathcal{I}}
\providecommandordefault{\Diff}{\mathscr{D}}
\providecommandordefault{\S}{\mathcal{S}}
\providecommandordefault{\i}{i}
\providecommandordefault{\x}{x}
\providecommandordefault{\nrparams}{\mathfrak{d}} %
\providecommandordefault{\y}{y}
\providecommandordefault{\b}{m}
\providecommandordefault{\m}{m}
\providecommandordefault{\B}{B}
\providecommandordefault{\factData}{\lambda_{\text{Data}}}
\providecommandordefault{\factPDE}{\lambda_{\text{PDE}}}
\providecommandordefault{\B}{\mathcal{M}}
\providecommandordefault{\BData}{\mathcal{M}_{\text{Data}}}
\providecommandordefault{\BPDE}{\mathcal{M}_{\text{PDE}}}
\providecommandordefault{\LossData}{\mathscr{L}_{\text{Data}}}
\providecommandordefault{\LossPDE}{\mathscr{L}_{\text{PDE}}}
\providecommandordefault{\GradData}{\mathscr{G}_{\text{Data}}}
\providecommandordefault{\GradPDE}{\mathscr{G}_{\text{PDE}}}
\providecommandordefault{\II}{\mathbb{I}}
\providecommandordefault{\XX}{\mathbb{X}}
\providecommandordefault{\YY}{\mathbb{Y}}
\providecommandordefault{\empRisk}{\mathscr{R}}
\providecommandordefault{\Grad}{\mathscr{G}}
\begin{algo}[\PINO\ methods for boundary value \PDE\ problems]
\label{algo:PINO_general}
Let 
	$\dparam, \dcond, \dsol, \p, \BData, \allowbreak \BPDE \in \N$,
	$\factData, \factPDE \in [0,\infty)$,
	$\Dparam \subseteq \R^\d$,
	$\Dsol \subseteq \R^{\dsol}$,
	$\Dcond \subseteq \R^{\dcond}$,
	$\In \subseteq C(\Dparam, \R)$,
let
	$\Diff \colon \In \times C^\p(\Dsol, \R) \to C(\Dcond, \R^\dcond)$,
	$\S \colon \In \to C^\p(\Dsol, \R)$,
$\neuralOp \colon \R^\nrparams \times \In \to C^{\p}(\Dsol, \R)$,
	$\LossData \colon \R^\nrparams \times \In \times \Dsol \to \R$, and  
	$\LossPDE \colon \R^\nrparams \times \In \times \Dcond \to \R$
satisfy for all
	$\theta \in \R^\nrparams$,
	$\i \in \In$,
	$\x \in \Dsol$,
	$\y \in \Dcond$
that
\begin{equation}
\label{PINO_general:eq1}
\begin{split} 
	\LossData(\theta, \i, \x)
=
	\abs*{
		(\neuralOp_{\theta}(\i)) (\x) - (\S(\i))(\x)
	}^2
\qandq
	\LossPDE(\theta, \i, \y)
=
	\norm*{
		\pr[\big]{\Diff(\i, \neuralOp_{\theta}(\i))} (\y)
	}^2,
\end{split}
\end{equation}
let $(\Omega, \cF, \P)$ be a probability space,
let $\II_{k, n, \m} \colon \Omega \to \In$, $k, n, \m \in \N$, be random variables, 
let $\XX_{n, \m} \colon \Omega \to \Dsol$, $n, \m \in \N$, be random variables, 
let $\YY_{n, \m} \colon \Omega \to \Dcond$, $n, \m \in \N$, be random variables, 
for every
	$n \in \N$
let $\empRisk_n \colon \R^\nrparams \times \Omega \to \R$ satisfy for all
	$\theta \in \R^\nrparams$
that
\begin{equation}
\label{PINO_general:eq2}
\begin{split} \textstyle
	\empRisk_n(\theta)
=
	\frac{\factData}{\BData} 
	\br*{
		\sum_{\m = 1}^{\BData}
			\LossData( \theta,  \II_{1, n, m}, \XX_{n, \m})
	}
	+
	\frac{\factPDE}{\BPDE} 
	\br*{
		\sum_{\m = 1}^{\BPDE}
			\LossPDE( \theta,  \II_{2, n, m}, \YY_{n, \m})
	},
\end{split}
\end{equation}
for every
	$n \in \N$
let $\Grad_n \colon \R^\nrparams \times \Omega \to \R^\nrparams$ satisfy for all
	$\theta \in \R^\nrparams$,
	$\omega \in \{w \in \Omega \colon \empRisk_n(\cdot, w) \text{ is differentiable at } \theta\}$
that
$
	\Grad_n(\theta, \omega)
=
	(\nabla_\theta \empRisk_n)(\theta, \omega)
$,
let $(\gamma_n )_{n \in \N} \subseteq \R$, 
and let $\Theta\colon \N_0 \times \Omega \to \R^{ \nrparams}$ satisfy for all $n \in \N_0$ that 
\begin{equation}
\label{PINO_general:eq3}
\begin{split}
	\Theta_{ n + 1 } 
=
	\Theta_n 
	- 
	\gamma_n
	\Grad_n(\Theta_n).
\end{split}
\end{equation}
\end{algo}

\begin{remark}[Explanations for \cref{algo:PINO_general}]
{\sl 
We think of $\Diff$ as the residual of a parametric \PDE\ with boundary conditions,
we think of $\S$ as (an approximation of) the solution operator corresponding to the parametric \PDE\ 
in the sense that for every parameter 
	$\i \in \In$
we have that
\begin{equation}
\label{T_B_D}
\begin{split} 
	\Diff(\i, \S(\i)) = 0,
\end{split}
\end{equation}  
	we think of $\In$ as the set of admissible parameters of the \PDE,
	we think of $\Dparam$ as the domain of the functions which act as parameters for the \PDE,
	we think of $\Dsol$ as the domain of the solution of the \PDE, %
	we think of $\Dcond$ as the domain on which the \PDE\ and the boundary conditions are jointly evaluated, 
we think of $\neuralOp$ as the neural operator that we want to utilize to learn the operator $\S$,
we think of $\Theta$ as an \SGD\ optimizing process for the parameters of the model $\neuralOp$,
for every 
	$n \in \N$ 
we think of $\empRisk_n$ as the empirical risk employed in the $n$-th step of the \SGD\ optimization procedure $\Theta$,
we think of $\LossData$ as the loss coming from the deviation between the model and the target operator, 
we think of $\LossPDE$ as the loss coming from the deviation of the model from the \PDE\ and its boundary conditions,
we think of $\factData$ and $\factPDE$ as the factors which scale the losses $\LossData$ and $\LossPDE$ in the empirical risk evaluations,
we think of $\BData$ and $\BPDE$ as the number of samples for the losses $\LossData$ and $\LossPDE$ in the empirical risk evaluations,
we think of $(\II_{1, n, \b})_{(n, \b) \in \N^2}$ as the \PDE\ parameters for which reference solutions $(\S(\II_{1, n, \b}))_{(n, \b) \in \N^2}$ are computed in empirical risk evaluations,
we think of $(\XX_{ n, \b})_{(n, \b) \in \N^2}$ as the points at which the model is compared to the reference solutions $(\S(\II_{1, n, \b}))_{(n, \b) \in \N^2}$ in empirical risk evaluations,
we think of $(\II_{2, n, \b})_{(n, \b) \in \N^2}$ as the \PDE\ parameters for which the residual of the \PDE\ and its boundary conditions are computed in empirical risk evaluations, 
and
we think of $(\YY_{ n, \b})_{(n, \b) \in \N^2}$ as the points at which the residual of the \PDE\ and its boundary conditions are evaluated in empirical risk evaluations.
For sufficiently large $n \in \N$ we hope that 
$
	\neuralOp_{\Theta_n} \approx \S
$.
}
\end{remark}

\begin{remark}[Applications and extensions of \cref{algo:PINO_general}]
{\sl 
Being the squared residual of a \PDE, the term $\LossPDE$ in \cref{algo:PINO_general} typically involves the derivation of the neural operator $\neuralOp$ with respect to the input space variables. In \cite{Li2021} different ways to compute this derivation are proposed, including automatic derivation and, in the case of \FNOs, derivation in Fourier space.

Moreover, in \cite{Li2021} an extension of \cref{algo:PINO_general} is proposed which aims to further combine the \PINO\ methodology with the pure \PINNs\  methodology in \cite{Karniadakis21}. Roughly speaking, after an initial training phase as described in \cref{algo:PINO_general} they suggest using the \PINNs\ methodology (i.e., only using $\LossPDE$ as loss) for a fixed $\i \in \In$ to improve the approximation for this specific instance of the parametric \PDE.

}
\end{remark}
\endgroup

\begingroup
\providecommandordefault{\factData}{\lambda_{\text{Data}}}
\providecommandordefault{\factPDE}{\lambda_{\text{PDE}}}
\providecommandordefault{\factBd}{\lambda_{\text{Boundary}}}
\providecommandordefault{\factInit}{\lambda_{\text{Init}}}
\providecommandordefault{\BData}{\mathcal{M}_{\text{Data}}}
\providecommandordefault{\BPDE}{\mathcal{M}_{\text{PDE}}}
\providecommandordefault{\BBd}{\mathcal{M}_{\text{Boundary}}}
\providecommandordefault{\BInit}{\mathcal{M}_{\text{Init}}}
\providecommandordefault{\B}{\mathcal{M}}
\providecommandordefault{\Loss}{\mathscr{L}}
\providecommandordefault{\LossData}{\mathscr{L}_{\text{Data}}}
\providecommandordefault{\LossPDE}{\mathscr{L}_{\text{PDE}}}
\providecommandordefault{\LossBd}{\mathscr{L}_{\text{Boundary}}}
\providecommandordefault{\LossInit}{\mathscr{L}_{\text{Init}}}
\providecommandordefault{\Grad}{\mathscr{G}}
\providecommandordefault{\GradData}{\mathscr{G}_{\text{Data}}}
\providecommandordefault{\GradPDE}{\mathscr{G}_{\text{PDE}}}
\providecommandordefault{\GradBd}{\mathscr{G}_{\text{Boundary}}}
\providecommandordefault{\S}{\mathcal{S}}
\providecommandordefault{\II}{\mathbb{I}}
\providecommandordefault{\XX}{\mathbb{X}}
\providecommandordefault{\YY}{\mathbb{Y}}
\providecommandordefault{\TT}{\mathbb{T}}
\providecommandordefault{\In}{\mathcal{I}}
\providecommandordefault{\D}{D}
\providecommandordefault{\Diff}{\mathscr{D}}
\providecommandordefault{\Bd}{\mathscr{B}}
\providecommandordefault{\x}{x}
\providecommandordefault{\d}{d}
\providecommandordefault{\y}{y}
\providecommandordefault{\i}{i}
\providecommandordefault{\u}{u}
\providecommandordefault{\p}{p}
\providecommandordefault{\t}{t}
\providecommandordefault{\b}{m}
\providecommandordefault{\nrparams}{\fd}
\providecommandordefault{\m}{m}
\providecommandordefault{\empRisk}{\mathscr{R}}
\providecommandordefault{\Grad}{\mathscr{G}}
\begin{algo}[\PINO\ methods for time-dependent initial value \PDE\ problems]
\label{algo:PINO_time}
Let 
	$T \in (0,\infty)$,
	$\d, \p, \allowbreak \BData, \BPDE, \allowbreak \BBd, \BInit \in \N$,
	$\factData, \factPDE, \factBd, \factInit \in [0,\infty)$,
	$\D \subseteq \R^\d$,
	$\In \subseteq C(\D, \R)$,
let
	$\Diff \colon C^\p([0,T] \times \D, \R) \to C([0,T] \times \D, \R)$,
	$\Bd \colon C^\p([0,T] \times \D, \R) \to C([0,T] \times\partial \D, \R)$,
	$\S \colon \In \to C^\p([0,T] \times \D, \R)$,
	$\neuralOp \colon \R^\nrparams \times \In \to C^{\p}([0,T] \times \D, \R)$,
	$\LossData \colon \R^\nrparams \times \In \times [0,T] \times \D \to [0,\infty)$, 
	$\LossPDE \colon \R^\nrparams \times \In \times [0,T] \times \D \to [0,\infty)$, 
	$\LossBd \colon \R^\nrparams \times \In \times \partial  \D \to [0,\infty)$, and
	$\LossInit \colon \R^\nrparams \times \In \times \D \to [0,\infty)$,
satisfy for all
	$\theta \in \R^\nrparams$,
	$\i \in \In$,
	$t \in [0,T]$,
	$\x \in \D$,
	$\y \in \partial \D$
that
\begin{equation}
\label{T_B_D}
\begin{split} 
	\LossData(\theta, \i, t, \x)
=
	\abs*{
		\pr[\big]{\neuralOp_{\theta}(\i)} (t, \x) - (\S(\i)) (t, \x)
	}^2,
\qquad
	\LossPDE(\theta, \i, t, \x)
=
	\abs*{
		\pr[\big]{\Diff(\neuralOp_{\theta}(\i))} (t, \x)
	}^2,
\end{split}
\end{equation}
\begin{equation}
\label{T_B_D}
\begin{split} 
	\LossBd(\theta, \i, t, \y)
=
	\abs*{
		\pr[\big]{\Bd\pr[\big]{\neuralOp_{\theta}(\i)}}(t, \y)
	}^2,
\quad\text{and}\quad
	\LossInit(\theta, \i, \y)
=
	\abs*{
		\pr[\big]{\neuralOp_{\theta}(\i)} (t, \x) - \i(\x)
	}^2,
\end{split}
\end{equation}
let $(\Omega, \cF, \P)$ be a probability space,
let $\II_{k, n, \m} \colon \Omega \to \In$, $k, n, \m \in \N$, be random variables, 
let $\XX_{n, \m} \colon \Omega \to \D$, $k, n, \m \in \N$, be random variables, 
let $\TT_{n, \m} \colon \Omega \to [0,T]$, $k, n, \m \in \N$, be random variables, 
let $\YY_{n, \m} \colon \Omega \to \partial \D$, $k, n, \m \in \N$, be random variables, 
for every
	$n \in \N$
let $\empRisk_n \colon \R^\nrparams \times \Omega \to \R$ satisfy for all
	$\theta \in \R^\nrparams$
that
\begin{equation}
\label{T_B_D}
\begin{split} \textstyle
	\empRisk_n(\theta)
&=\textstyle
	\frac{\factData}{\BData} 
	\br*{
		\sum\limits_{\m = 1}^{\BData}
			\LossData( \theta,  \II_{1, n, m}, \TT_{1, n, \m}, \XX_{1, n, \m})
	}
	+
	\frac{\factPDE}{\BPDE} 
	\br*{
		\sum\limits_{\m = 1}^{\BPDE}
			\LossPDE( \theta,  \II_{2, n, m}, \TT_{2, n, \m}, \XX_{2, n, \m})
	}\\
&\quad\textstyle
	\frac{\factData}{\BData} 
	\br*{
		\sum\limits_{\m = 1}^{\BData}
			\LossBd( \theta,  \II_{3, n, m}, \TT_{3, n, \m}, \YY_{3, n, \m})
	}
	+
	\frac{\factPDE}{\BPDE} 
	\br*{
		\sum\limits_{\m = 1}^{\BPDE}
			\LossInit( \theta,  \II_{4, n, m}, \YY_{4, n, \m})
	},
\end{split}
\end{equation}
for every
	$n \in \N$
let $\Grad_n \colon \R^\nrparams \times \Omega \to \R^\nrparams$ satisfy for all
	$\theta \in \R^\nrparams$,
	$\omega \in \{w \in \Omega \colon \empRisk_n(\cdot, w) \text{ is differentiable at } \theta\}$
that
$
	\Grad_n(\theta, \omega)
=
	(\nabla_\theta \empRisk_n)(\theta, \omega)
$,
let $(\gamma_n )_{n \in \N} \subseteq \R$, 
and let $\Theta\colon \N_0 \times \Omega \to \R^{ \nrparams}$ satisfy for all $n \in \N_0$ that 
\begin{equation}
\begin{split}
	\Theta_{ n + 1 } 
=
	\Theta_n - 
	\gamma_n
	\Grad(\Theta_n)
	.
\end{split}
\end{equation}
\end{algo}

\begin{remark}[Explanations for \cref{algo:PINO_time}]
{\sl
We think of $\Diff$ as the residual of a time-dependent parametric \PDE\ with time horizon $T$,
we think of $\Bd$ as the residual of the boundary conditions of the \PDE,
we think of $\S$ as (an approximation of) the solution operator corresponding to the \PDE\ in the sense that for every parameter $\i \in \In$ we have that
\begin{equation}
\label{T_B_D}
\begin{split}
	(\S(\i))(0, \cdot) = \i,
\qquad
	\Diff(\i, \S(\i)) = 0,
\qandq
	\Bd(\S(\i)) = 0,
\end{split}
\end{equation}
we think of $\In$ as the set of admissible initial conditions of the time-dependent \PDE,
we think of $\D$ as the spatial domain of the \PDE,
we think of $\neuralOp$ as the neural operator that we want to utilize to learn the operator $\S$,
we think of $\Theta$ as an \SGD\ optimizing process for the parameters of the model $\neuralOp$,
for every $n \in \N$ we think of $\empRisk_n$ as the empirical risk employed in the $n$-th step of the \SGD\ optimization procedure $\Theta$,
we think of $\LossData$ as the loss coming from the deviation between the model and the target operator,
we think of $\LossPDE$ as the loss coming from the deviation of the model from the \PDE,
we think of $\LossBd$ as the loss coming from the deviation of the model from the boundary conditions of the \PDE,
we think of $\LossInit$ as the loss coming from the deviation of the model from the initial conditions of the \PDE,
we think of $\factData$, $\factPDE$, $\factBd$, and $\factInit$ as the factors which scale the losses $\LossData$, $\LossPDE$, $\LossBd$, and $\LossInit$ in the empirical risk evaluations,
we think of $\BData$, $\BPDE$, $\BBd$, and $\BInit$ as the number of samples for the losses $\LossData$, $\LossPDE$, $\LossBd$, and $\LossInit$ in the empirical risk evaluations,
and
we think of 
	$(\II_{k, n, \b})_{(k, n, \b) \in \N^3}$,
	$(\XX_{k, n, \b})_{(k, n, \b) \in \N^3}$,
	$(\TT_{k, n, \b})_{(k, n, \b) \in \N^3}$, and
	$(\YY_{k, n, \b})_{(k, n, \b) \in \N^3}$
as random variables which specify the training data for the \SGD\ optimization procedure.
For sufficiently large $n \in \N$ we hope that
$
	\neuralOp_{\Theta_n} \approx \S
$.
} 
\endgroup

\subsection{Other deep operator learning approaches}
\label{sect:OL_refs}

\begingroup 
\renewcommand{\ex}{e.g.,}
In this section we list a selection of other promising deep operator learning approaches in the literature.
We refer, \ex\ to \cite{Li2022} for a generalization of \FNOs\ to more complicated geometries,
we refer, \ex\ to \cite{Brandstetter2022} for an extension of \FNOs\ by using Clifford layers where calculations are performed in higher-dimensional non-commutative Clifford algebras,
we refer, \ex\ to \cite{Lanthaler2022} for a generalization of \DeepONets\ to a more sophisticated nonlinear architecture,
we refer, \ex\ to \cite{Wang2021} for physics-informed \DeepONets,
we refer, \ex\ to \cite{Meuris2023} for a combination of \DeepONets\ with spectral methods,
we refer, \ex\ to \cite{Liu2022a} for an approach to learn the solution operator of a time-evolution \PDE\ for long time horizons based on approximating the time propagation of the \PDE\ for short time steps using a variant of \DeepONet,
we refer, \ex\ to \cite{Goswami2023} for a variant of \DeepONets\ for learning solution operators to stiff initial value problems, 
we refer, \ex\ to \cite{Tan2022} for a variant of \DeepONets\ capable of taking several functions as input,
we refer, \ex\ to \cite{Kontolati2023} for a variant of \DeepONets\ which uses latent representations of the high-dimensional input and output functions,
we refer, \ex\ to \cite{Wu2023} for asymptotic-preserving convolutional \DeepONets\ for linear transport equations,
we refer, \ex\ to \cite{Du2023} for \DeepONets\ applied to singularly perturbed problems,
we refer, \ex\ to \cite{Zhang2023} for \DeepONets\ applied to evolutionary problems based on approximating the initial condition with a \DeepONet\ and then applying a time evolution to the parameters of the \DeepONet,
we refer, \ex\ to \cite{Lu2022} for a comparison between the \DeepONet\ and \FNO\ methodologies,
we refer, \ex\ to \cite{Rafiq2022} for a spatio-spectral neural operator,
we refer, \ex\ to \cite{Xiong2023} for the Koopman neural operator,
we refer, \ex\ to \cite{Fresca2021,Fresca2022} for deep learning-based reduced order models for parametric \PDE\ problems,
we refer,  \ex\ to \cite{Schwab2023,Deng2022,DeRyck2022} for approximation rates and convergence results for neural operators, 
we refer, \ex\ to \cite{Pham2022} for a generalization of \DeepONets\ to operator learning on Wasserstein spaces,
we refer, \ex\ to \cite{Liu2022} for a method to learn the entire flow map associated to \ODEs\ by training a different \ANN\ in each time-step and combining these \ANNs\ with classical Runge-Kutta methods on different time scales,
we refer, \ex\ to \cite{Guo2016,Raonic2023,Heiss2023,Zhu2018,Khoo2021} for operator learning architectures based on \CNNs,
we refer, \ex\ to \cite{Kovachki2021,Chen2023} for estimates for approximation and generalization errors in network-based operator learning for \PDEs,
and 
we refer, \ex\ to \cite[Appendix D]{Brandstetter2022} and \cite[Section 1.7.4]{Jentzen2023} for other literature overviews on operator learning approaches.

\subsection{Numerical results}
\label{sect:OL_simul}

In this section we numerically test some of the operator learning models discussed in \cref{sect:OL_architectures}
in the case of four operators related to parametric \PDE\ problems
and compare their performance with classical numerical methods.
Specifically, we consider 
	the numerical approximation of
	an operator mapping initial values to terminal values of the viscous Burgers equation in \cref{sect:OL_simul_Burgers},
	the numerical approximation of
	operators mapping
	initial values to terminal values of one and two-dimensional Allen-Cahn equations in \cref{sect:OL_simul_AC},	
	and
	the numerical approximation of
	an operator mapping source terms to terminal values of a reaction-diffusion equation in \cref{sect:OL_simul_ReactionDiffusion}.

In every considered problem, all models 
are trained using the same training and validation set with the Adam optimizer with adaptive learning rates starting at 
$\frac{1}{1000}$
instead of the \SGD\ optimizer considered in \cref{algo:operator_learning}
(cf.\ \cite[Appendix A.2]{adanns2023}
for a detailed description of our adaptive training procedure).
Moreover, 
we repeat the training of every operator learning model several times with different initializations and select the best performing trained model over all training runs as the approximation for that architecture.
In every considered problem, the $L^2$-errors of all methods are approximated using a \MC\ approximation based on the same test set.
The parameters chosen to generate train, validation, and test sets, as well as additional hyperparameters for each problem are listed in 
\cref{table:training_hyperparams_OL}.

All the simulations were run on a remote machine on \texttt{https://vast.ai} equipped with an
\textsc{NVIDIA GeForce RTX 3090} GPU with 24 GB RAM
and an
\textsc{AMD Ryzen} 5950X 16-core CPU with 16 GB of total system RAM.
As the evaluation time of models on GPUs can be highly variable, we report the average evaluation time 
over $\nrEvalRuns$ test set evaluations for each method
in \cref{table:OLBurgers,table:AC1,table:AC2,table:OLReactDiff}.
The code for all numerical simulations is available at 
\url{https://github.com/deeplearningmethods/deep-learning-pdes}.

\subsubsection{Viscous Burgers equation}
\label{sect:OL_simul_Burgers}

\begingroup
\providecommandordefault{\eqname}{Burgers_T1.0_S6.283185307179586_nu0.1_var1000_decay3.0_offset9.999999999999998_innerdecay2.0}

\providecommandordefault{\Domain}{(0, \OLBurgersSpaceSize)}
\providecommandordefault{\IVSpace}{H^2_\text{per}(\Domain; \R)}
\providecommandordefault{\EndSpace}{H^2_\text{per}(\Domain; \R)}

In this section we present numerical simulations for the approximation of an operator  mapping initial values to terminal values of the viscous Burgers equation with periodic boundary conditions.
We introduce below the viscous Burgers equation in conservative form and the corresponding operator.

Specifically,
assume \cref{algo:operator_learning}, 
let
$T = \OLBurgersT$,
$\laplaceFactor = \OLBurgersLaplaceFactor$,
for every 
$\IVfunction \in \IVSpace$
let 
$\sol_\IVfunction \colon [0, T] \to \EndSpace$ 
be a mild solution of the \PDE\
\begin{equation}
\label{OLBurgers:eq1}
\begin{split}
	\pr*{ \tfrac{\partial}{\partial t} \sol_\IVfunction } (t, x)
	=
	\laplaceFactor
	(\Delta_x \sol_\IVfunction)(t, x) 
	- 
	\tfrac{1}{2}
	\pr*{\tfrac{\partial}{\partial t} \sol_\IVfunction^2}(t, x),
\quad
	(t, x) \in [0, T] \times \Domain,
\quad
	\sol_\IVfunction(0) = \IVfunction
\end{split}
\end{equation}
with periodic boundary conditions,
assume 
$\In = \EndValues = \IVSpace$,
and assume that the operator
$\solOp \colon \In \to \EndValues$
we want to approximate is
given for all
$\IVfunction \in \In$ by
\begin{equation}
\label{OLBurgers:eq2}
\begin{split}
	\solOp(\IVfunction)
	=
	\sol_\IVfunction(T).
\end{split}
\end{equation}
Moreover, let 
$\initialRV \colon \Omega \to \In$
be a 
$\OLBurgersInitDistr$-distributed random initial value
where $\Delta_x$ is the Laplace operator on $\Ltwo{\Domain}$
with periodic boundary conditions, 
fix a space discretization $\nrspacediscr=\OLBurgersSpaceStep$,
and assume for all
$\EndVariable \in \EndValues$
that
\begin{equation}
\label{OLBurgers:eq3}
\begin{split}
		\lossmetric{\EndVariable}^2
	=
		\frac{
			\OLBurgersSpaceSize
		}{\nrspacediscr}
		\br*{
			\sum_{\fx \in \{\frac{0}{\nrspacediscr}, \frac{1}{\nrspacediscr}, \ldots, \frac{\nrspacediscr-1}{\nrspacediscr}\}}
				(\EndVariable(
					\OLBurgersSpaceSize 
					\fx
				))^2
		}
	\approx
		\int_0^{\OLBurgersSpaceSize}
			(\EndVariable(x))^2
		d x
		.
\end{split}
\end{equation}
We measure the quality of approximations $\solOpAlt \colon \In \to \EndValues$ by means of the $L^2$-error
\begin{equation}
	\label{OLBurgers:eq4}
	\begin{split}
		\pr[\big]{
			\Exp[\big]{
				\lossmetric{\solOpAlt(\initialRV) - \solOp(\initialRV)}^2
			}
		}^{1/2}
	\approx
		\pr*{
			\Exp*{
				\int_0^{\OLBurgersSpaceSize}
					\pr[\big]{
						\solOpAlt(\initialRV)(x) - \sol_\initialRV(T, x)
					}^2
				d x
			}
		}^{1/2}.
	\end{split}
\end{equation}

We test five different types of architectures for the neural operator $\neuralOp$.
Specifically, 
we test 
	\ANNs\ with \GELU\ activation function (cf.\ \cref{sect:ANN_NO}, rows 1-4 in \cref{table:OLBurgers}, and \cref{OLBurgersSamples:ann}),
	simple periodic convolutional neural operators (cf.\ \cref{sect:periodic_cnn}, rows 5-8 in \cref{table:OLBurgers}, and \cref{OLBurgersSamples:cnn_periodic}),
	simple convolutional neural operators with encoder-decoder architecture (cf.\ \cref{sect:dec_enc_cnn}, rows 9-12 in \cref{table:OLBurgers}, and \cref{OLBurgersSamples:cnn_enc_dec}),
	\FNOs\ (cf.\ \cref{sect:FNO}, rows 13-16 in \cref{table:OLBurgers}, and \cref{OLBurgersSamples:fno}),
	and
	\DeepONets\ (cf.\ \cref{sect:deepONets}, rows 17-20 in \cref{table:OLBurgers}, and  \cref{OLBurgersSamples:deeponet}).
We also evaluate three different classical approximation methods.
Specifically,
we evaluate 
	\FDMs\ with Crank-Nicolson explicit midpoint \LIRK\ discretization in time (cf.\ 
	rows 21-24 in \cref{table:OLBurgers} and \cref{OLBurgersSamples:fdm}),
	\FEMs\ with Crank-Nicolson explicit midpoint \LIRK\ discretization in time (cf.\ 
	rows 25-28 in \cref{table:OLBurgers} and \cref{OLBurgersSamples:fem}),
	and
	spectral-Galerkin methods with Crank-Nicolson explicit midpoint \LIRK\ discretization in time (cf.\ 
	rows 29-32 in \cref{table:OLBurgers} and \cref{OLBurgersSamples:spectral}).
The performance of all considered methods is summarized in \cref{table:OLBurgers} and graphically illustrated in \cref{fig:OLBurgers}.
In addition, some approximations for a randomly chosen test sample are shown in \cref{fig:OLBurgersSamples}.

\begin{table} 
\tiny
\resizebox{\textwidth}{!}{
\csvreader[
	tabular=|c|c|c|c|c|,
	separator=semicolon,
     table head=
     \hline 
     \thead{Method} &  
     \thead{Estimated \\ $L^2$-error \\ in \cref{OLBurgers:eq4}} &
	 \thead{Average evaluation time \\ for $\OLBurgersNrTestSamples$ test samples \\ over $\nrEvalRunsOL$ runs (in s)} &
     \thead{Number \\ of trainable \\ parameters} &
     \thead{Precomputation\\time (in s)} 
     \\\hline,
    late after line=\\\hline
]
{sections/OL_Phil/1_Tables/rounded_methods_data_\eqname.csv}
{
	Method=\method, 
	L2_error = \llerror, 
	nr_params = \numparams, 
	training_time = \traintime, 
	test_time = \evaltime
}
{\method& \llerror& \evaltime &\numparams&\traintime} %
}
\caption{\label{table:OLBurgers}
Comparison of the performance of different methods for the approximation of the operator in \eqref{OLBurgers:eq2} mapping initial values to terminal values of the viscous Burgers equation in \cref{OLBurgers:eq1}.}
\end{table}

\begin{figure}
\centering
\includegraphics[width=0.7\linewidth]{sections/OL_Phil/2_Plots/Burgers/error_scatter_plot_\eqname.pdf}
\caption{\label{fig:OLBurgers}
Graphical illustration of the performance of the methods in \cref{table:OLBurgers}.
}
\end{figure}

\begin{figure}
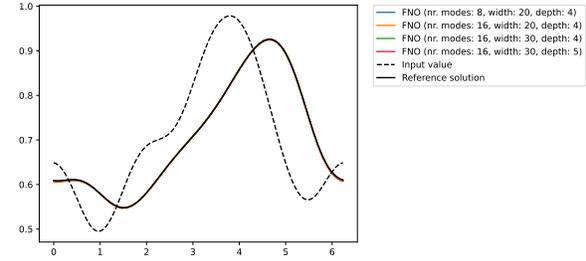
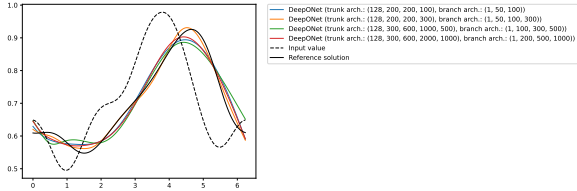
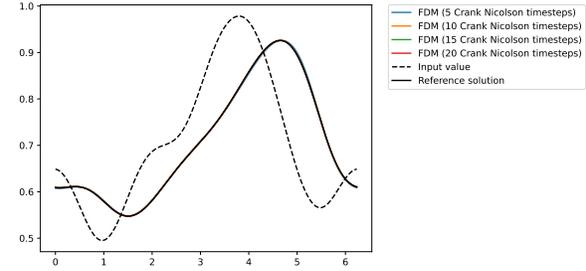
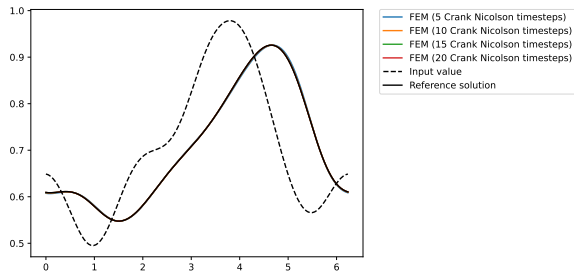
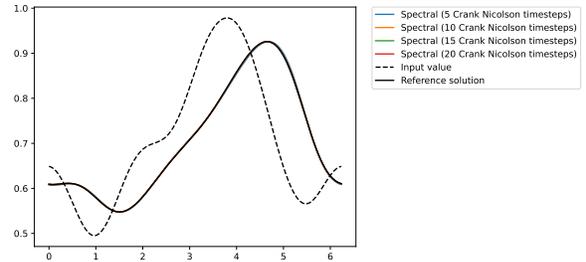

  \centering
  \begin{subfigure}{.45\textwidth}
    \centering
    \includegraphics[width=.95\linewidth]{sections/OL_Phil/2_Plots/Burgers/Sample_plots/ANN_plots_\eqname_0.pdf}
    \caption{\ANN\ plots}
    \label{OLBurgersSamples:ann}
    \vspace{1cm}
  \end{subfigure}%
  \begin{subfigure}{.45\textwidth}
    \centering
    \includegraphics[width=.95\linewidth]{sections/OL_Phil/2_Plots/Burgers/Sample_plots/Periodic-CNN_plots_\eqname_0.pdf}
    \caption{\CNN\ periodic plots}
    \label{OLBurgersSamples:cnn_periodic}
    \vspace{1cm}
  \end{subfigure}%
  
    \begin{subfigure}{.45\textwidth}
      \centering
      \includegraphics[width=.95\linewidth]{sections/OL_Phil/2_Plots/Burgers/Sample_plots/Enc.-Dec.-CNN_plots_\eqname_0.pdf}
      \caption{\CNN\ encoder-decoder plots}
      \label{OLBurgersSamples:cnn_enc_dec}
    	\vspace{1cm}
    \end{subfigure}
  \begin{subfigure}{.45\textwidth}
    \centering
    \includegraphics[width=.95\linewidth]{sections/OL_Phil/2_Plots/Burgers/Sample_plots/FNO_plots_\eqname_0.pdf}
    \caption{\FNO\ plots}
    \label{OLBurgersSamples:fdm}
    \vspace{1cm}
  \end{subfigure}%
  
    \begin{subfigure}{.45\textwidth}
      \centering
      \includegraphics[width=.95\linewidth]{sections/OL_Phil/2_Plots/Burgers/Sample_plots/DeepONet_plots_\eqname_0.pdf}
      \caption{\DeepONet\ plots}
      \label{OLBurgersSamples:deeponet}
      \vspace{1cm}
    \end{subfigure}
  \begin{subfigure}{.45\textwidth}
    \centering
    \includegraphics[width=.95\linewidth]{sections/OL_Phil/2_Plots/Burgers/Sample_plots/FDM_plots_\eqname_0.pdf}
    \caption{\FDM\ plots}
    \label{OLBurgersSamples:fem}
    \vspace{1cm}
  \end{subfigure}

  \begin{subfigure}{.45\textwidth}
    \centering
    \includegraphics[width=.95\linewidth]{sections/OL_Phil/2_Plots/Burgers/Sample_plots/FEM_plots_\eqname_0.pdf}
    \caption{\FEM\ plots}
    \label{OLBurgersSamples:fno}
    \vspace{1cm}
  \end{subfigure}%
  \begin{subfigure}{.45\textwidth}
    \centering
    \includegraphics[width=.95\linewidth]{sections/OL_Phil/2_Plots/Burgers/Sample_plots/Spectral_plots_\eqname_0.pdf}
    \caption{Spectral plots}
    \label{OLBurgersSamples:spectral}
    \vspace{1cm}
  \end{subfigure}
\caption{Example approximation plots for a randomly chosen initial value for the viscous Burgers equation in \cref{OLBurgers:eq1}.}
\label{fig:OLBurgersSamples}
\end{figure}

\endgroup

\subsubsection{Allen-Cahn equation}
\label{sect:OL_simul_AC}

\begingroup
\providecommandordefault{\eqnameone}{Semilinear_heat_1-dimensional_T_3.0_space_size_1.0_laplace_factor_0.002_nonlin_AllenCahn_var_5000_decay_rate_2_offset_70.71067811865476_inner_decay_1.0_start_var_0.2}
\providecommandordefault{\eqnametwo}{Semilinear_heat_2-dimensional_T_3.0_space_size_1.0_laplace_factor_0.002_nonlin_AllenCahn_var_5000_decay_rate_2_offset_70.71067811865476_inner_decay_1.0_start_var_0.2}
\providecommandordefault{\eqnamethree}{Semilinear_heat_3-dimensional_T_3.0_space_size_1.0_laplace_factor_0.002_nonlin_AllenCahn_var_5000_decay_rate_2_offset_70.71067811865476_inner_decay_1.0_start_var_0.2}

\providecommandordefault{\Domain}{(0,\OLACtwodSpaceSize)^d}
\providecommandordefault{\IVSpace}{H^2_{\text{per}}(\Domain; \R)}
\providecommandordefault{\EndSpace}{H^2_{\text{per}}(\Domain; \R)}

In this section we present numerical simulations for the approximation of operators mapping initial values to terminal values of one and two-dimensional Allen-Cahn equations.
We introduce below the considered Allen-Cahn equations and the corresponding operators.

Specifically,
assume \cref{algo:operator_learning}, 
let
$d \in \{1, 2\}$,
$T = \OLACtwodT$,
$\laplaceFactor = \OLACtwodLaplaceFactor$,
for every 
$\IVfunction \in \IVSpace$
let 
$\sol_\IVfunction \colon [0,T] \to \EndSpace$
be a mild solution of the \PDE\
\begin{equation}
	\label{OLAC:eq1}
	\begin{split} 
		\pr*{ \tfrac{\partial}{\partial t} \sol_\IVfunction } (t, x)
		=
		\laplaceFactor
		(\Delta_x \sol_\IVfunction)(t, x) -  (\sol_\IVfunction (t, x))^3 + (\sol_\IVfunction (t, x)),
	\quad 
		(t, x) \in [0,T] \times \Domain,
	\quad
		\sol_\IVfunction(0) = \IVfunction
	\end{split}
\end{equation} 
with periodic boundary conditions,
assume 
$\initialValues = \EndValues = \IVSpace$,
and 
assume 
that the operator
$\solOp \colon \initialValues \to \EndValues$
we want to approximate is
given for all
$\IVfunction \in \initialValues$ by
\begin{equation}
\label{OLAC:eq2}
\begin{split}
	\solOp(\IVfunction)
	=
	\sol_\IVfunction(T).
\end{split}
\end{equation}
Moreover, let
$\initialRV \colon \Omega \to \initialValues$
be a 
$\OLACtwodInitDistr$-distributed random initial value
where $\Delta_x$ is the Laplace operator on $\Ltwo{\Domain}$
with periodic boundary conditions,
fix a space discretization $\nrspacediscr \in \N$,
and assume for all
$\EndVariable \in \EndValues$
that
\begin{equation}
\label{OLAC:eq3}
\begin{split}
		\lossmetric{\EndVariable}^2
	=
		\frac{1}{\nrspacediscr^d}
		\br*{
			\sum_{\fx \in \{\frac{0}{\nrspacediscr}, \frac{1}{\nrspacediscr}, \ldots, \frac{\nrspacediscr-1}{\nrspacediscr}\}^d}
				(\EndVariable(\fx))^2
		}
	\approx
		\int_{\Domain} 
			(\EndVariable(x))^2
		d x
		.
\end{split}
\end{equation}
We measure the quality of approximations $\solOpAlt \colon \In \to \EndValues$ by means of the $L^2$-error
\begin{equation}
	\label{OLAC:eq4}
	\begin{split}
		\pr[\big]{
			\Exp[\big]{
				\lossmetric{\solOpAlt(\initialRV) - \solOp(\initialRV)}^2
			}
		}^{1/2}
	\approx
		\pr*{
			\Exp*{
				\int_{\Domain}
					\pr[\big]{
						\solOpAlt(\initialRV)(x) - 	\sol_\initialRV(T, x)
					}^2
				d x
			}
		}^{1/2}.
	\end{split}
\end{equation}

We test five different types of architectures for the neural operator $\neuralOp$.
Specifically, 
we test 
	\ANNs\ with \GELU\ activation function (cf.\ \cref{sect:ANN_NO}, rows 1-3 in \cref{table:AC1}, and rows 1-3 in \cref{table:AC2}),
	simple periodic convolutional neural operators (cf.\ \cref{sect:periodic_cnn}, rows 4-6 in \cref{table:AC1}, and rows 4-6 in \cref{table:AC2}),
	simple convolutional neural operators with encoder-decoder architecture (cf.\ \cref{sect:dec_enc_cnn}, rows 7-9 in \cref{table:AC1}, and rows 7-9 in \cref{table:AC2}),
	\FNOs\ (cf.\ \cref{sect:FNO}, rows 10-12 in \cref{table:AC1}, and rows 10-12 in \cref{table:AC2}),
	and
	\DeepONets\ (cf.\ \cref{sect:deepONets}, rows 13-15 in \cref{table:AC1}, and rows 13-15 in \cref{table:AC2}).
We also evaluate spectral-Galerkin methods with Crank-Nicolson explicit midpoint \LIRK\ discretization in time (cf.\ 
rows 16-18 in \cref{table:AC1} and rows 16-18 in \cref{table:AC2}).
The performance of all considered methods in the case $d =1$ is summarized in \cref{table:AC1} 
and graphically illustrated in \cref{fig:AC1}.
The performance of all considered methods in the case $d =2$ is summarized in \cref{table:AC2}
and graphically illustrated in \cref{fig:AC2}.
In addition, some approximations for a randomly chosen test samples are shown in \cref{fig:AC1_plots,fig:AC2_plots}.

\begin{table} 
\tiny
\resizebox{\textwidth}{!}{
\csvreader[
	tabular=|c|c|c|c|c|c|,
	separator=semicolon,
		table head=
		\hline 
		\thead{Method} &  
		\thead{Estimated \\ $L^2$-error} &
		\thead{Time for 1024\\evaluations (in s)} &
		\thead{Number \\ of trainable \\ parameters} &
		\thead{Precomputation\\time (in s)} &
		\thead{Number of \\ train steps}
		\\\hline,
	late after line=\\\hline
]
{sections/OL_Phil/1_Tables/rounded_methods_data_\eqnameone.csv}
{
	Method=\method, 
	L2_error = \llerror, 
	nr_params = \numparams, 
	training_time = \traintime, 
	test_time = \evaltime,
	done_trainsteps = \trainsteps
}
{\method& \llerror& \evaltime &\numparams&\traintime &\trainsteps}%
}
\caption{\label{table:AC1}
Comparison of the performance of different methods for the approximation of the operator in \eqref{OLAC:eq2} mapping initial values to terminal values of the Allen-Cahn equation in \cref{OLAC:eq1} in the case $d=1$.
}
\end{table}

\begin{figure}
\centering
\includegraphics[width=0.7\linewidth]{sections/OL_Phil/2_Plots/AC1/error_scatter_plot_\eqnameone.pdf}
\caption{
\label{fig:AC1}
Graphical illustration of the performance of the methods in \cref{table:AC1}.
}
\end{figure}

\begin{figure}
	\centering
	\begin{subfigure}{.45\textwidth}
	  \centering
	  \includegraphics[width=.95\linewidth]{sections/OL_Phil/2_Plots/AC1/Sample_plots/ANN_plots_\eqnameone_2.pdf}
	  \caption{\ANN\ plots}
	  \label{AC1_plots:ann}
	  \vspace{1cm}
	\end{subfigure}%
	\begin{subfigure}{.45\textwidth}
	  \centering
	  \includegraphics[width=.95\linewidth]{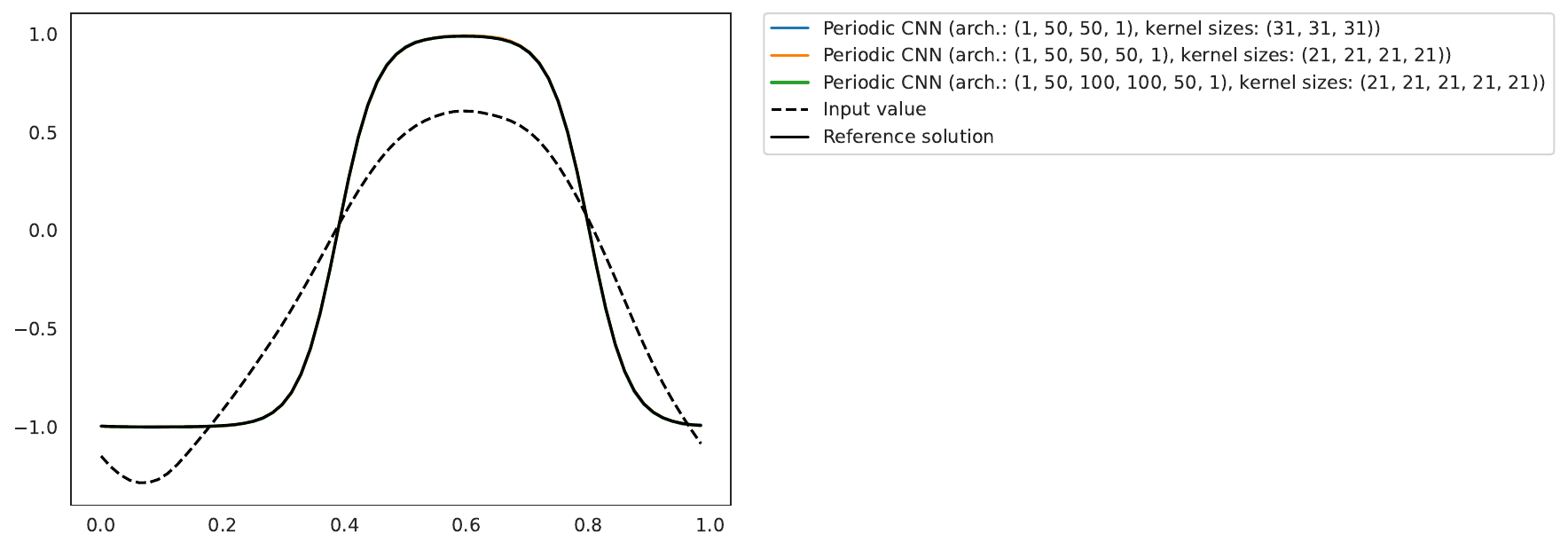}
	  \caption{\CNN\ periodic plots}
	  \label{AC1_plots:cnn_periodic}
	  \vspace{1cm}
	\end{subfigure}%
	
	  \begin{subfigure}{.45\textwidth}
		\centering
		\includegraphics[width=.95\linewidth]{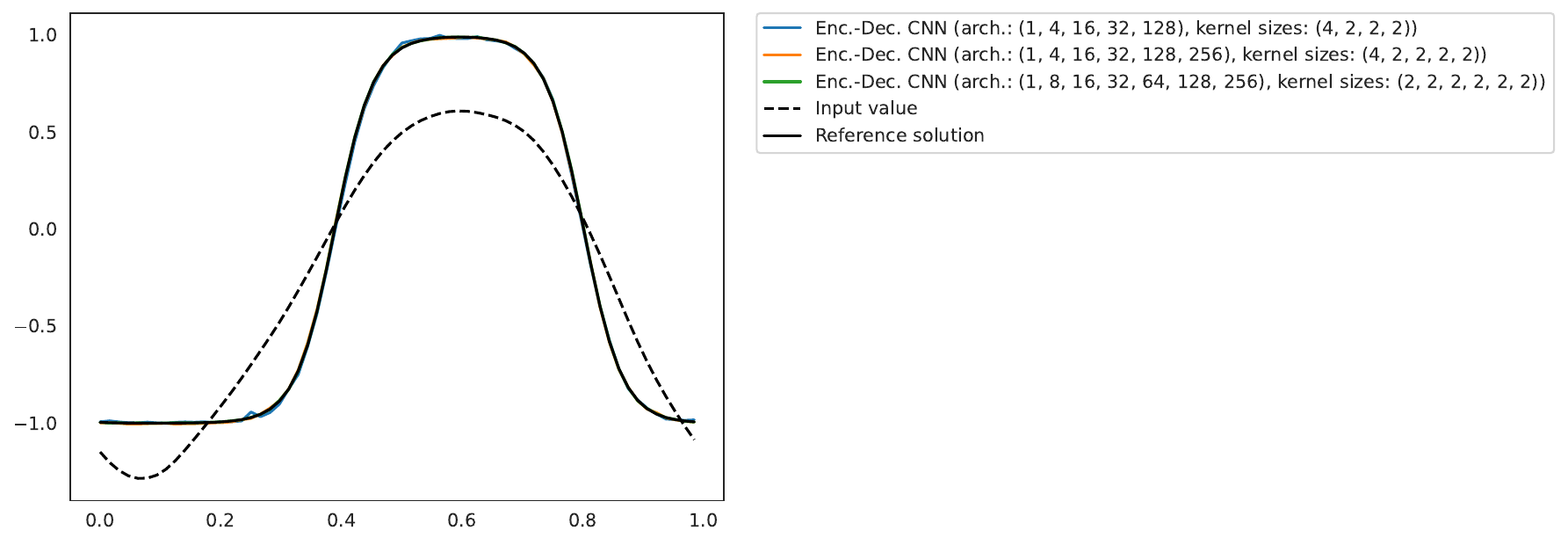}
		\caption{\CNN\ encoder-decoder plots}
		\label{AC1_plots:cnn_enc_dec}
		  \vspace{1cm}
	  \end{subfigure}
	\begin{subfigure}{.45\textwidth}
	  \centering
	  \includegraphics[width=.95\linewidth]{sections/OL_Phil/2_Plots/AC1/Sample_plots/FNO_plots_\eqnameone_2.pdf}
	  \caption{\FNO\ plots}
	  \label{AC1_plots:fdm}
	  \vspace{1cm}
	\end{subfigure}%
	
	  \begin{subfigure}{.45\textwidth}
		\centering
		\includegraphics[width=.95\linewidth]{sections/OL_Phil/2_Plots/AC1/Sample_plots/DeepONet_plots_\eqnameone_2.pdf}
		\caption{\DeepONet\ plots}
		\label{AC1_plots:deeponet}
		\vspace{1cm}
	  \end{subfigure}
	\begin{subfigure}{.45\textwidth}
	  \centering
	  \includegraphics[width=.95\linewidth]{sections/OL_Phil/2_Plots/AC1/Sample_plots/Spectral_plots_\eqnameone_2.pdf}
	  \caption{\FDM\ plots}
	  \label{AC1_plots:fem}
	  \vspace{1cm}
	\end{subfigure}
  \caption{
  Example approximation plots for a randomly chosen initial value for the Allen-Cahn equation in \cref{OLAC:eq1} in the case $d = 1$.
  }
  \label{fig:AC1_plots}
  \end{figure}

\begin{table} 
\tiny
\resizebox{\textwidth}{!}{
\csvreader[
	tabular=|c|c|c|c|c|c|,
	separator=semicolon,
     table head=
     \hline 
     \thead{Method} &  
     \thead{Estimated \\ $L^2$-error} &
      \thead{Time for 1024\\evaluations (in s)} &
     \thead{Number \\ of trainable \\ parameters} &
     \thead{Precomputation\\time (in s)} &
     \thead{Number of \\ train steps}
     \\\hline,
    late after line=\\\hline
]
{sections/OL_Phil/1_Tables/rounded_methods_data_\eqnametwo.csv}
{
	Method=\method, 
	L2_error = \llerror, 
	nr_params = \numparams, 
	training_time = \traintime, 
	test_time = \evaltime,
	done_trainsteps = \trainsteps
}
{\method& \llerror& \evaltime &\numparams&\traintime &\trainsteps}%
}
\caption{\label{table:AC2}
Comparison of the performance of different methods for the approximation of the operator in \eqref{OLAC:eq2} mapping initial values to terminal values of the Allen-Cahn equation in \cref{OLAC:eq1} in the case $d = 2$.}
\end{table}

\begin{figure}
\centering
\includegraphics[width=0.7\linewidth]{sections/OL_Phil/2_Plots/AC2/error_scatter_plot_\eqnametwo.pdf}
\caption{
\label{fig:AC2}
Graphical illustration of the performance of the methods in \cref{table:AC2}.
}
\end{figure}

\begin{figure} 
\centering
	\includegraphics[width=0.63\linewidth]{sections/OL_Phil/2_Plots/AC2/Sample_plots/all_plots_\eqnametwo_1.pdf}
\caption{
\label{fig:AC2_plots}
Example approximation plots for a randomly chosen initial value for the Allen-Cahn equation in \cref{OLAC:eq1} in the case $d = 2$.
}
\end{figure}

\endgroup

\subsubsection{Reaction-diffusion equation}
\label{sect:OL_simul_ReactionDiffusion}

\begingroup

\providecommandordefault{\eqname}{ReactionDiffusion_T1.0_S2.0_nu0.05_k2.0_nonlinAllenCahn_var10000_decay2_offset100.0_innerdecay1.0}

\providecommandordefault{\reactionRate}{k}

\providecommandordefault{\Domain}{(0, \OLReactDiffSpaceSize)}
\providecommandordefault{\IVSpace}{H^2_\text{per}(\Domain; \R)}
\providecommandordefault{\EndSpace}{H^2_\text{per}(\Domain; \R)}

In this section we present numerical simulations for the approximation of an operator  mapping source terms to terminal values of a reaction-diffusion equation.
The considered operator is inspired by the reaction-diffusion equation in \cite[Section 4.3]{Lu2021}.
We introduce below the considered reaction-diffusion equation and the corresponding operator.

Specifically,
assume \cref{algo:operator_learning},
let
$T = \OLReactDiffT$,
$\laplaceFactor = \OLReactDiffLaplaceFactor$,
$\reactionRate = \OLReactDiffReactionRate$,
for every 
$\IVfunction \in \IVSpace$
let 
$\sol_\IVfunction \colon [0, T] \to \EndSpace$
be a mild solution of the \PDE\
\begin{equation}
	\label{OLReactionDiffusion:eq1}
	\begin{split} 
		\pr*{ \tfrac{\partial}{\partial t} \sol_\IVfunction } (t, x)
		=
		\laplaceFactor
		(\Delta_x \sol_\IVfunction)(t, x) 
		+
		\reactionRate 
		\pr*{
			\sol_\IVfunction(t, x) - (\sol_\IVfunction(t, x))^3
		}
		+
		\IVfunction(x),
	\end{split}
\end{equation}
\begin{equation}
\label{OLReactionDiffusion:eq1.1}
	\begin{split}
		(t, x) \in [0, T] \times \Domain,
		\qquad
		\sol_\IVfunction(0) = 0
	\end{split}
\end{equation} 
with periodic boundary conditions,
assume
$\initialValues = \EndValues = \IVSpace$,
and assume that the operator
$\solOp \colon \initialValues \to \EndValues$
we want to approximate is
given for all
$\IVfunction \in \initialValues$ by
\begin{equation}
\label{OLReactionDiffusion:eq2}
\begin{split}
	\solOp(\IVfunction)
	=
	\sol_\IVfunction(T).
\end{split}
\end{equation}
Moreover, let
$\initialRV \colon \Omega \to \initialValues$
be a
$\OLReactDiffInitDistr$-distributed random source term
where $\Delta_x$ is the Laplace operator on 
$\Ltwo{\Domain}$
with periodic boundary conditions,
fix a space discretization $\nrspacediscr=128$,
and assume for all
$\EndVariable \in \EndValues$
that
\begin{equation}
\label{OLReactionDiffusion:eq3}
\begin{split}
		\lossmetric{\EndVariable}^2
	=
		\frac{\OLReactDiffSpaceSize}{\nrspacediscr}
		\br*{
			\sum_{\fx \in \{\frac{0}{\nrspacediscr}, \frac{1}{\nrspacediscr}, \ldots, \frac{\nrspacediscr-1}{\nrspacediscr}\}}
				(\EndVariable(\OLReactDiffSpaceSize\fx))^2
		}
	\approx
		\int_0^{\OLReactDiffSpaceSize}
			(\EndVariable(x))^2
		d x
		.
\end{split}
\end{equation}
We measure the quality of approximations $\solOpAlt \colon \In \to \EndValues$ by means of the $L^2$-error
\begin{equation}
	\label{OLReactionDiffusion:eq4}
	\begin{split}
		\pr[\big]{
			\Exp[\big]{
				\lossmetric{\solOpAlt(\initialRV) - \solOp(\initialRV)}^2
			}
		}^{1/2}
	\approx
		\pr*{
			\Exp*{
				\int_0^{\OLReactDiffSpaceSize}
					\pr[\big]{
						\solOpAlt(\initialRV)(x) - \sol_\initialRV(T, x)
					}^2
				d x
			}
		}^{1/2}.
	\end{split}
\end{equation}

We test five different types of architectures for the neural operator $\neuralOp$.
Specifically, 
we test 
	\ANNs\ with \GELU\ activation function (cf.\ \cref{sect:ANN_NO}, rows 1-3 in \cref{table:OLReactDiff}, and \cref{OLReactDiffSamples:ann}),
	simple periodic convolutional neural operators (cf.\ \cref{sect:periodic_cnn}, rows 4-6 in \cref{table:OLReactDiff}, and \cref{OLReactDiffSamples:cnn_periodic}),
	simple convolutional neural operators with encoder-decoder architecture (cf.\ \cref{sect:dec_enc_cnn}, rows 7-9 in \cref{table:OLReactDiff}, and \cref{OLReactDiffSamples:cnn_enc_dec}),
	\FNOs\ (cf.\ \cref{sect:FNO}, rows 10-12 in \cref{table:OLReactDiff}, and \cref{OLReactDiffSamples:fno}),
	and
	\DeepONets\ (cf.\ \cref{sect:deepONets}, rows 13-15 in \cref{table:OLReactDiff}, and  \cref{OLReactDiffSamples:deeponet}).
We also evaluate 
	\FDMs\ with Crank-Nicolson explicit midpoint \LIRK\ discretization in time (cf.\ 
	rows 16-18 in \cref{table:OLReactDiff} and \cref{OLReactDiffSamples:fdm}).
The performance of all considered methods is summarized in \cref{table:OLReactDiff} and graphically illustrated in \cref{fig:OLReactDiff}.
In addition, some approximations for a randomly chosen test sample are shown in \cref{fig:OLReactDiffSamples}.

\begin{table} 
\tiny
\resizebox{\textwidth}{!}{
\csvreader[
	tabular=|c|c|c|c|c|,
	separator=semicolon,
     table head=
     \hline 
     \thead{Method} &  
     \thead{Estimated \\ $L^2$-error \\  in \cref{OLReactionDiffusion:eq4}} &
	 \thead{Average evaluation time \\ for $\OLReactDiffNrTestSamples$ test samples \\ over $\nrEvalRunsOL$ runs (in s)} &
     \thead{Number \\ of trainable \\ parameters} &
     \thead{Precomputation\\time (in s)} 
     \\\hline,
    late after line=\\\hline
]
{sections/OL_Phil/1_Tables/rounded_methods_data_\eqname.csv}
{
	Method=\method, 
	L2_error = \llerror, 
	nr_params = \numparams, 
	training_time = \traintime, 
	test_time = \evaltime
}
{\method& \llerror& \evaltime &\numparams&\traintime} %
}
\caption{\label{table:OLReactDiff}
Comparison of the performance of different methods for the approximation of the operator in \eqref{OLReactionDiffusion:eq2} mapping source terms to terminal values of the reaction-diffusion equation in \cref{OLReactionDiffusion:eq1}.
}
\end{table}

\begin{figure}
\centering
\includegraphics[width=0.7\linewidth]{sections/OL_Phil/2_Plots/Reaction_Diffusion/error_scatter_plot_\eqname.pdf}
\caption{\label{fig:OLReactDiff}
Graphical illustration of the performance of the methods in \cref{table:OLReactDiff}.
}
\end{figure}

\begin{figure}
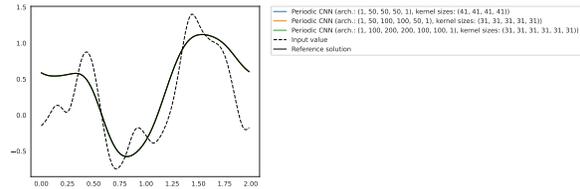
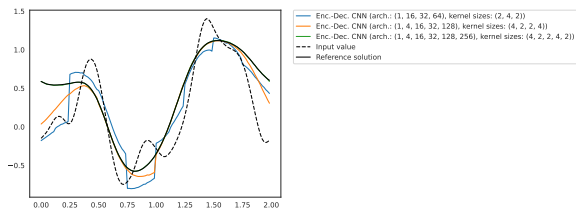
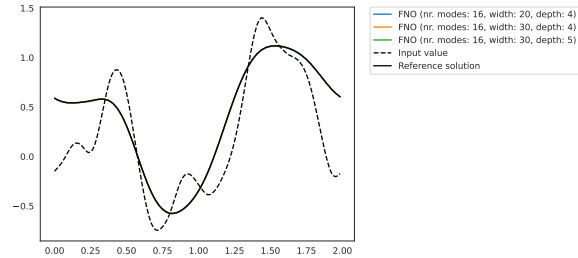
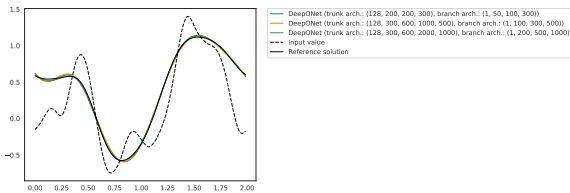
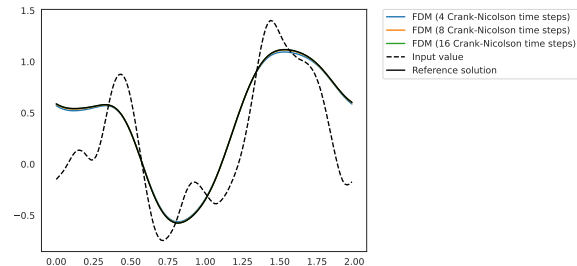

  \centering
  \begin{subfigure}{.45\textwidth}
    \centering
    \includegraphics[width=.95\linewidth]{sections/OL_Phil/2_Plots/Reaction_Diffusion/Sample_plots/ANN_plots_\eqname_1.pdf}
    \caption{\ANN\ plots}
    \label{OLReactDiffSamples:ann}
    \vspace{1cm}
  \end{subfigure}%
  \begin{subfigure}{.45\textwidth}
    \centering
    \includegraphics[width=.95\linewidth]{sections/OL_Phil/2_Plots/Reaction_Diffusion/Sample_plots/Periodic-CNN_plots_\eqname_1.pdf}
    \caption{\CNN\ periodic plots}
    \label{OLReactDiffSamples:cnn_periodic}
    \vspace{1cm}
  \end{subfigure}%
  
    \begin{subfigure}{.45\textwidth}
      \centering
      \includegraphics[width=.95\linewidth]{sections/OL_Phil/2_Plots/Reaction_Diffusion/Sample_plots/Enc.-Dec.-CNN_plots_\eqname_1.pdf}
      \caption{\CNN\ encoder-decoder plots}
      \label{OLReactDiffSamples:cnn_enc_dec}
    	\vspace{1cm}
    \end{subfigure}
  \begin{subfigure}{.45\textwidth}
    \centering
    \includegraphics[width=.95\linewidth]{sections/OL_Phil/2_Plots/Reaction_Diffusion/Sample_plots/FNO_plots_\eqname_1.pdf}
    \caption{\FNO\ plots}
    \label{OLReactDiffSamples:fno}
    \vspace{1cm}
  \end{subfigure}%
  
    \begin{subfigure}{.45\textwidth}
      \centering
      \includegraphics[width=.95\linewidth]{sections/OL_Phil/2_Plots/Reaction_Diffusion/Sample_plots/DeepONet_plots_\eqname_1.pdf}
      \caption{\DeepONet\ plots}
      \label{OLReactDiffSamples:deeponet}
      \vspace{1cm}
    \end{subfigure}
  \begin{subfigure}{.45\textwidth}
    \centering
    \includegraphics[width=.95\linewidth]{sections/OL_Phil/2_Plots/Reaction_Diffusion/Sample_plots/FDM_plots_\eqname_1.pdf}
    \caption{\FDM\ plots}
    \label{OLReactDiffSamples:fdm}
    \vspace{1cm}
  \end{subfigure}

\caption{
Example approximation plots for a randomly chosen initial value for the reaction-diffusion equation in \cref{OLReactionDiffusion:eq1}.
}
\label{fig:OLReactDiffSamples}
\end{figure}

\endgroup

\begin{sidewaystable}[h!]
	\centering
		\begin{adjustbox}{max width=1.1\paperwidth,center}
			\begin{tabular}{lcccc}
			\toprule
			&Burgers & Allen-Cahn ($d=1$) & Allen-Cahn ($d=2$) & Reaction-diffusion \\
			& (cf.\ \cref{sect:OL_simul_Burgers}) & (cf.\ \cref{sect:OL_simul_AC}) & (cf.\ \cref{sect:OL_simul_AC}) & (cf.\ \cref{sect:OL_simul_ReactionDiffusion}) \\
			\midrule
			\# space steps per dimension ($N$) 
				& $\OLBurgersSpaceStep$ & $\OLAConedSpaceStep$ & $\OLACtwodSpaceStep$ & $\OLReactDiffSpaceStep$ \\
			\midrule 
			Training &&&& \\
			\hspace{0.45cm} Batch size 
				& $\OLBurgersBatchSize$ & $\OLAConedBatchSize$ & $\OLACtwodBatchSize$ & $\OLReactDiffBatchSize$ \\
			\hspace{0.45cm} \# tr.\ steps between val.\ error eval.\
				& $\OLBurgersEvalSteps$ & $\OLAConedEvalSteps$ & $\OLACtwodEvalSteps$ & $\OLReactDiffEvalSteps$ \\
			\hspace{0.45cm} Improvement tolerance 
				& $\OLBurgersImprTol$ & $\OLAConedImprTol$ & $\OLACtwodImprTol$ & $\OLReactDiffImprTol$ \\
			\hspace{0.45cm} \# of tr.\ runs per model 
				& $\OLBurgersNrOLRuns$ & $\OLAConedNrOLRuns$ & $\OLACtwodNrOLRuns$ & $\OLReactDiffNrOLRuns$ \\

			\midrule
			Algorithm for reference sol.
				& \OLBurgersRefAlgOne  & \OLAConedRefAlgOne  & \OLACtwodRefAlgOne  & \OLReactDiffRefAlgOne  \\
				& \OLBurgersRefAlgTwo  & \OLAConedRefAlgTwo  & \OLACtwodRefAlgTwo  & \OLReactDiffRefAlgTwo  \\
			\midrule
			Training set  &&&& \\
			\hspace{0.45cm} \# samples 
				& $\OLBurgersNrTrainSamples$ & $\OLAConedNrTrainSamples$ & $\OLACtwodNrTrainSamples$ & $\OLReactDiffNrTrainSamples$ \\
			\hspace{0.45cm} \# space steps  per dimension
				& $\OLBurgersNrTrainSpaceDiscr$ & $\OLAConedNrTrainSpaceDiscr$ & $\OLACtwodNrTrainSpaceDiscr$ & $\OLReactDiffNrTrainSpaceDiscr$ \\
			\hspace{0.45cm} \# time steps 
				& $\OLBurgersNrTrainTimeSteps$ & $\OLAConedNrTrainTimeSteps$ & $\OLACtwodNrTrainTimeSteps$ & $\OLReactDiffNrTrainTimeSteps$ \\
			\midrule
			Validation set  &&&& \\
			\hspace{0.45cm} \# samples 
				& $\OLBurgersNrValidSamples$ & $\OLAConedNrValidSamples$ & $\OLACtwodNrValidSamples$ & $\OLReactDiffNrValidSamples$ \\
			\hspace{0.45cm} \# space steps  per dimension
				& $\OLBurgersNrValidSpaceDiscr$ & $\OLAConedNrValidSpaceDiscr$ & $\OLACtwodNrValidSpaceDiscr$ & $\OLReactDiffNrValidSpaceDiscr$ \\
			\hspace{0.45cm} \# time steps 
				& $\OLBurgersNrValidTimeSteps$ & $\OLAConedNrValidTimeSteps$ & $\OLACtwodNrValidTimeSteps$ & $\OLReactDiffNrValidTimeSteps$ \\
			\midrule
			Test set  &&&& \\
			\hspace{0.45cm} \# samples 
				& $\OLBurgersNrTestSamples$ & $\OLAConedNrTestSamples$ & $\OLACtwodNrTestSamples$ & $\OLReactDiffNrTestSamples$ \\
			\hspace{0.45cm} \# space steps per dimension
				& $\OLBurgersNrTestSpaceDiscr$ & $\OLAConedNrTestSpaceDiscr$ & $\OLACtwodNrTestSpaceDiscr$ & $\OLReactDiffNrTestSpaceDiscr$ \\
			\hspace{0.45cm} \# time steps
				& $\OLBurgersNrTestTimeSteps$ & $\OLAConedNrTestTimeSteps$ & $\OLACtwodNrTestTimeSteps$ & $\OLReactDiffNrTestTimeSteps$ \\
			\bottomrule
			\end{tabular}
		\end{adjustbox}
		\caption{\label{table:training_hyperparams_OL}
			Hyperparameters for the training of the models in the numerical simulations in \cref{sect:OL_simul}.
		}
\end{sidewaystable}

\clearpage

\section*{Acknowledgements}

This project has been partially funded by the National Science Foundation of China (NSFC) under grant number 12250610192.
This project has also been partially funded by the Deutsche Forschungsgemeinschaft (DFG, German Research Foundation) under Germany's Excellence Strategy EXC 2044-390685587, Mathematics M\"unster: Dynamics-Geometry-Structure.
This project has also been partially funded by the Deutsche Forschungsgemeinschaft (DFG, German Research Foundation) in the frame of the priority programme SPP 2298 '{Theoretical Foundations of Deep Learning}' – Project no.\ 464123384.
This project has also been partially funded by the Fundamental Research Funds for the Central Universities in China, No.\ 30924010943.
We gratefully acknowledge Thomas Kruse for several useful comments regarding the analysis of solutions of backward stochastic differential equations.

\bibliographystyle{acm}	
\bibliography{refnew1}

\end{document}